\documentclass[11pt,final]{article}
\usepackage{articlemacro}
\usepackage{tabularx}
\usepackage{etoc}

\definecolor{mycolor}{HTML}{F5BA73}

\usepackage[normalem]{ulem}

\usepackage[inline]{showlabels}

\setlist[enumerate,1]{label=\textup{(\arabic*)}}

\renewcommand{\H}{\mathrm{H}}
\newcommand{\RigSH}{\mathrm{RigSH}}
\newcommand{\RigH}{\mathrm{RigH}}
\newcommand{\rig}{\mathrm{rig}}

\newcommand{\Rig}{\mathrm{Rig}}

\newcommand{\Pro}{\mathrm{Pro}}
\newcommand{\Ind}{\mathrm{Ind}}
\newcommand{\Cond}{\mathrm{Cond}}

\newcommand{\Spc}{\mathrm{Spc}}
\newcommand{\PSh}{\mathrm{PSh}}
\newcommand{\Sh}{\mathrm{Sh}}
\newcommand{\Shhyp}{\Sh^{\mathrm{hyp}}} 
\renewcommand{\Shhyp}{\Sh^\wedge} 
\newcommand{\ProFin}{\mathrm{ProFin}}
\newcommand{\EDS}{\mathrm{EDS}}
\newcommand{\Cat}{\mathrm{Cat}}
\newcommand{\Nuc}{\mathrm{Nuc}}
\newcommand{\Tate}{\mathrm{Tate}}
\renewcommand{\Adic}{\mathrm{Adic}}
\newcommand{\PreAdic}{\mathrm{PreAdic}} 
\newcommand{\AnAdic}{\mathrm{AnAdic}} 
\newcommand{\Adicsm}{\mathrm{Adic}^{\mathrm{sm}}}
\newcommand{\Adicsmqcqs}{\mathrm{Adic}^{\mathrm{sm,\,qcqs}}}
\newcommand{\Set}{\mathrm{Set}}
\newcommand{\Algrm}{\mathrm{Alg}} 
\newcommand{\fib}{\mathrm{fib}} 
\newcommand{\cofib}{\mathrm{cofib}} 
\newcommand{\AnRing}{\mathrm{AnRing}} 
\newcommand{\Quiv}{\mathrm{Quiv}}
\newcommand{\const}{\mathrm{const}}
\newcommand{\can}{\mathrm{can}}
\newcommand{\LMod}{\mathrm{LMod}}

\RequirePackage{xspace}

\newcommand{\base}{B} 
\newcommand{\Gmbase}{\GG_{\mathrm{m},\base}}

\newcommand{\qcqs}{\mathrm{qcqs}}
\newcommand{\Nis}{\mathrm{Nis}}
\newcommand{\an}{\mathrm{an}}
\renewcommand{\et}{\mathrm{\acute{e}t}}
\newcommand{\discrete}{\mathrm{disc}}
\newcommand{\ft}{\mathrm{ft}}
\newcommand{\lft}{\mathrm{lft}}
\newcommand{\cond}{\mathrm{cond}}
\newcommand{\aff}{\mathrm{aff}}
\newcommand{\nuc}{\mathrm{nuc}}
\newcommand{\fin}{\mathrm{fin}}

\newcommand{\cn}{\mathrm{cn}}
\newcommand{\stable}{\mathrm{st}} 
\newcommand{\pre}{\mathrm{pre}} 

\newcommand{\BM}{\mathrm{BM}}
\newcommand{\rev}{\mathrm{rev}}

\newcommand{\K}{\mathrm{K}}
\newcommand{\KH}{\mathrm{KH}}
\newcommand{\NK}{\mathrm{NK}}
\newcommand{\Nup}{\mathrm{N}}
\newcommand{\Nuphat}{\hat{\mathrm{N}}}
\newcommand{\Kan}{\K^\mathrm{an}}
\newcommand{\Kanu}{\underline{\K}^\mathrm{an}}
\newcommand{\kan}{\mathrm{k}^\mathrm{an}}
\newcommand{\kanu}{\underline{\mathrm{k}}^\mathrm{an}}
\newcommand{\Kcont}{\K^\mathrm{cont}}
\newcommand{\Kcontu}{\underline{\K}^\mathrm{cont}}

\newcommand{\KHcontu}{\underline{\mathrm{KH}}^\mathrm{cont}}

\newcommand{\Knuc}{\K^{\nuc}}
\newcommand{\Knucu}{\underline{\K}^{\nuc}}
\newcommand{\on}{\,\mathrm{on}\,} 

\newcommand{\KGL}{\mathrm{KGL}}
\newcommand{\Bass}{\mathrm{B}}

\newcommand{\Map}{\mathrm{Map}}
\newcommand{\Mapu}{\underline{\mathrm{Ma\smash{p}}}}
\newcommand{\topspace}[1]{|#1|} 
 
\newcommand{\Arig}{\AA^{1}}

\newcommand{\Brig}{\BB^{1}}

\newcommand{\skp}[1]{\langle #1 \rangle}
\newcommand{\tea}{\skp{t}}
\newcommand{\unif}{\mathrm{unif}}
\newcommand{\Hm}{\HH_\mathrm{m}}

\renewcommand{\to}[1][]{\overset{#1}{\rightarrow}}	
	
\newcommand{\To}[1][]{\overset{#1}{\longrightarrow}}	
	
\newcommand{\inj}[1][]{\overset{#1}{\hookrightarrow}}		
\newcommand{\surj}[1][]{\overset{#1}{\twoheadrightarrow}}		

\DeclareMathOperator*{\prolim}{``lim''}

\newcommand{\varplim}{\prolim_{t\mapsto\varpi t}}
\DeclareMathOperator*{\indcolim}{``colim''}

\newcommand{\icolimvarpi}{\indcolim_{t\mapsto\varpi t}}
\renewcommand{\L}{\mathrm{L}}
\newcommand{\Lmot}{\L_{\mathrm{mot}}}

\newcommand{\PrL}{\mathrm{Pr}^{\mathrm{L}}}
\newcommand{\PrLst}{\mathrm{Pr}_{\mathrm{st}}^{\mathrm{L}}}
\newcommand{\PrLL}{\mathrm{Pr}^{\mathrm{LL}}}
\newcommand{\PrLLst}{\mathrm{Pr}_{\mathrm{st}}^{\mathrm{LL}}}
\newcommand{\Catdual}{\mathrm{Cat}^{\mathrm{dual}}_{\mathrm{st}}}
\newcommand{\PrLO}{\mathrm{Pr}^{\mathrm{L},\otimes}}
\newcommand{\CAlgPr}{\mathrm{CAlg}(\PrLO)}

\newcommand{\infcat}{category}

\newcommand{\carre}[4]{\[\begin{xy}\xymatrix{#1\ar[r]\ar[d]&#2\ar[d]\\#3\ar[r]&#4}\end{xy}\]}

\makeatletter
\newcommand{\customlabel}[2]{%
    \protected@edef\@currentlabel{#2}
    \phantomsection
    \label{#1}
}
\makeatother

\date{}
\title{Representability of continuous K-theory in rigid analytic motivic $\AA^1$-homotopy theory}
\author{Christian Dahlhausen, Can Yaylali, and Yicheng Zhou}

\begin{document}
\maketitle
\begin{abstract}
\let\thefootnote\relax
We prove that both continuous K-theory and analytic K-theory of rigid analytic spaces (à la Kerz--Saito--Tamme) satisfy descent with respect to the Nisnevich topology. Together with the fact that it is $\Arig$-invariant assuming resolutions of singularities, we deduce that it is representable in the $\mathbb{A}^{1}$-homotopy category of rigid spaces (à la Dahlhausen--Yaylali). We identify the representing object with both $\mathbb{Z}\times\BGL$ and the analytification of algebraic K-theory. As a consequence, we get a representability statement for coefficients in light condensed spectra.
Moreover, we show Weibel vanishing and that continuous K-theory is $\Arig$-invariant on local Tate pairs (without any regularity assumption).
\end{abstract}

\thispagestyle{empty}
\setcounter{tocdepth}{1}
\tableofcontents

\section{Introduction}
Algebraic K-theory is a Nisnevich sheaf on qcqs schemes and it is $\AA^1$-invariant and connective on regular schemes. Therefore, for any regular scheme $\base$, it can be seen as an object $\K$ in Morel--Voevodsky's unstable motivic homotopy category $\H(\base)$ where there is an equivalence $\K \simeq$  $\ZZ\times \BGL$ \cite[\S 4.3, Prop.~3.10]{MV1}. This description was used by Riou to establish the Adams-filtration on the underlying spectrum and upgrade the Grothendieck--Riemann--Roch isomorphism motivically. Based on this result Cisinski--Déglise proposed a category of rational motives (Beilinson-motives) admitting a full six-functor formalism and recovering rational Chow groups as extension groups, at least for smooth schemes. In the non-smooth case, similar results are achieved by working with homotopy invariant connective K-theory \cite{cisinski-descente}.

\vspace{6pt}\noindent
The goal of this article is to prove an analogous representability result for continuous K-Theory 
\[ \Kcontu \colon \AnAdic \To \Pro(\Sp^+) \]
within the rigid analytic $\Arig$-motivic homotopy category $\RigH(\base,\Pro(\Sp^+))$ which has been introduced by the two first named authors \cite{DY}.
We recall that the definition of $\Kcontu$, but with values in spectra, has been suggested by Morrow \cite{Morrow} and that its foundations have been extensively studied by Kerz--Saito--Tamme \cite{kst-i,kst-ii}. So far, it has already been established that $\Kcontu$ is a sheaf with respect to the analytic topology, see \cite[Thm.~3.5]{Morrow} and \cite[\S 7]{dahli-thesis-paper}, and that it is $\Arig$-invariant \emph{assuming locally the existence of regular models} \cite[Prop.~5.14]{kst-i}. In this paper, we establish analytic descent and Nisnevich descent which is a necessary ingredient for representability.

\begin{intro-thm}[{Theorem~\ref{Thm:Nisnevich-descent-Kcont}}]
\label{thm:Nisnevich-descent-Kcont-intro}
The presheaf
\[ \Kcontu \,\colon\, \AnAdic^{\opp} \To \Pro(\Sp^+),\quad \Spa(A,A^+) \mapsto \Kcontu(A) \]
satisfies descent for the analytic topology on affinoids, hence it extends uniquely to a sheaf for the analytic topology, which we still denote by $\Kcontu$. Furthermore, it satisfies Nisnevich descent.
\end{intro-thm}
As a consequence, we can show that $\KHcontu\coloneqq \L_{\AA^1}\Kcontu$ is also a Nisnevich sheaf on quasi-separated spaces of finite dimension (Lemma~\ref{lem:Nisnevich-descent-KHcont}).
Moreover, $\KHcontu(A) \simeq \Kanu(A)$ for any Tate ring $A$ of topologically finite type over a nonarchimedean field $k$, or a Tate ring that admits a noetherian and finite-dimensional ring of definition (Proposition~\ref{prop:A1-localisation}).
Moreover, we generalise Weibel vanishing (Proposition~\ref{prop:weibel-vanising-Kcont}) 
beyond the case over discretely valued fields shown by Kerz \cite[Thm.~12]{kerz-icm2018} and the first named author \cite[Thm.~A]{dahli-thesis-paper}.

Consequently, under an appropriate assumption \ref{spadesuit-condition-analytic} about the existence of regular models locally in the analytic topology, we establish the following result. Note that the analogous assumption \ref{spadesuit-condition-etale-k} for the étale topology is void (Remark~\ref{rem:key-assumption}) so that an étale version of this holds unconditionally (Remark~\ref{rem:representability-etale-case}).

\begin{intro-thm}[Representability Theorem~\ref{thm:representability}]
\label{thm:representability-intro}
Let $\base = \Spa(R, R^+)$ be an adic space such that $R$ admits a ring of definition which is noetherian and of finite dimension.
We assume \ref{spadesuit-condition-analytic}.
Then there is a canonical equivalence
\[\Omega^\infty\Kanu_{\geq 0} \simeq \Lmot(\ZZ\times\BGL) \]
in the category $\RigH(\base,\Pro(\Spc^+))$ and similarly in $\RigH(\base, \Cond^\omega(\Spc))$.
In particular, for every $X \in \Adicsm_\base$ there is a functorial equivalence
\[\Omega^\infty\gamma^\omega\Kanu_{\geq 0}(X) \simeq \Homline_{\Cond^\omega(\Spc)}(\Lmot X,\Lmot(\ZZ\times\BGL)) \]
in the category $\Cond^\omega(\Spc)$ of light condensed spaces; here $\gamma^\omega$ is defined below and the right-hand side denotes the enriched Hom-space as a left $\Cond^\omega(\Spc)$-module, see Appendix~\ref{app.yoneda}.
As a consequence of the Bass Fundamental Theorem (Corollary~\ref{bass-fundamental-theorem-for-Kan--cor}), we obtain a commutative algebra object $\KGL^\an\in \RigSH(\base)$ representing analytic K-theory (Corollary~\ref{cor-K-P1-spectrum}).
\end{intro-thm}

Let us comment on the role of condensed mathematics in this paper. First, passing from pro-coefficients to condensed coefficients enables us to deduce the representability formula in Theorem~\ref{thm:representability-intro}, because with coefficients in $\Cond^\omega(\Spc)$ the category $\RigH(\base,\Cond^\omega(\Spc))$ admits an enrichment.
More importantly, there exists a limit-preserving comparison functor
\[ \gamma^\omega \colon \Pro^\omega(\Sp^+)\to \Cond^{\omega}(\Sp) \]
is conservative on eventually constant objects (Theorem~\ref{thm:comparison-conservative}).
Due to results of Efimov and Andreychev, for a Tate ring $A$ the underlying spectrum $\Kcont(A)$ of the pro-spectrum $\Kcontu(A)$ is equivalent to $\K(\Nuc(A))$, i.e.\@ the K-theory of the dualisable category $\Nuc(A)$ of nuclear modules on the analytic ring (à la Clausen--Scholze) associated with $A$. Using that the category $\AnAdic$ is tensored over the category of profinite sets (Lemma~\ref{lem:XotimesSsheafy}), we lift the functor $\K(\Nuc(-))$ to a functor $\Knucu$ with values in condensed spectra. Subsequently, we prove an enriched equivalence
\[ \Knucu \simeq \gamma_\kappa\Kcontu \]
where $\gamma_\kappa \colon \Pro^\omega(\Sp^+) \to \Cond_\kappa(\Sp)$ is the comparison functor (Proposition~\ref{prop:Kcond=Kcont-enriched}). Using Andreychev's result that $\K(\Nuc(-))$ satisfies Nisnevich descent \cite[Satz~5.12]{andreychev-thesis}, we deduce Nisnevich descent for its condensed enrichment $\Knucu$ and Theorem~\ref{thm:Nisnevich-descent-Kcont-intro}.

\vspace{6pt}\noindent
One drawback of Theorem~\ref{thm:representability-intro} is its assumption \ref{spadesuit-condition-analytic} which is only known to be true in the case where the base Tate ring $R$ admits a quasi-excellent ring of definition of characteristic zero, see Remark~\ref{rem:key-assumption}. It is still an open problem whether for a \emph{regular} adic space $X$ the canonical map
\begin{equation} \tag{$\clubsuit$}
    \Kcontu(X) \To \Kcontu(X \times \Arig) 
\end{equation}
is an equivalence. As a potential first step towards a positive answer, we prove the case for local Tate pairs(Definition~\ref{def:local-huber-pair}); observe that there is no regularity assumption on the ring $A$, which is quite unexpected.

\begin{intro-thm}[{Theorem~\ref{thm:Kcont-local-Huber-pair}}]
\label{thm:Kcont-local-Huber-pair-intro}
    Let $(A,A^+)$ be a local Tate pair with pseudo-uniformiser $\varpi$ such that $A^+$ is bounded (i.e.\@ a ring of definition). Then the canonical map
    \[ \Kcontu(A) \To \varplim \Kcontu(A\tea) \simeq \Kcontu(\Arig_A) \]
    is an equivalence of pro-spectra.
\end{intro-thm}
 
Given an adic space $X$, one can check whether the map $(\clubsuit)$ is an equivalence or not on the stalks of the sheaf $\Kcontu \colon \Open(X) \to \Pro(\Sp^+)$. Moreover, with any analytic point $x\in X$ one can associate a local Huber pair $(A_x,A_x^+)$ satisfying the prerequisites of Theorem~\ref{thm:Kcont-local-Huber-pair-intro} and such that its underlying pair of rings equals $(\Ocal_{X,x},\Ocal_{X,x}^+)$, the pair of (completed) stalks of the structure sheaves $(\Ocal_X,\Ocal_X^+)$, see Remark~\ref{rem:local-Huber-pairs}. Hence it would be sufficient to know that the stalk $(\Kcontu)_x$ is indeed equivalent to $\Kcontu(A_x)$ which may be true or may be not. Since categorical K-theory commutes with filtered colimits, for the analogous question on underlying spectra it would be enough to show that the canonical map
\[ \colim_{x\in U\subseteq X} \Nuc(U) \To \Nuc(A_x) \]
is an equivalence in $\PrL$. For the latter, a better understanding of the fully faithful functor \cite[Prop.~13.16]{analytic}
\[ \Tate \To \AnRing,\quad (A,A^+) \mapsto (A,A^+)_\blacksquare \]
shall be helpful.

\vspace{6pt}\noindent\textbf{Acknowledgements.}
We thank Peter Scholze for communicating to us the proof of Theorem~\ref{thm:comparison-conservative} and Vladimir Hinich for answering our questions about the enriched Yoneda lemma.
Furthermore, we want to thank Gregory Andreychev, Dustin Clausen, Georg Tamme, and Luca Passolunghi for helpful discussions about continuous K-theory.

This collaboration was supported by the Deutsche Forschungsgemeinschaft (DFG)\linebreak through the Collaborative Research Centre TRR 326 \textit{Geometry and Arithmetic of Uniformized Structures}, project number 444845124. Furthermore, CY was supported through the Walter Benjamin Programme (DFG, project number 524431573) during his stay at the Laboratoire de Mathématiques d'Orsay. He wants to thank Vincent Pilloni for hosting him and all the other people giving a friendly and mathematically interesting environment.
YZ was supported by Simons Collaboration on Perfection in Algebra, Geometry, and Topology, and the Fondation Mathématique Jacques Hadamard.

\vspace{6pt}\noindent\textbf{Assumptions and notations.}
Throughout, we fix some inaccessible regular cardinal $\kgbar$. By \textit{small}, we will mean \textit{$\kgbar$-small}. 
A \emph{category} always means an $\infty$-category, i.e. an $(\infty,1)$-category in the sense of Lurie \cite{HTT}, and we identify 1-categories with their image under the nerve functor.
The term \emph{essentially unique} is short hand for the expression \emph{unique up to contractible choice}.

Let us fix some notation for the categories appearing in this article.
\begin{enumerate}
    \item[$\bullet$] $\Spc$, the category of small spaces, or equivalently small $\infty$-groupoids or small $\infty$-sets, or small anima (depending on your taste).
    \item[$\bullet$] $\Sp$, the category of spectra.
    \item[$\bullet$] $\PrL$, the category of presentable categories with colimit preserving functors.
    \item[$\bullet$] $\PrLO$, the category $\PrL$ endowed with the Lurie tensor product \cite[\S 4.8.1]{HA}.
    \item[$\bullet$] $\Tate$, the category of affinoid Tate rings $(A,A^+)$.
    \item[$\bullet$] $\PreAdic$, the category of pre-adic spaces (whose structure presheaf is not necessarily sheafy, as opposed to adic spaces).
    \item[$\bullet$] $\Adic$, the category of adic spaces \cite{huber1996etale}. 
    \item[$\bullet$] $\AnAdic$, the category of analytic adic spaces, cf.\@ \cite[Def.~1.1.2]{kedlaya2017sheaves}, \cite[Def.~4.3.2]{weinstein-scholze}, \cite[Def.~6.1]{huebner-adic}. These are precisely those adic spaces which are locally isomorphic to $\Spa(A,A^+)$ for an affinoid Tate ring $(A,A^+)$ \cite[Prop.~6.5]{huebner-adic}.
    \item[$\bullet$] $\Adic^\lft_\base$, the category of adic spaces locally of topologically finite type over an adic space $\base$. If $\base$ is analytic, then $\Adic^\lft_\base$ is contained in $\AnAdic$.
    \item[$\bullet$] $\Adicsm_\base$, the category of adic spaces that are smooth over an adic space $\base$.
    \item[$\bullet$] $\Cond(\Ccal)$, for a presentable category $\Ccal$, denotes either of the categories $\Cond^\omega(\Ccal)$ of light condensed objects or $\Cond_\kappa(\Ccal)$ of $\kappa$-condensed objects, see Section~\ref{sec:pro-and-condensed}.
\end{enumerate}
We notice that the existing terminologies for different variants of K-theory are not entirely consistent across the literature. We shall use the following terminology:
\begin{enumerate}
    \item[$\bullet$] We write $\K \colon \Catdual \To \Sp$ for the universal localising invariant on dualisable stable categories, i.e.\@ Efimov's extension of (an $\infty$-categorical version) of the categorical K-theory functor à la Blumberg--Gepner--Tabuada \cite{efimov-k-theory}.
    \item[$\bullet$] $\Kcont$ denotes Morrow's continuous K-theory as a spectrum-valued functor \cite[\S 3]{Morrow}.
    \item[$\bullet$] $\Kcontu$ denotes the pro-enrichement of $\Kcont$ by Kerz--Saito--Tamme \cite[\S 5]{kst-i}.
    \item[$\bullet$] $\Knuc \coloneqq \K\circ\Nuc$ denotes the K-theory of nuclear modules \cite[\S 5]{andreychev-thesis}.
    \item[$\bullet$] $\Knucu$ denotes a condensed enrichement of $\Knuc$ (Definition~\ref{Def:nuclear-K-theory}).
\end{enumerate}

\vspace{6pt}\noindent\textbf{Rigid spaces vs.\@ adic spaces.} In this article, we work with the notion of \emph{adic spaces}. On the other hand, in the article \cite{DY} the two first named authors worked with the category $\Rig$ of \emph{rigid spaces} à la Raynaud \cite{raynaud}, Abbes \cite{egr}, and Fujiwara--Kato \cite{fuji-kato}. Recall that there exists a fully faithful functor $\iota\colon\Adic^\unif \inj \Rig$ sending $\Spa(R,R^+)$ to $\Spf(R^+)^\rig$ which is compatible with gluing along open immersions \cite[\S 1.2]{AGV}. 
Moreover, for a nonarchimedean field $k$, Tate's classical category of rigid analytic $k$-varieties embeds fully faithfully into both $\Adic$ and $\Rig$; for uniform spaces these embeddings are compatible with the functor $\iota$. In the adic world, a (quasi-separated) rigid analytic $k$-variety corresponds to a (quasi-separated) adic space which is locally of finite type over $\Spa(k,k^\circ)$ \cite[Prop.~4.5]{huber1994generalization}.
Hence we can use the results of \cite{DY} about the rigid motivic $\Arig$-homotopy category $\RigSH(\base,\Vcal)$ (for a stable category $\Vcal$).
For us, the term \emph{rigid analytic space} and its abbreviation \emph{rigid space} refer to the platonic concept without choosing a concrete model category for it.

\section{Preliminaries on adic spaces}
\label{sec:adic-spaces}

In this section, we treat some background on adic spaces that is used in the sequel. Eventually, we may enrich the nuclear K-theory $\Knuc \coloneqq \K \circ \Nuc(-)$, which is valued in spectra, to a condensed K-theory $\Knucu$ valued in condensed spectra (Definition~\ref{Def:nuclear-K-theory}). In order to do that, we need the fact that the category of analytic adic spaces is tensored over $\ProFin$ (Lemma~\ref{lem:XotimesSsheafy}), which should be known to experts but which we could not find in the current literature. First, with any profinite set we associate an adic space.

\begin{defi}[{\cite[\S 4.1]{weinstein-scholze}}]
\label{Def:profinite-tensor-adic-2}
    For any profinite set $S$, let us define the pre-adic space
    \[\underline{S} \coloneqq \Spa(\Cont(S,\ZZ)) \coloneqq  \Spa(\Cont(S,\ZZ), \Cont(S, \ZZ)),\]
    where $\Cont(S,\ZZ)$ is the discrete ring of continuous (thus locally constant) maps from $S$ to $\ZZ$.
    Moreover, for any pre-adic space $X$, we set
    \[ S \otimes X \coloneqq \underline{S} \times_{\Spa(\ZZ)}X. \]
\end{defi}

\begin{rem}[{\cite[\S 4.1]{weinstein-scholze}}]
\begin{enumerate}
    \item The pre-adic space $\underline{S}$ is actually an adic space\footnote{%
        Let us stress that $\underline{S}$ is really an adic space but not an analytic adic space as it is \emph{not} locally Tate.
    }
    (since it is associated with a discrete Huber pair), and represents the functor:
    \[\Adic\op \to \Set, \quad X \mapsto \Cont(\topspace{X}, S).\]
    \item     The underlying topological space $\topspace{\underline{S}}$ is identified with $S \times \topspace{\Spa(\ZZ)}$.
    \item The assignment $S \mapsto \underline{S}$ defines a functor $\ProFin \inj \Adic$ which is fully faithful; indeed, we have
    \[\Hom_\Adic(\underline{T}, \underline{S}) \simeq \Cont(T \times \topspace{\Spa(\ZZ)}, S) \simeq \lim_n \colim_m \Cont(T_m \times \topspace{\Spa(\ZZ)}, S_n).\]
    But there is no nonconstant map to a profinite set from $\topspace{\Spa(\ZZ)}$ since the latter space has one single generic point, hence we obtain
    \[ \lim_n \colim_m \Cont(T_m \times \topspace{\Spa(\ZZ)}, S_n) \simeq \lim_n \colim_m \Cont(T_m, S_n) \simeq \Cont(T, S).\]
\end{enumerate}
\end{rem}

\begin{rem}
    If $S$ is a finite set, then $\underline{S}$ is canonically identified with $\bigsqcup_S \Spa(\ZZ)$. For a general profinite set with profinite presentation $S = \lim_i S_i$, we have canonical identifications
    \[\underline{S} = \Spa(\Cont(S, \ZZ)) \simeq \Spa(\colim_i \Cont(S_i, \ZZ)) \simeq \lim_i \Spa(\Cont(S_i, \ZZ)) = \lim_i \underline{S_i}\]
    in the category of (pre-)adic spaces. It follows immediately that $S \otimes X \simeq \lim_i (S_i \otimes X)$ as pre-adic spaces.
\end{rem}

The space $S \otimes X$ has a convenient description when $X$ is an affinoid (pre-)adic space (Lemma~\ref{lem:profinite-tensor-adic}).

Let $(A_0, I)$ be a pair of definition of the Huber pair $(A, A^+)$; we may assume $A_0 \subset A^+$. Then $\Cont(S, A^+)$ becomes naturally a topological ring such that $\Cont(S, A_0)$ is an open subring with topology generated by the ideal of definition $\Cont(S, I)$. Hence the following statement is meaningful.

\begin{lem}
\label{lem:profinite-tensor-adic}
    Let $S$ be a profinite set and let $X = \Spa(A, A^+)$ be an affinoid pre-adic space. 
    We have a canonical isomorphism of pre-adic spaces
    \[S \otimes X \simeq \Spa(\Cont(S, A), \Cont(S, A^+)).\]
    The Huber pair $(\Cont(S, A), \Cont(S, A^+))$ is complete if the pair $(A,A^+)$ is complete.
\end{lem}
\begin{proof}
    Replacing $(A, A^+)$ with its completion and 
    by construction of the fibre product of pre-adic spaces, the left-hand side is given by the pair
    \[(\Cont(S, \ZZ) \otimes_\ZZ A, \Cont(S, \ZZ) \otimes_\ZZ A^+) \simeq (\Cont(S, A^\discrete), \Cont(S, A^{+,\discrete})).\]
    Consider the ring $D_0 \coloneqq \Cont(S, A_0^\discrete) \simeq \Cont(S, \ZZ) \otimes A_0$ and its ideal $J \coloneqq \Cont(S, \ZZ) \otimes_\ZZ I$ (recalling that $\Cont(S, \ZZ)$ has discrete topology).
    Let us show that it defines the same adic spectrum as in the statement.
    The subring $\Cont(S, A^{+,\discrete}) \subset \Cont(S, A^\discrete)$ is already integrally closed.
    Observing that the fundamental neighbourhood of this pair is given by
    \[J^n = I^n \Cont(S, A_0^\discrete) = \Cont(S, I^n) \subset D_0 = \Cont(S, A_0^\discrete), \quad n \in \NN,\]
    we see immediately that it is a Huber pair whose completion is $(\Cont(S, A), \Cont(S, A^+))$. In particular, its adic spectrum represents the desired fibre product of pre-adic spaces.
\end{proof}

\begin{cor}
    \label{cor:profinite-tensor-adic}
    For a complete Huber pair $(A,A^+)$ and a profinite set $S=\lim_iS_i$ (where all $S_i$ are finite), the canonical morphism
    \[ (\colim_i\Cont(S_i,A),\colim_i\Cont(S_i,A^+)) \To (\Cont(S,A),\Cont(S,A^+)) \]
    exhibits the completion of the Huber pair in the source.
\end{cor}

\begin{proof}
We compute
    \begin{align*}
    \Spa(\Cont(S,A),\Cont(S,A^+))
    &\cong S\otimes X \\
    &\cong \lim_i(S_i\otimes X) \\
    &\cong \lim_i(\Spa(\Cont(S_i,A),\Cont(S_i,A^+)) \\
    &\cong \Spa(\colim_i\Cont(S_i,A),\colim_i\Cont(S_i,A^+))
    \end{align*}
where the first isomorphism is Lemma~\ref{lem:profinite-tensor-adic}.
\end{proof}

\begin{lem}
    \label{lem:XotimesSsheafy}
    For any analytic adic space $X$ and any profinite set $S$, the pre-adic space $S \otimes X$ is an analytic adic space. In other words, $\AnAdic$ is stable under the action $S \otimes (-)$ by any profinite set $S$, i.e.\@ the category $\AnAdic$ is tensored over the category $\ProFin$.
\end{lem}

\begin{proof}
    We have to show that the pre-adic space $S\otimes X$ is locally Tate and that the structure presheaf $\Ocal_{S \otimes X}$ is a sheaf.

    By a standard argument of \v{C}ech cohomology, we may assume $X = \Spa(A, A^+)$ with $A$ being a Tate ring. We may assume the Huber pair $(A, A^+)$ to be complete. Let $(A_0, I)$ be a pair of definition of $(A, A^+)$.

    First note that the space $S \otimes X$ is the adic spectrum associated with the complete Huber pair $(\Cont(S, A), \Cont(S, A^+))$ by Lemma~\ref{lem:profinite-tensor-adic}. If $A$ is a Tate ring with topologically nilpotent unit $\pi$, then $\Cont(S, A)$ is also a Tate ring with the constant function $\pi \in \Cont(S, A)$ being a topologically nilpotent unit.     
    
    In order to show sheafiness, let us write $S = \lim_{i \in I} S_i$ as a cofiltered limit of finite sets $S_i$.
    We have
    \begin{align*}
        S \otimes X & = \Spa \big( \Cont(S, \ZZ) \otimes_\ZZ A, \Cont(S, \ZZ) \otimes_\ZZ A^+ \big) \\
        & \simeq \Spa \big( \colim_i \Cont(S_i, A), \colim_i \Cont(S_i, A^+) \big),
    \end{align*}
    hence any rational subset of $S \otimes X$ is of the form $U = (S \otimes X)\left(\tfrac{f_1, \dots, f_n}{g}\right)$, for some given locally constant functions $f_1, \dots, f_n, g \in \Cont(S_i, A)$ for some $i \in I$.\footnote{%
        One may wonder why we may always take locally constant functions. In fact, here we are considering the adic spectrum of a Huber pair $(R, R^+)$ which is not necessarily complete. In the proof we claim that the structure presheaf of $\Spa(R, R^+)$ is sheafy.
        As for its associated completed Huber pair $(\hat R, \hat R^+)$, we have a natural isomorphism of pre-adic spaces $\Spa(\hat R, \hat R^+) \to \Spa(R, R^+)$ by \cite[Proposition~3.9, Lemma~3.10]{huber1993continuousvaluation}, the crucial reason being the approximation lemma (\cite[Lemma~3.10]{huber1993continuousvaluation}) for elements defining a rational subset.
    }
    
    Moreover, we can split $U$ as a disjoint union according to values of these functions on $S_i$. For $\alpha \in S_i$ denote by $S_\alpha \subset S$ its preimage. Then we have
    \begin{equation}
        \label{eq:decompSotimesXSheafy}
        U = \bigsqcup_{\alpha \in S_i} (S_\alpha \otimes X)\left(\tfrac{f_1(\alpha), \dots, f_n(\alpha)}{g(\alpha)}\right) \simeq \bigsqcup_{\alpha \in S_i} S_\alpha \otimes X\left(\tfrac{f_1(\alpha), \dots, f_n(\alpha)}{g(\alpha)}\right).
    \end{equation}
    The last isomorphism follows from the fact that both sides are the completion of the finite product of
    \begin{equation*}
        \big(\Cont(S_\alpha, \ZZ) \otimes_\ZZ A\big)[\tfrac{1}{g}] \simeq \Cont(S_\alpha, \ZZ) \otimes_\ZZ A[\tfrac{1}{g}], \qquad \alpha \in S_i
    \end{equation*}
    with ideal of definition $C(S_\alpha, \ZZ) \otimes_\ZZ \big(I \cdot A_0[\tfrac{1}{g}]\big)$. Hence, the global sections of $\Ocal_{S \otimes X}$ on the spaces at the right hand side of (\ref{eq:decompSotimesXSheafy}) are given by
    \begin{equation}
        \label{eq:globalSectionSotimesXSheafy}
        \Cont \Big(S_\alpha, A \big\langle \tfrac{f_1(\alpha), \dots, f_n(\alpha)}{g(\alpha)} \big\rangle \Big).
    \end{equation}

    Let $\Ucal$ be a cover of $S\otimes X$. It can be refined by a finite cover by rational subsets. For any such finite cover, by the decomposition in the previous paragraph, we may assume that $\Ucal$ comes by base change from a finite rational cover $\Ucal_0$ of $X$, i.e.\@ we may take $S_i$ to be the singleton $\{\alpha\}$ and $S_\alpha = S$ in the above. The sheaf condition of $\Ocal_{S \otimes X}$ for $\Ucal = S \otimes \Ucal_0$ then follows from the sheaf condition of $\Ocal_X$ for $\Ucal_0$. Indeed, as in the proof of \cite[Thm.~7]{buzzardverberkmoes2018stablyuniformsheafy}, by standard considerations of Laurent covers and induction, we may reduce to checking the sheaf condition for $\Ucal_0 = \{X_1, X_2\}$ being a standard binary rational cover of $X$. By the sheaf condition of $\Ocal_X$ for $\Ucal_0$, we have a left exact sequence
    \begin{equation}
        \label{eq:sheafyBase}
        0 \to \Ocal_X(X) \to \Ocal_X(X_1) \oplus \Ocal_X(X_2) \to \Ocal_X(X_1 \cap X_2),
    \end{equation}
    which is then strict by \cite[Lem.~2]{buzzardverberkmoes2018stablyuniformsheafy} (here we have the hypothesis that $A$ is a Tate ring).
    Applying $\Cont(S, -)$ to this strictly left exact sequence, we obtain a short left exact sequence\footnote{%
        The left exact sequence (\ref{eq:sheafyBaseTensorS}) is even right exact. Indeed, the sequence (\ref{eq:sheafyBase}) is even surjective on the right by \cite[Lem.~2]{buzzardverberkmoes2018stablyuniformsheafy}. As $A$ is a Tate ring, it is a surjection of ultrametric Banach spaces, hence it admits a continuous (topological) section by Michael’s zero-dimensional selection theorem.
    }
    \begin{equation}
        \label{eq:sheafyBaseTensorS}
        0 \to \Ocal_{S \otimes X}(S \otimes X) \to \Ocal_{S \otimes X}(S \otimes X_1) \oplus \Ocal_{S \otimes X}(S \otimes X_2) \to \Ocal_{S \otimes X}(S \otimes X_1 \cap S \otimes X_2).
    \end{equation}
    This completes the proof of sheafiness of $\Ocal_{S \otimes X}$.
\end{proof}

\begin{lem} 
\label{Lem:EDS-tensor-Nisnevich-square}
    For every Nisnevich square
    \begin{equation*}
    \begin{tikzcd}
        \label{Diagram:NisnevichSquare}
        \tag{$\square$}
		V \arrow[r,""]\arrow[d,""]& Y \arrow[d,""]\\
		U \arrow[r,""]& X
    \end{tikzcd}
    \end{equation*}
    of analytic adic spaces and every $S\in\ProFin$, the tensored square
    \begin{equation*}
    \begin{tikzcd}
        \label{Diagram:NisnevichSquareTensorS}
        \tag{$S \otimes \square$}
		S \otimes V \arrow[r,""]\arrow[d,""]& S \otimes Y \arrow[d,""]\\
		S \otimes U \arrow[r,""]& S \otimes X
    \end{tikzcd}
    \end{equation*}
    is a Nisnevich square, too.
    Moreover, for every cover $X=\bigcup_{i\in I}X_i$ in the analytic topology, we get an induced cover $S\otimes X = \bigcup_{i\in I} S\otimes X_i$.
\end{lem}
\begin{proof}
    The diagram (\ref{Diagram:NisnevichSquareTensorS}) is the base change, by definition, of the diagram (\ref{Diagram:NisnevichSquare}) along the map $S \otimes X \to X$ by Definition~{\ref{Def:profinite-tensor-adic-2}}, hence it is also a Nisnevich square since Nisnevich covers are stable under base change of analytic adic spaces by \cite[Lemma~A.12]{andreychev-thesis}.
    Similarly, the covers in the analytic topology are also preserved under base change.
\end{proof}

As an application to the tensor-enrichment, we later shall lift functors on analytic adic spaces with values in a stable category $\Vcal$ to an enriched functor with values in the condensed category $\Cond(\Vcal)$, see Section~\ref{subsec:lifting-functors-to-condensed}.

\section{Preliminaries on pro-categories and condensed categories}
\label{sec:pro-and-condensed}
In the main body of this article, we deal with continuous K-theory as defined and studied by Kerz-Saito-Tamme \cite{kst-i,kst-ii} which is a functor
\[ \Kcontu \colon \AnAdic \To \Pro(\Sp) \]
with values in the category of pro-spectra. Since pro-categories are not that well-behaved, e.g.\@ computing colimits in them is not easy, we eventually will postcompose $\Kcontu$ with the canonical functor $\Pro(\Sp)\to\Cond(\Sp)$ to condensed spectra.

\subsection{Pro-categories} 
\label{subsec:pro-categories}

Analytic K-theory of rigid spaces is naturally a pro-object in the category of spectra. Let us quickly recall some facts about pro-categories  \cite[\S A.8.1]{SAG}. 
Let $\Ccal$ be a category which is accessible and admits finite limits.
Then we can define $\Pro(\Ccal)$ as the opposite of the category of functors in $\Fun(\Ccal,\Spc)$ that are accessible and preserve finite limits. Equivalently, we have $\Pro(\Ccal)\simeq \Ind(\Ccal\op)\op$. An object $X\in\Pro(\Ccal)$ can be thought of a limit of a cofiltered  diagram $\lbrace X_{i}\rbrace$, where each $X_{i}$ is contained in the essential image of the Yoneda embedding $\Ccal\rightarrow \Pro(\Ccal)$. Following \cite{kst-i}, we denote these objects as $\prolim X_{i}$.
If $\Ccal$ admits all cofiltered limits, we have a limit functor $\Pro(\Ccal)\rightarrow \Ccal$ that maps an object $\prolim X_i$ to $\lim X_i \in \Ccal$. 

Let us remark that finite limits and finite colimits in $\Pro(\Ccal)$ can be computed pointwise in the following sense (if they exist in $\Ccal$). Let $K$ be a finite simplicial set. The limit functor $\Fun(K,\Ccal)\to \Ccal$ induces a functor $\Pro(\Fun(K,\Ccal))\rightarrow \Pro(\Ccal)$, which we call \textit{level-wise limit} (similarly for colimits). Furthermore, there exists a functor $\Pro(\Fun(K,\Ccal))\rightarrow \Fun(K,\Pro(\Ccal))$, using which one can show that the levelwise limit of $F\in \Pro(\Fun(K,\Ccal))$ is the limit of the induced diagram $K\rightarrow \Pro(\Ccal)$ (similarly for finite colimits) \cite[Lem.~2.1]{kst-i}.

In any case, the category $\Pro(\Ccal)$ is complete. If $\Ccal$ admits finite colimits resp.\@ all small colimits, then $\Pro(\Ccal)$ does so, too \cite[Lem.~2.3]{kst-i}.

Let $\Ccal = \Sp$. Then $\Pro(\Sp)$ is stable (Lemma~\ref{Pro-omega-stable--lemma}) and we can define homotopy groups on $\Pro(\Sp)$ pointwise \cite{kst-i}. To be more precise, for any $n\in \ZZ$ and any $\prolim X_i\in \Pro(\Sp)$, we set 
$$
    \pi_{n} (\prolim X_i )\coloneqq \prolim \pi_{n} X_{i} \in \Pro(\Ab).
$$
We will see that $\Pro(\Sp)$ admits a $t$-structure  and the pointwise homotopy groups are precisely the homotopy groups induced by this $t$-structure (Lemma \ref{lem-t-structure-pro}).

Lastly, let us recall the notion of \textit{weak equivalence}. In op.\! cit. it is shown that there is a functor $\iota^{*}\colon \Pro(\Sp)\rightarrow \Pro(\Sp^{+})$, where $\Sp^{+}$ denotes the category of bounded above spectra. On objects this functor is given by 
$$
    \iota^{*}(\prolim_I X_i)\simeq \prolim_{\NN\times I} \tau_{\leq n}X_{i}.
$$
A morphism of pro-spectra $f\colon X\rightarrow Y$ is a weak equivalence if $\iota^{*}f$ is an equivalence. Analogously, a fibre sequence in $\Pro(\Sp)$ is called a \textit{weak fibre sequence} if it is a fibre sequence after applying $\iota^{*}$.
Note that $f$ is  weak equivalence if and only if $\tau_{\leq n}f$ is an equivalence for some $n\in\ZZ$ and $\pi_{i}f$ is an isomorphism for all $i\in \ZZ$ \cite[Lem.~2.8]{kst-i}.

To simplify our language, for two pro-spectra $X$ and $Y$, when we talk about their equivalence in $\Pro(\Sp^+)$, we will mean an equivalence between $\iota^* X$ and $\iota^* Y$.

\begin{defi} \label{Def:Pro-omega}For a \infcat{} $\Ccal$ denote by $\Pro^\omega(\Ccal)$ the full subcategory of $\Pro(\Ccal)$ of objects $X$ that are representable by diagrams $X \colon I\to\Ccal$ for a \emph{countable} cofiltered category $I$.
\end{defi}

\begin{lemma} \label{Pro-omega-stable--lemma}
Let $\Ccal$ be a stable category. Then the categories $\Pro(\Ccal)$ and $\Pro^\omega(\Ccal)$ both are stable.
In particular, the categories $\Pro(\Sp^+)$ and $\Pro^\omega(\Sp^+)$ are stable.
\end{lemma}
\begin{proof}
Since $\Ccal$ is stable, the pro-category $\Pro(\Ccal)$ is stable as well \cite[2.5]{kst-i}. Since finite colimits in pro-categories can be computed levelwise, the inclusion $\Pro^\omega(\Ccal) \inj \Pro(\Ccal)$ creates finite colimits. Hence $\Pro^\omega(\Ccal)$ admits finite colimits and the suspension functor is an equivalence. Thus $\Pro^\omega(\Ccal)$ is stable \cite[1.4.2.27]{HA}.
The rest follows as the category $\Sp^+$ is a stable subcategory of $\Sp$ \cite[p.~44]{HA}.
\end{proof}

In the work of Kerz-Saito-Tamme, the notion of \emph{pro-homotopy invariance} plays a major role when working with rings \cite{kst-bass-quillen}, \cite[\S 5.3]{kst-i}. For sheaves, pro-homotopy invariance is equivalent to $\Arig$-invariance.

\begin{lemma} \label{pro-invariant=Aan-invariant--lem}
A sheaf $F \colon \AnAdic\op \to \Pro(\Ccal)$ is $\Arig$-invariant if and only if it is pro-$\Brig$-invariant, i.e. if the canonical map 
    \[
	F(X) \To \prolim_{t\mapsto\pi t} F(X\times\Brig)
	\]
is an equivalence in $\Pro(\Ccal)$ for any $X \in \AnAdic$.
\end{lemma}
\begin{proof}
Note that a pro-system $\prolim_iC_i$ is the limit of the constant pro-systems given by the $C_i$ in $\Pro(\Ccal)$. Hence we compute
	\begin{align*}
	F(X\times\Arig) 
	\simeq F(\colim_{t\mapsto\pi t}(X\times\Brig)) 
	\simeq \prolim_{t\mapsto\pi t} F(X\times\Brig).
	\end{align*}
Hence the map $F(X) \to F(X\times\Arig)$ is an equivalence if and only if the map $F(X) \to \prolim_{t\mapsto\pi t} F(X\times\Brig)$ is an equivalence.
\end{proof}

\begin{lem}
\label{lem:pro-spaces-universal-colimits}
    For any pro-object of simplicial spaces $X_\bullet = \prolim_\lambda X_\bullet^{(\lambda)}$, the natural map of pro-spaces 
    \[ \colim_{\Delta\op} X_\bullet \To \prolim_\lambda (\colim_{\Delta\op} X_\bullet^{(\lambda)}) \]
    is a weak equivalence.
\end{lem}
\begin{proof}
The assertion is a variant of \cite[Lem.~2.8]{kst-ii} which relies on a general result \cite[Prop.~1.2.4.5]{HA} for \emph{stable} categories, whereas here $\Pro(\Spc)$ is not stable. However, a similar argument using Postnikov filtration works. A precise proof of the statement is given in \cite[Prop.~1.17]{haine-2026}; see also \cite[Prop.~2.2]{bachmann-engelmann-mattis-2025-monadic} for a generalisation of it from pro-spaces to pro-$\infty$-topoi.
\end{proof}

\subsection{Condensed objects}
\label{subsec:recollections-condensed}

\vspace{6pt}\noindent\textbf{Notation.} For this entire section, we fix an uncountable strong limit cardinal $\kappa$.
Given a category $\Ccal$, we work with the categories $\Cond^\omega(\Ccal)$ of light condensed objects in $\Ccal$ and $\Cond_\kappa(\Ccal)$ of $\kappa$-condensed objects in $\Ccal$. One advantage is the following: If the category $\Ccal$ is presentable, then $\Cond^\omega(\Ccal)$ and $\Cond_\kappa(\Ccal)$ are also presentable (Lemma~\ref{condensed-objects-presentable--lem}). Hence we can apply Adjoint Functor Theorems, for instance. We also write $\Cond(\Ccal)$ for either $\Cond_\kappa(\Ccal)$ or $\Cond^\omega(\Ccal)$ when there is no need for specification. The material about light condensed objects is extracted from the lectures by Clausen and Scholze on Analytic Stacks.%
\footnote{%
    See \url{https://www.youtube.com/playlist?list=PLx5f8IelFRgGmu6gmL-Kf_Rl_6Mm7juZO}.
}
For a written  we refer to Rodríguez Camargo's notes \cite{camargo2026notesolid}.

\begin{defi}[Condensed objects]
\begin{enumerate}
\item A profinite set is called \emph{light} if it is isomorphic to a countable limit of finite sets. Denote by $\ProFin^\omega$ the full subcategory of $\ProFin$ that is spanned by light profinite sets. 
We equip the category $\ProFin^\omega$ with the topology generated by finite families of jointly surjective maps.
\item Analogously, let $\ProFin_{<\kappa}$ be the full subcategory of $\ProFin$ that is spanned by $\kappa$-small sets, equipped with the topology generated by finite families of jointly surjective maps.
\end{enumerate}
Let $\Ccal$ be a category that admits finite limits.
\begin{enumerate} \setcounter{enumi}{2}
\item A \emph{light condensed object in} $\Ccal$ is a hypersheaf $X$ on $\ProFin^\omega$ with values in $\Ccal$.
We denote $\Cond^\omega(\Ccal) \coloneqq \Sh^{\wedge}(\ProFin^\omega,\Ccal)$ the category of light condensed objects in $\Ccal$.
\item A $\kappa$\emph{-small condensed object in} $\Ccal$ is a hypersheaf on $\ProFin_{<\kappa}$ with values in $\Ccal$.
We denote $\Cond_\kappa(\Ccal) \coloneqq \Shhyp(\ProFin_{<\kappa} ,\Ccal)$ the category of $\kappa$-small condensed objects in $\Ccal$.
\end{enumerate}
\end{defi}

\begin{rem}
Explicitly, a presheaf $X\colon (\ProFin^\omega)\op\to\Ccal$ is a \emph{sheaf} if and only if it satisfies
\begin{enumerate}
    \item $X(\varnothing)\simeq *$,
    \item $X(S\sqcup S') \to[\simeq] X(S)\times X(S')$, and
    \item $X(S) \to[\simeq] \lim_\Delta X(\Check{C}(T\to S))$ 
\end{enumerate}
for all $S,S'$ and every surjective map $T\surj S$ in $\ProFin^\omega$.\\
A sheaf $X\colon (\ProFin^\omega)\op\to\Ccal$ is called a \emph{hypersheaf} if for all hypercovers $S_\bullet\to S$, the natural map
\[
    X(S)\to \lim_{\Delta} X(S_{\bullet})
\]
is an equivalence.
\end{rem}

\begin{lemma} \label{condensed-objects-presentable--lem}
Let $\Ccal$ be a presentable category. Then the categories $\Cond^\omega(\Ccal)$ and $\Cond_\kappa(\Ccal)$ are presentable.
\end{lemma}
\begin{proof}
Being a category of sheaves on the small site $\ProFin^\omega$, the category $\Cond^\omega(\Spc)$ is an $\infty$-topos \cite[6.2.2.7]{HTT}.
We can identify $\Cond^\omega(\Ccal)$ with the hypercomplete objects inside $\Sh(\Cond^\omega(\Spc),\Ccal)$. The latter category is presentable as $\Ccal$ is presentable \cite[1.3.1.6]{SAG}, hence $\Cond^\omega(\Ccal)$ is also presentable by \cite[Propositions~6.5.2.8 and 5.5.4.15]{HTT}.
Same argument for $\Cond_\kappa(\Ccal)$.
\end{proof}

\begin{lemma} \label{condensed-objects-stable--lem}
Let $\Ccal$ be a stable category. Then the categories $\Cond^\omega(\Ccal)$ and $\Cond_{\kappa}(\Ccal)$ are stable.
\end{lemma}
\begin{proof}
As $\Ccal$ is stable, the functor category $\Fun(\ProFin^{\omega,\opp},\Ccal)$ is stable \cite[1.1.3.1]{HA} and the suspension functor is an equivalence. Being right adjoint to the hypersheafification functor, the inclusion functor $\Cond^\omega(\Ccal) \inj \Fun(\ProFin^{\omega,\opp},\Ccal)$ creates finite limits. Hence, the category $\Cond^\omega(\Ccal)$ admits finite limits and the suspension functor is an equivalence, thus the category is stable \cite[1.4.2.27]{HA}.
Same argument for $\Cond_\kappa(\Ccal)$.
\end{proof}

\begin{lemma} \label{condensed-objects-symmetric-monoidal}
Let $\Ccal$ be a category that admits small colimits and finite limits. Then the categories $\Cond^\omega(\Ccal)$ and $\Cond_\kappa(\Ccal)$ both admit finite limits. In particular, they admit the structure of a cartesian symmetric monoidal category. 
\end{lemma}
\begin{proof}
By assumption on $\Ccal$, the functor category $\Fun(\ProFin^{\omega,\opp},\Ccal)$ admits small colimits and finite limits \cite[5.1.2.3]{HTT} and this property passes to its full subcategory $\Cond^\omega(\Ccal)$ of sheaves.
Now the statement about the cartesian symmetric monoidal category from the existence of finite products \cite[2.4.1.5.(5)]{HA}.
Same argument for $\Cond_\kappa(\Ccal)$.
\end{proof}

\begin{remark}[Extremally disconnected sets]
\label{extremally-disconnected--rem}
A profinite set is called an \emph{extremally disconnected set} if it is a projective object in the category $\ProFin$. Denote by $\EDS_{<\kappa}$ the full subcategory of $\ProFin_{<\kappa}$ spanned by extremally disconnected sets. By definition of being projective, every cover of profinite sets $S'\to S$ with $S$ extremally disconnected has a split. In particular, the identity map $\id_S$ is cofinal among all covers of $S$. We shall use this property in the proof of Lemma~\ref{pro-vanishing-and-Mittag-Leffler--lemma} below.

As a matter of fact, the restriction along the inclusion functor $\EDS_{<\kappa} \inj \ProFin_{<\kappa}$ induces an equivalence from the category $\Cond_\kappa(\Ccal)$ onto the full subcategory\linebreak $\Fun^\times(\EDS_{<\kappa}\op,\Ccal)$ of $\Fun(\EDS_{<\kappa}\op,\Ccal)$ which is spanned by those objects that preserve finite products, i.e.\@ sending finite disjoint unions in $\EDS_{<\kappa}$ to products in $\Ccal$.
\end{remark}

\begin{remark}[light vs. $\kappa$-small] 
\label{comparison-light-kappa-condensed}
Let $\Ccal$ be a category that admits limits and colimits. The restriction along the inclusion functor $j\colon \ProFin^\omega \inj \ProFin_{<\kappa}$ induces an adjunction

\begin{center}
\begin{tikzcd}
 \Fun(\ProFin^{\omega,\opp},\Ccal) \arrow[r, bend left, "j_!"] \arrow[r, bend right, "j_*", swap] &  \Fun(\ProFin_{<\kappa}\op,\Ccal) \arrow[l, "j^*", swap]
\end{tikzcd}
\end{center}
where $j_!$ is a left Kan extension and $R \coloneqq j_*$ is a right Kan extension.
Hence the functor $j^*$ preserves limits, so that it preserves sheaves and restricts to a functor $j^* \colon \Cond_\kappa(\Ccal)\to\Cond^\omega(\Ccal)$. Now the localisation functor
\[ \L_\kappa \colon \Fun(\ProFin_{<\kappa}\op,\Ccal) \to \Cond_\kappa(\Ccal) \simeq \Fun^{\times}(\EDS\op_{<\kappa},\Ccal) \]
factors over the restriction functor $\Fun(\ProFin_{<\kappa}\op,\Ccal) \to\Fun(\EDS_{<\kappa}\op,\Ccal)$. Given a light condensed object $X\in\Cond^\omega(\Ccal)$ and $S_1,S_2\in\EDS_{<\kappa}$, we have
\begin{align*}
    RX(S_1\sqcup S_2) 
    &\simeq \lim_{\ProFin^\omega\ni W\to S_1\sqcup S_2} X(W) \\
    &\simeq \lim_{\ProFin^\omega\ni W_1\to S_1} X(W_1) \times \lim_{\ProFin^\omega\ni W_2\to S_2} X(W_2) \\
    &\simeq RX(S_1) \times RX(S_2).
\end{align*} 
Hence the adjunctions above induce adjunctions 
\begin{center}
\begin{tikzcd}
 \Cond^{\omega}(\Ccal) \arrow[r, bend left, "\L_\kappa j_!"] \arrow[r, bend right, "j_*", swap] &  \Cond_{\kappa}(\Ccal) \arrow[l, "j^*", swap]
\end{tikzcd}
\end{center}
where $j^*$ and $j_*$ are just the restricted functors.
Via the explicit formula for Kan extensions we see that the composition $j^*\circ R$ is equivalent to the identity functor, hence $R$ is fully faithful. Since $\L_\kappa$ amounts to sheafification, we have
\[ \L_\kappa j_! X(S_1 \sqcup S_2) \simeq \L_\kappa j_!X(S_1) \times \L_\kappa j_!X(S_2). \]
If filtered colimits commute with finite products in $\Ccal$, analogously as above we have
\begin{align*}
    j_!X(S_1\sqcup S_2) 
    &\simeq \colim_{S_1\sqcup S_2 \to W\in\ProFin^\omega} X(W) \\
    &\simeq \colim_{S_1 \to W_1\in\ProFin^\omega} X(W_1) \times \colim_{S_2 \to W_2\in\ProFin^\omega} X(W_2) \\
    &\simeq j_!X(S_1) \times j_!X(S_2).
\end{align*} 
\end{remark}

Next, we examine the relation between sheafification and stabilisation of presheaf categories. This will apply to the passage to condensed objects. First, we start with a general lemma:

\begin{lemma} \label{condensed-sheaves--lem}
Let $(\Ccal,\tau)$ and $(\Dcal,\sigma)$ be two sites. Then:
\begin{enumerate}
\item $\Fun(\Ccal\op,\Sh_\sigma(\Dcal)) \simeq \Sh_\sigma(\Dcal,\PSh(\Ccal))$;
\item $\Sh_\tau(\Ccal,\Sh_\sigma(\Dcal)) \simeq \Sh_{\tau\times\sigma}(\Ccal\times\Dcal)$, where $(\Ccal \times \Dcal, \tau \times \sigma)$ is the product site;
\item $\Shhyp_\tau(\Ccal, \Sh_\sigma(\Dcal)) \simeq \Sh_\sigma(\Dcal, \Shhyp_\tau(\Ccal))$.
\end{enumerate}
\end{lemma}
\begin{proof}
(1) Both categories are canonically subcategories of $\Fun(\Ccal\op\times\Dcal\op,\Spc)$. Since equivalences of presheaves can be computed pointwise, the condition for lying in both sides amounts to the sheaf condition for $\sigma$, hence the subcategories agree.\\
(2) This holds since one can check the respective sheaf conditions in each factor independently. \\
(3) This follows from (1) by checking $\tau$-descent on both sides.
\end{proof}

\begin{cor} \label{condensed-sheaves--cor}
Given a site $(\Ccal,\tau)$, there is a canonical equivalence
\[
\Sh_\tau(\Ccal,\Cond^{\omega}(\Spc)) \simeq \Cond^{\omega}(\Sh_\tau(\Ccal))
\]
between the categories of $\tau$-sheaves of condensed spaces and condensed $\tau$-sheaves.
The analogous result for $\Cond_\kappa(-)$ is true as well.
\end{cor}

\begin{lemma} \label{sheaves-and-stabilisation--lem}
Let $(\Ccal,\tau)$ be a site and $\Dcal$ be a category. There is a canonical equivalence
\[
\Sh_\tau(\Ccal,\Sp(\Dcal)) \simeq \Sp\bigl( \Sh_\tau(\Ccal,\Dcal) \bigr) \quad \text{resp.} \quad \Shhyp_\tau(\Ccal,\Sp(\Dcal)) \simeq \Sp\bigl( \Shhyp_\tau(\Ccal,\Dcal) \bigr) 
\]
between $\tau$-sheaves (resp. $\tau$-hypersheaves) of spectra and spectrum objects in $\tau$-sheaves (resp. in $\tau$-hypersheaves).
\end{lemma}

\begin{proof}
This is \cite[1.3.2.2]{SAG} applied to the category $\Sh_\tau(\Ccal,\Dcal)$.
The argument loc.\@ cit.\@ deals with $\Spc$-valued sheaves, but the proof remains valid for $\Dcal$-valued sheaves, in view of the defintion of $\Sp(\Dcal)$ \cite[Definition~1.4.2.8]{HA} as being the category of the reduced and excisive functors from the category of pointed finite spaces $\Spc^\fin_*$ to $\Dcal$.
Moreover, the claim for hypersheaves holds since the hypercompleteness on both sides coincides.
\end{proof}

\begin{cor} \label{condensed-and-stabilisation--cor}
For any \infcat{} $\Dcal$ there is a canonical equivalence
\[
\Cond^{\omega}(\Sp(\Dcal)) \simeq \Sp(\Cond^{\omega}(\Dcal)).
\]
The analogous result for $\Cond_\kappa(-)$ is true as well.
\end{cor}
\begin{proof}
This follows from Lemma~\ref{sheaves-and-stabilisation--lem}.
\end{proof}

In order to render the functor $\Kan_{\geq 0} \colon \Rig_k \to \Pro(\Spc) \to \Cond^{\omega}(\Spc)$ representable, we need to have internal condensed mapping spectra in $\RigH(k,\Cond^{\omega}(\Spc))$. By \cite[Rem.~3.17]{DY} it is enough to see that the $\otimes$-product in $\Cond^{\omega}(\Spc)$ given by Lemma \ref{condensed-objects-symmetric-monoidal} preserves colimits in each variable. 

\begin{lem}
    The cartesian symmetric monoidal structure on $\Cond(\Spc)$ is closed, i.e.\@ for any objects $Y,Z\in\Cond(\Spc)$ there exists an internal mapping object $\Mapu(Y,Z)\in\Cond(\Spc)$ such that for any $X\in\Cond(\Spc)$ there is a canonical equivalence 
\[
\Map_{\Cond(\Spc)}(X,\Mapu(Y,Z)) \simeq \Map_{\Cond(\Spc)}(X\times Y,Z).
\]
Setting $X=y(S)$ for $S \in \ProFin$, we get that
\[
\Mapu(Y,Z)(S) \simeq \Map_{\Cond(\Spc)}(y(S),\Mapu(Y,Z)) \simeq \Map_{\Cond(\Spc)}(y(S)\times Y,Z).
\]
Moreover, for any $\Vcal\in \CAlgPr$, the category $\Cond(\Vcal)$ admits a closed monoidal structure.
\end{lem}
\begin{proof}
    This follows immediately from the definition, since $\Cond^{\omega}(\Spc)$ and $\Cond_\kappa(\Spc)$ are $\infty$-topoi and thus colimits are universal there. If $\Vcal$ is a symmetric monoidal presentable category, then we have $\Cond^{\omega}(\Spc)\otimes \Vcal\simeq \Cond^{\omega}(\Vcal)$ and likewise for $\Cond_\kappa(-)$ \cite[Prop. 2.4]{DrewMHM}.
\end{proof}

\subsection{t-structures}
In this subsection, we examine a t-structure on condensed spectra which is induced by the canonical t-structure on spectra.

\begin{lemma}
There exists a t-structure on $\Cond_\kappa(\Sp)$ with connective part $\Cond_\kappa(\Sp)_{\geq 0} = \Cond_\kappa(\Sp_{\geq 0})$ and coconnective part $\Cond_\kappa(\Sp)_{\leq 0} = \Cond_\kappa(\Sp_{\leq 0})$.
Analogously, there exists a t-structure on $\Cond^\omega(\Sp)$ with connective part $\Cond^\omega(\Sp)_{\geq 0} \simeq \Cond^\omega(\Sp_{\geq 0})$ and coconnective part $\Cond^\omega(\Sp)_{\leq 0} = \Cond^\omega(\Sp_{\leq 0})$.
\end{lemma}
\begin{proof}
Since $\Sp_{\geq 0}$ is presentable, this inherits to the presheaf category $\Fun(\EDS_{<\kappa}\op,\Sp_{\geq 0})$ \cite[5.5.3.6]{HTT} and its Bousfield localisation $\Cond_\kappa(\Sp_{\geq 0})$ \cite[5.5.4.15]{HTT}. The category $\Sp_{\geq 0}$ is closed under small colimits in $\Sp$ since the inclusion functor is left-adjoint. Since colimits in presheaf categories can be computed pointwise, $\Fun(\EDS_{<\kappa}\op,\Sp_{\geq 0})$ is closed under small colimits in $\Fun(\EDS_{<\kappa}\op,\Sp)$. Since a colimit of sheaves is given by the sheafification of the colimits of the underlying presheaves, $\Cond_\kappa(\Sp_{\geq 0})$ is closed under small colimits in $\Cond_\kappa(\Sp)$. From the long exact sequences of homotopy groups one sees that $\Cond_\kappa(\Sp_{\geq 0})$ is closed under extension in $\Cond_\kappa(\Sp)$. These three closure properties imply the existence of a t-structure on $\Cond_\kappa(\Sp)$ whose connective part is $\Cond_\kappa(\Sp_{\geq 0})$ \cite[1.4.4.11]{HA}.

The category $\Cond_\kappa(\Sp)_{\geq 0}$ is the same as the full subcategory of 0-connective objects in the sense of \cite[Def.~1.3.2.5]{SAG}. According to \cite[Prop.~1.3.2.7]{SAG}, the coconnective part of the t-structure is given by the full subcategory of 0-truncated objects; these are precisely those sheaves that take value in 0-truncated spectra \cite[Rem.~1.3.2.6]{SAG} (since the inclusion functor $\Sp_{\leq 0} \inj \Sp$ preserves limits). 

The case of $\Cond^\omega(\Sp)$ works similarly; we refer to \cite[\S 1.3.2 and \S 1.3.5]{SAG} and \cite[B.1.6]{zhou-these} for details.
\end{proof}

If we consider countable products in $\Cond(\Sp)$, then we can prove that these are $t$-exact. This will be crucial for Theorem \ref{thm:comparison-conservative}.

\begin{lem}
\label{Lem:count-prod-t-exact}
    Countable products in $\Cond^{\omega}(\Sp)$ are $t$-exact.
\end{lem}
\begin{proof}
    Let $(X_i)_{i \in I}$ be a countable family of condensed spectra. Let $n \in \ZZ$. The fundamental group $\pi_n(\prod_{i \in I} X_i)$ can be computed as the sheafification of the presheaf
    \[\pi_n^\pre(\prod_{i \in I} X_i) \colon \ProFin^\omega \to \Ab, \quad S \mapsto \pi_n \big( \prod_{i \in I} X_i(S) \big) \simeq \prod_{i \in I} \pi_n (X_i(S)) \simeq \prod_{i \in I} \pi_n^\pre(X_i)(S),\]
    where the second to last equivalence is due to the t-exactness of products on $\Sp$.
    That is $\pi_n(\prod_{i \in I} X_i)$ is identified with the sheafification of $\prod_{i \in I} \pi_n^\pre(X_i)$.
    The presheaves $\pi_n^\pre(X_i)$ are multiplicative (see \cite[Proposition~3.21]{mondalreinecke2025postnikov}).
    Since the $1$-topos $\Cond_\kappa(\Set)$ (resp. $\Cond^\omega(\Set)$) is replete \cite[Theorem~2.2]{condensed} (resp. \cite[Remark~2.2.2]{camargo2026notesolid}), sheafification commutes with countable products of multiplicative presheaves \cite[Proposition~2.14]{mondalreinecke2025postnikov}. (Beware that sheafification does not necessarily commute with countable products of arbitrary presheaves.)
\end{proof}

\subsection{From pro-objects to condensed objects}
\label{subsec:pro-to-cond}

\textbf{Notation.} In this subsection, let $\Ccal$  be a category that is accessible, admits colimits and admits finite limits.
Denote by $\Pro(\Ccal)$ the associated category of pro-objects $\Ccal$, cf.\@ the beginning of subsection~\ref{subsec:pro-categories}.
We have adjunctions
	\[
	\const \colon \Ccal \rightleftarrows \Pro(\Ccal) \colon \mathrm{lim} \quad\quad\quad \const\colon\Ccal \rightleftarrows \Cond^{\omega}(\Ccal) \colon \can
	\]
where the left-adjoints are the respective constant functors which are fully faithful.

\begin{lemma} \label{comparison-functor--lemma}
Assume additionally that $\Ccal$ admits small cofiltered limits.\footnote{%
    In particular, $\Ccal$ admits all limits since it has pullbacks and products (which are cofiltered limits of finite products) \cite[Prop. 4.4.2.6]{HTT}.
}
Then there is a canonical comparison functor 
\[ \gamma_\kappa \colon \Pro(\Ccal) \to \Cond_{\kappa}(\Ccal) \]
commuting with small limits and fitting into a commutative diagram
\[ \begin{xy} \xymatrix{
	&\Ccal \ar[dl]_{\const} \ar[d]^{\const} \ar[dr]^{\const} \\
	\Pro(\Ccal) \ar[dr]_{\mathrm{lim}} \ar[r]^-{\gamma_\kappa} & \Cond_{\kappa} \ar[d]^{\can}(\Ccal) \ar[r]^{j^*} & \Cond^\omega(\Ccal) \ar[dl]^{\can}
	 \\ & \Ccal
} \end{xy} \]
Explicitly, the functor $\gamma_\kappa$ sends the pro-object $\prolim_{i\in I}X_i$ to the limit $\lim_{i\in I}\underline{X_i}$.
\end{lemma}
\begin{proof}
The commutativity of the lower right triangle follows since can is given by evaluating at a single point.
For showing the commutativity of the upper right triangle consider the following square of left adjoint functors
\[ \begin{tikzcd}
    \Fun(\ProFin_{<\kappa}\op,\Ccal) \arrow[d,"\L_\kappa",swap]\arrow[r,"j^*"]
    &  \Fun(\ProFin^\omega,\Ccal) \arrow[d,"\L^\omega"]\\
    \Cond_\kappa(\Ccal) \arrow[r,"j^*"] & \Cond^\omega(\Ccal).
\end{tikzcd} \]
This square commutes because the associated square of right adjoint functors commutes.
Since $\Ccal$ admits small filtered limits, the category $\Cond_{\kappa}(\Ccal)$ admits small filtered limits as well. Hence restriction along the constant functor $\Ccal \to \Pro(\Ccal)$ induces an equivalence 
\[ \Fun^\mathrm{lim}(\Pro(\Ccal),\Cond_\kappa(\Ccal)) \To[\simeq] \Fun(\Ccal,\Cond_\kappa(\Ccal)) \]
where the source denotes the full subcategory of all functors spanned by those which preserve small filtered limits \cite[A.8.1.6]{SAG}. Thus the constant functor $\Ccal \to \Cond_\kappa(\Ccal)$ yields the functor $\gamma_\kappa$ and the upper left commutative triangle. 
Analogously, we get a functor $\gamma^\omega \colon \Pro(\Ccal) \to \Cond^\omega(\Ccal)$ restricting to the constant functor $\Ccal \to \Cond^\omega(\Ccal)$. Since both $\gamma^\omega$ and $j^*\circ\gamma_\kappa$ restrict to the constant functor, we see that they are equivalent.

The commutativity of the lower left triangle follows if $\gamma \coloneqq \gamma_\kappa$ preserves all limits, as both compositions $\lim \circ \const$ are equivalent to the identity functor.

We are left to show that $\gamma$ preserves finite pullbacks, as then $\gamma$ preserves all limits by \cite[Prop. 4.4.2.7]{HTT}. For this, we remark that finite limits in $\Pro(\Ccal)$ can be computed pointwise in the sense of \cite[\S 2.1]{kst-i}. Thus, any diagram 
$$
\begin{tikzcd}
    & X\arrow[d,""]\\
    Y\arrow[r,""]& Z
\end{tikzcd}
$$
in $\Pro(\Ccal)$ can be lifted to an element $F\in \Pro(\Fun(K,\Ccal))$, where $K$ is the set representing the diagram above, such that its image under 
$$
    \Pro(\lim)\colon \Pro(\Fun(K,\Ccal))\To\Pro(\Ccal)
$$
is equivalent to $X\times_Z Y$. We also have a functor
$$
    \Cond(\lim)\colon \Cond_\kappa(\Fun(K,\Ccal))\To\Cond_\kappa(\Ccal)
$$
that is induced by the following construction. We have a limit preserving functor 
$$ \gamma^K\colon \Cond_\kappa(\Fun(K,\Ccal))\inj\Fun(\ProFin_{<\kappa},\Fun(K,\Ccal))\simeq \Fun(K,\Fun(\ProFin_{<\kappa},\Ccal)).
$$
Then $\Cond(\lim)$ is the composition of the above functor with the limit functor and sheafification. Now we obtain an essentially commutative square
\begin{equation}\tag{$\clubsuit$}   
    \begin{tikzcd}
        \Pro(\Fun(K,\Ccal)) \arrow[d,"\Pro(\lim)",swap]\arrow[r,"\gamma^{K}"]& \Cond_\kappa(\Fun(K,\Ccal))\arrow[d,"\Cond(\lim)"]\\
        \Pro(\Ccal)\arrow[r,"\gamma"]&\Cond_\kappa(\Ccal).
    \end{tikzcd}
\end{equation}
Note that limits and colimits in functor categories can be computed pointwise, thus we have an equivalence $\Cond_\kappa(\Fun(K,\Ccal))\simeq \Fun(K,\Cond_\kappa(\Ccal))$. Under this equivalence $\Cond(\lim)$ corresponds to the limit functor $\Fun(K,\Cond_{\kappa}(\Ccal))\rightarrow \Cond_\kappa(\Ccal)$. Therefore, commutativity of $(\clubsuit)$ shows that $\gamma$ preserves $K$-indexed limits, i.e. pullbacks. 
\end{proof}

\begin{remark} \label{continuous-comparison-lemma}
An informal argument for the comparison functor $\gamma$ to commute with cofiltered limits goes as follows:
Let $A \to \Pro(\Sp), \alpha \mapsto X_\alpha$ be a cofiltered diagram of pro-spectra $X_\alpha = \prolim_{i\in I_\alpha}X_{\alpha, i}$. We may assume that $A$ is cofinite and that the diagram $A\to\Pro(\Sp)$ is given in level-representation with cofiltered index set $I$ ($=I_\alpha$ for all $\alpha\in A$) so that $A\times I$ is cofiltered again and $\prolim_{A\times I}X_{\alpha,i} \simeq \lim_AX_\alpha$ is the limit \cite[\S 4]{isaksen-calculating}. Then
	\begin{align*}
	\gamma( \lim_{A} X_\alpha ) 
	\simeq \gamma( \prolim_{A\times I} X_{\alpha, i} ) 
	= \lim_{A\times I} \underline{X_{\alpha, i}} 
	\simeq   \lim_{A} \, \lim_{I} \underline{X_{\alpha, i}} 
	=  \lim_{A} \gamma( X_\alpha ).
	\end{align*}
\end{remark}

\begin{cor}
The comparison functor $\gamma$ preserves cofibre sequences and pushouts.
\end{cor}
\begin{proof}
This follows since both $\Pro(\Sp)$ and $\Cond(\Sp)$ are stable categories.
\end{proof}

Before we continue, let us note in the following that we also have a canonical $t$-structure on $\Pro(\Sp)$ obtained from the $t$-structure on $\Sp$. Moreover, the comparison functor $\gamma$ is left $t$-exact with respect to the $t$-structures on both sides.

\begin{lemma}
\label{lem-t-structure-pro}
There exists a t-structure on $\Pro(\Sp)$ with $\Pro(\Sp)_{\geq 0}$ given by the essential image of the functor $\Pro(\Sp_{\geq 0})\to\Pro(\Sp)$ and with $\Pro(\Sp)_{\leq 0}$ given by the full subcategory $\Pro(\Sp_{\leq 0}) \inj \Pro(\Sp)$.  

For $X=\prolim X_i\in\Pro(\Sp)$ and $n\in \ZZ$ we can compute the $n$-truncations $\tau_{\leq n}$ and $\tau_{\geq n}$ pointwise, in particular $\pi_n X\simeq \prolim \pi_n  X_i$.
\end{lemma}
\begin{proof}
First let us remark that $\Pro(\Sp)$ is a stable category \cite[Lem.~2.5]{kst-i}. The inclusion $\iota\colon\Sp_{\leq 0}\subseteq \Sp$ induces an adjunction 
$$
 \begin{tikzcd}
	\iota^{*}\colon \Pro(\Sp)\arrow[r,"",shift left = 0.3em]&\arrow[l,"",shift left = 0.3em]\Pro(\Sp_{\leq 0})\colon  \Pro(\iota)
\end{tikzcd}
$$
\cite[Ex.~A.8.1.8]{SAG}. The functor $\Pro(\iota)$ is fully faithful \cite[Prop.~A.8.1.9]{SAG}. 

We can apply the same argument to the truncation $\tau_{\geq 0}\colon \Sp\rightarrow \Sp_{\geq 0}$ and get an adjunction
$$
 \begin{tikzcd}
	\tau_{\geq 0}^{*}\colon \Pro(\Sp_{\geq 0})\arrow[r,"",shift left = 0.3em]&\arrow[l,"",shift left = 0.3em]\Pro(\Sp)\colon \Pro(\tau_{\geq 0}).
\end{tikzcd}
$$
Let us now define $\Pro(\Sp)_{\geq 0}$ as the essential image of $\tau_{\geq 0}^{*}$. We claim that the tuple 
$$
    (\Pro(\Sp)_{\geq 0}, \Pro(\Sp_{\leq 0}))
$$
defines a $t$-structure on $\Pro(\Sp)$.

Indeed, finite limits and colimits can be computed levelwise in $\Pro(\Sp)$ \cite[Lem.~2.1]{kst-i}. So, checking the axioms of a $t$-structure reduces to the axioms for $(\Sp_{\geq 0},\Sp_{\leq 0})$, which hold by design.

For the homotopy groups note that $\tau_{\leq n}$ can be computed pointwise by construction. For $\tau_{\geq n}$ we note that for any $X\in \Pro(\Sp)$, we have a fibre sequence
\begin{equation*}
    \tau_{\geq n} X\to X\to \tau_{\leq n-1} X
\end{equation*}
by the proof of \cite[Prop.~1.2.1.5]{HA}. As finite limits in $\Pro(\Sp)$ can be computed pointwise, we see that also $\tau_{\geq n}$ can be computed pointwise.
\end{proof}

\begin{prop}
The canonical functor $\gamma\colon \Pro(\Sp)\to\Cond(\Sp)$ is left t-exact.
\end{prop}
\begin{proof}
Since the inclusion functor $\Cond(\Sp_{\leq 0}) \inj \Cond(\Sp)$ preserves limits, the composition $\Pro(\Sp_{\leq 0}) \to \Pro(\Sp) \to[\gamma] \Cond(\Sp)$ factors over the inclusion of $\Cond(\Sp_{\leq 0}) \inj \Cond(\Sp)$.
Thus $\gamma$ is left t-exact. 
\end{proof}

Next we want to show that the comparison functor is conservative. For this purpose, we start with some preparations.

\begin{lemma} \label{lem:condensed-milnor-sequence}
Let $\bigl( (X_i)_i, (p_i\colon X_i\to X_{i-1})_i \bigr)$ be a tower of condensed spectra and $n\in\ZZ$. Then there is an exact sequence
\[
0 \To {\lim_i}^1\pi_{n+1}(X_i) \To \pi_n(\lim_iX_i) \To \lim_i\pi_n(X_i) \To 0
\]
of condensed abelian groups, the so-called Milnor sequence.
\end{lemma}
\begin{proof}
This is a standard proof.\footnote{%
    Cf.\@~\url{https://ncatlab.org/nlab/show/lim^1+and+Milnor+sequences}.
}
The sequential limit $\lim_iX_i$ fits into a fibre sequence
\[
\lim_i X_i \To \prod_iX_i \To[\id-p_*] \prod_iX_i
\]
where $p_*$ is the map induced by the maps $(p_i\colon X_i\to X_{i-1})_i$.
Consider the associated long exact sequence of homotopy groups
\[
\ldots \To \prod_i\pi_{n+1}(X_i) \To[\id-p_*] \prod_i\pi_{n+1}(X_i) \To[\alpha] \pi_n(\lim_iX_i) \To[\beta] \prod_i\pi_n(X_i) \To[\id-p_*] \prod_i\pi_n(X_i) \To \ldots,
\]
where we have used the fact that countable products are t-exact in $\Cond^\omega(\Sp)$ by Lemma \ref{Lem:count-prod-t-exact}.
We have that
\[
\im(\beta) = \ker\bigl( \prod_i\pi_n(X_i) \To[\id-p_*] \prod_i\pi_n(X_i) \bigr) = \lim_i\pi_n(X_i)
\]
and
\[
\ker(\alpha) = \coker\bigl( \prod_i\pi_{n+1}(X_i) \To[\id-p_*] \prod_i\pi_{n+1}(X_i) \bigr) = {\lim_i}^1\pi_{n+1}(X_i)
\]
which shows the claim.
\end{proof}

The following result and its proof was communicated to us by Peter Scholze and we do not claim originality. 

A pro-spectrum $X = \prolim_i X_i \in \Pro^\omega(\Sp)$ is called \emph{eventually constant} if there exists $N \in \NN$ such that $\tau_{\leq N} X$ is equivalent to a constant pro-spectrum in $\Pro^\omega(\Sp)$ (or equivalently, in $\Pro^\omega(\Sp^+)$); note that $\tau_{\leq N} X = \prolim_i \tau_{\leq N} X_i$ by Lemma~\ref{lem-t-structure-pro}.
For example, bounded below pro-spectra are eventually constant, and for any ring $R$ and any ideal $I \subset R$, the pro-spectrum $\prolim_{n \in \NN} \K(R/I^n)$ is eventually constant by nil-invariance of non-positive $\K$-groups.

\begin{thm}[Clausen--Scholze] \label{thm:comparison-conservative}
The comparison functor 
\[ \gamma^\omega = j^* \circ \gamma_\kappa \,\colon\, \Pro^\omega(\Sp^+) \to \Cond^\omega(\Sp) \]
is conservative on eventually constant objects.
In particular, also the comparison functor $\gamma_\kappa$ is conservative on eventually constant objects.%
    \footnote{The authors thank Georg Tamme for suggesting the generalisation from \emph{bounded below objects} to \emph{eventually constant objects}.}
\end{thm}

\begin{proof}
Since the categories in question are stable (Lemma~\ref{Pro-omega-stable--lemma} and Lemma~\ref{condensed-objects-stable--lem}), it suffices to check that the functor $\gamma$ reflects zero objects. Let $\ldots\to X_2 \to X_1 \to X_0$ be a
system of spectra such that $\lim_i\underline{X_i} = 0$ as condensed spectra (where $\underline{X_i}$ denotes the associated constant condensed spectrum, and the
limit is taken in condensed spectra). 
We have to see that the pro-spectrum $X = \prolim_i X_i$ is zero.
If $X$ is eventually constant, there exists an $N \in \NN$ be such that $\tau_{\leq N} X = \prolim_i \tau_{\leq N} X_i$ is constant; we may assume that $\tau_{\leq N} X = \prolim_i \tau_{\leq N} X_i \simeq \tau_{\leq N} X_0$. It suffices to show that $\tau_{\leq N} X_0$ is zero as a spectrum and that the homotopy pro-groups $\pi_n(X)$ are zero for $n > N$ \cite[Lem.~2.8]{kst-i}.
The Milnor sequences (Lemma~\ref{lem:condensed-milnor-sequence})
\[
0 \To {\lim_i}^1\underline{\pi_{n+1}(X_i)} \To \pi_n(\lim_i \underline{X_i}) \To \lim_i\underline{\pi_n(X_i)} \To 0, \quad n \in \NN
\]
in $\Cond(\Ab)$ yield that $\lim_i \underline{\pi_n(X_i)} = \lim_i^1 \underline{\pi_n(X_i)} = 0$ in this category. Thus, the vanishing of $\pi_n(X) = \prolim_i \pi_n(X_i)$ for all $n \in \NN$ follows from the Lemma~\ref{pro-vanishing-and-Mittag-Leffler--lemma} below.
Moreover, for $n \leq N$, by $\lim_i\underline{\pi_n(X_i)} = 0$ and constancy we see that $\pi_n(X_0) = 0$; in particular $\tau_{\leq N} X_0$ is the zero spectrum by Whitehead's Theorem \cite[Prop.~1.4.3.6 (3)]{HA}.
\end{proof}

\begin{remark}
The comparison functor $\gamma\colon\Pro^\omega(\Ab)\to\Cond_\kappa(\Ab)$ is not conservative. Consider the tower $(p^i\ZZ)_i$ whose structure maps are the inclusions (for a prime number $p$). Then $\prolim_ip^i\ZZ$ is not zero in $\Pro(\Ab)$ as for every $j\geq i$ the map $p^j\ZZ \to p^i\ZZ$ is not zero. But $\gamma(\prolim_ip^i\ZZ) = \lim_i\underline{p^i\ZZ} = \underline{\lim_ip^i\ZZ} =0$ (Remark~\ref{lim-and-lim^1-in-Ab-and-CondAb--remark}).
\end{remark}

\begin{remark} \label{lim-and-lim^1-in-Ab-and-CondAb--remark}
The faithful functor $\Top\Ab \to \Cond_\kappa(\Ab), M\mapsto \underline{M} \coloneqq C(-,M)$, commutes with limits and hence admits a left adjoint $\L \colon \Cond_\kappa(\Ab)\to\Top\Ab$. Thus for a tower $(M_i)_i$ of topological abelian groups with limit $M\coloneqq\lim_iM_i$ in $\Top\Ab$, the canonical morphism $\underline{M} \to \lim_i\underline{M_i}$ of associated condensed abelian groups is an isomorphism.
In $\Cond_\kappa(\Ab)$ we have an exact sequence
\[ \tag{$\clubsuit$}
0 \To \lim_i\underline{M_i} \To \prod_i\underline{M_i} \To[\delta] \prod_i\underline{M_i} \To {\lim_i}^1\underline{M_i} \To 0
\]
where $\delta$ is the difference of the identity and the structure map of the tower.
As the canonical morphism $\underline{\prod_iM_i}\To\prod_i\underline{M_i}$ is an isomorphism and since for every $N\in\Top\Ab$ the counit $\L(\underline{N})\to N$ is an isomorphism, the sequence $\L(\clubsuit)$ identifies with the exact sequence
\[
0 \To \lim_iM_i \To \prod_iM_i \To[\delta] \prod_iM_i \To {\lim_i}^1M_i \To 0
\]
of topological abelian groups, hence $\L(\lim_i^1\underline{M_i}) = \lim_i^1M_i$.
\end{remark}

\begin{lemma} \label{pro-vanishing-and-Mittag-Leffler--lemma}
For a tower $(M_i)_{i \in \NN}$ of abelian groups the following are equivalent:
\begin{enumerate}
    \item The tower is Mittag-Leffler and $\lim_iM_i=0$.
    \item The tower is pro-zero, i.e.\@ $\prolim_iM_i=0$ in $\Pro(\Ab)$.
    \item It holds $\lim_i \underline{M_i} =0= \lim^1_i \underline{M_i}$ in $\Cond_\kappa(\Ab)$. 
    \item It holds $\lim_i \underline{M_i} =0= \lim^1_i \underline{M_i}$ in $\Cond^\omega(\Ab)$.
\end{enumerate}
\end{lemma}
\begin{proof}
First we show (1)$\Rightarrow$(2). The tower $(M_i)_i$ being Mittag-Leffler means that for every $i\geq 0$ the tower $(\im(M_j\to M_i))_{j\geq i}$ stabilises. Let $B_i\subset M_i$ be this stable image, i.e.\@ $B_i=\bigcap_{j\geq i}\im(M_j\to M_i)$. 
Then $(B_i)_i$ is a subtower of $(M_i)_i$ with surjective transition maps and $\lim_iB_i = \lim_iM_i = 0$, hence $B_i = 0$. Thus for every $i$, by stability there exists a $j\geq i$ such that the map $M_j\to M_i$ is zero, hence the pro-system $\prolim_iM_i$ is zero. The direction (2)$\Rightarrow$(1) works analogously.

Now we assume (3). For any profinite set $S$ we have $\Cont(S,M_i)
= \Cont(S,\ZZ) \otimes_\ZZ M_i$ and $\Cont(S,\ZZ)$ is some large free abelian group \cite[Thm. 5.4]{condensed}, and so has $\bigoplus_\NN$ $\ZZ$ as a direct factor if $S$ is large enough. 
Thus there exists an extremally disconnected set $S$ such that 
\[
{\lim_i}^1 \bigoplus_\NN M_i \subset {\lim_i}^1 \Cont(S,M_i) = {\lim_i}^1 \underline{M_i}(S) \overset{(\clubsuit)}{=} \bigl( {\lim_i}^1\underline{M_i} \bigr)(S)=0
\]
where $(\clubsuit)$ holds as $\id_S$ is cofinal among all covers of $S$ (Remark~\ref{extremally-disconnected--rem}). The condition that $\lim_i^1\bigoplus_\NN M_i=0$ in $\Ab$ is equivalent to the tower $(M_i)_i$ being Mittag-Leffler \cite[Cor. 6]{emmanouil}, hence we have (1).

Assuming (1), for any extremally disconnected set $S$, we have $\lim_i M_i = 0$ as a topological group, hence
$(\lim_i \underline{M_i})(S) = \lim_i \Cont(S, M_i) = \Cont(S, \lim_i M_i) = 0$;
moreover, we have
\[
({\lim_i}^1\underline{M_i})(S) = {\lim_i}^1\underline{M_i}(S) = {\lim_i}^1 \Cont(S,M_i) \cong {\lim_i}^1\bigoplus_J M_i
\]
for some set $J$ and the last term vanishes as the tower is Mittag-Leffler. Hence $\lim_i^1\underline{M_i}$ vanishes in $\Cond_\kappa(\Ab)$.

The implication (3)$\Rightarrow$(4) follows since $j^*$ preserves limits and colimits.

It remains to show (4)$\Rightarrow$(1). For this let $S\coloneqq\NN\cup\{\infty\}\in\ProFin^\omega$. Then $\ZZ[S]$ is a projective object in $\Cond^\omega(\Ab)$ \cite[Thm.~2.3.3]{camargo2026notesolid}\footnote{%
    See also Lecture 3 by Scholze in the lectures on Analytic Stacks, \url{https://youtu.be/me1KNo3WJHE?si=iSCb0CAYyCRvIn1R&t=2595}.
} so that the equality $(\clubsuit)$ above holds for this particular $S$ in $\Cond^\omega(\Ab)$.
\end{proof}

\begin{example}
If a tower $(M_i)_i$ of abelian groups satisfies $\lim_iM_i=0=\lim_i^1M_i$, then it needs not to be pro-zero. For instance, consider the tower $(p^i\ZZ_p)_i$ for some prime number~$p$.
\end{example}

\subsection{Lifting functors to condensed coefficients}
\label{subsec:lifting-functors-to-condensed}
We end this section by lifting functors to enriched functors with condensed coefficients by using that the category of
analytic adic spaces is tensored over the category of profinite sets (Lemma~\ref{lem:XotimesSsheafy}).

\begin{construction}
    \label{construction:condensify}
    Consider the following containments of categories inside $\ProFin$:
    \[\ProFin^\omega \overset{j}{\hookrightarrow} \ProFin_{< \kappa} \overset{\iota_\kappa}{\hookleftarrow} \EDS_{< \kappa}.\]
    The cardinality restriction $<\kappa$ may be removed for the construction below.
    
    Let $\Ccal$ be an essentially small category which admits an action by $\ProFin_{< \kappa}$ and let $\Vcal$ be a category.
    There are natural restriction functors
    \[\PSh(\ProFin^\omega, \Vcal) \xleftarrow{j^*} \PSh(\ProFin_{< \kappa}, \Vcal) \xrightarrow{\iota_\kappa^*} \PSh(\EDS_{< \kappa}, \Vcal).\]
    Both functors preserve sheaves and hypersheaves by Remark~\ref{comparison-light-kappa-condensed}. The functor $\iota_\kappa^*$ induces an equivalence $\Shhyp(\ProFin_{< \kappa}, \Vcal) \simeq \Shhyp(\EDS_{< \kappa}, \Vcal) \simeq \Fun^\times(\EDS_{< \kappa}\op, \Vcal)$ by Remark~\ref{extremally-disconnected--rem}, whose quasi-inverse is given by unfolding the hypersheaves.
    The action functor $\ProFin_{< \kappa} \times \Ccal \to \Ccal$ induces functors
    \[\begin{tikzcd}[column sep=1.5em, every cell/.append style={font=\small}]
        & \arrow[dl] \PSh(\Ccal, \Vcal) \arrow[d] \arrow[dr] \\
        \PSh(\ProFin^\omega \times \Ccal, \Vcal) \arrow[d, phantom, "\simeq", sloped] & \arrow[l, "j^*"'] \PSh(\ProFin_{< \kappa} \times \Ccal, \Vcal) \arrow[r, "\iota_\kappa^*"] \arrow[d, phantom, "\simeq", sloped]  & \PSh(\EDS_{< \kappa} \times \Ccal, \Vcal) \arrow[d, phantom, "\simeq", sloped]  \\
        \PSh(\Ccal, \PSh(\ProFin^\omega, \Vcal))  \arrow[d, phantom, "\simeq", sloped]& \arrow[l, "j^*"'] \PSh(\Ccal, \PSh(\ProFin_{< \kappa}, \Vcal)) \arrow[r, "\iota_\kappa^*"] \arrow[d, phantom, "\simeq", sloped] & \PSh(\Ccal, \PSh(\EDS_{< \kappa}, \Vcal)) \arrow[d, phantom, "\simeq", sloped] \\
        \PSh(\ProFin^\omega, \PSh(\Ccal, \Vcal)) & \arrow[l, "j^*"'] \PSh(\ProFin_{< \kappa}, \PSh(\Ccal, \Vcal)) \arrow[r, "\iota_\kappa^*"] & \PSh(\EDS_{< \kappa}, \PSh(\Ccal, \Vcal)),
    \end{tikzcd}\]
    where $F \in \PSh(\Ccal, \Vcal)$ is sent to $F^\cond \in \PSh(\Ccal, \PSh(\ProFin_{< \kappa}, \Vcal))$ in the middle, whose value at $X \in \Ccal$ is given by
    \begin{equation}
        F^\cond(X) \colon S \mapsto F(S \otimes X).
    \end{equation}
    We will also consider its restriction $j^*F^\cond$ to light profinite sets and its restriction $\iota_\kappa^* F^\cond$ to $\kappa$-extremally totally disconnected sets. Both will be simply denoted by $F^\cond$ whenever it is clear from the context.
    
    Now, let $\tau$ be a Grothendieck topology on $\Ccal$ which is stable under $S \otimes (-)$ for $S \in \ProFin_{< \kappa}$. Assume also that the natural morphism $(S_1 \otimes -) \sqcup (S_2 \otimes -) \to (S_1 \sqcup S_2) \otimes -$ is an equivalence for any $S_1, S_2 \in \ProFin_{< \kappa}$. Then, by imposing $\tau$-descent, the above (commutative) diagram restricts to the following:
    \[\begin{tikzcd}[column sep=1.5em, every cell/.append style={font=\small}]
        & \arrow[dl] \Sh_\tau(\Ccal, \Vcal) \arrow[d] \arrow[dr] \\
        \Sh_\tau(\Ccal, \Fun^\times(\ProFin^{\omega, \opp}, \Vcal)) & \arrow[l, "j^*"'] \Sh_\tau(\Ccal, \Fun^\times(\ProFin_{< \kappa}\op, \Vcal)) \arrow[r, "\iota_\kappa^*"] & \Sh_\tau(\Ccal, \Fun^\times(\EDS_{< \kappa}\op, \Vcal)) \\
        \Sh_\tau(\Ccal, \Cond^\omega(\Vcal)) \arrow[u, phantom, "\subset", sloped] & \arrow[l, "j^*"'] \Sh_\tau(\Ccal, \Cond_\kappa(\Vcal)) \arrow[r, "\iota_\kappa^*"] \arrow[u, phantom, "\subset", sloped] & \Sh_\tau(\Ccal, \Cond_\kappa(\Vcal)). \arrow[u, phantom, "\simeq"', sloped]
    \end{tikzcd}\]
    In general, we have $\iota_\kappa^* F^\cond \in \Sh_\tau(\Ccal, \Cond_\kappa(\Vcal))$ lying in the bottom right; however, $F^\cond$ does not necessarily satisfy (hyper)descent for $S \in \ProFin^\omega$. For example, if $\Ccal = \AnAdic$ and $\tau = \et$, for general map $S' \to S$ in $\ProFin^\omega$, the natural map $F(S \otimes X) \to \lim F(S'^{\times/S}) \otimes X$ is in general not an isomorphism as $S' \otimes X \to S \otimes X$ is a pro-finite étale cover which is not necessarily an étale cover.
\end{construction}

The following proposition is a consequence of the above construction.

\begin{prop} 
\label{prop:condensed-enrichment-of-sheaves}
    Let $\Ccal \subset \AnAdic$ be a full subcategory which is closed under the operation $S\otimes (-)$ on $\PreAdic$ for every $S\in\ProFin$.
    Let $\tau$ be a Grothendieck topology on $\Ccal$ which is stable under $S \otimes (-)$ for profinite sets $S$.
    Let $F\colon \Ccal\op \to \Vcal$ be a sheaf with values in a stable category $\Vcal$.
    Then there exists an essentially unique sheaf
    \[ F^\cond \colon \Ccal\op \To \Cond_\kappa(\Vcal) \]
    such that for every $X\in\Ccal$ and every $S\in\EDS_\kappa$ we have an equivalence
    \[ F^\cond(X)(S) \simeq F(S\otimes X) \]
    functorial in $X$ and $S$.
\end{prop}

\begin{proof}
    The desired functor $F^\cond$ is essentially uniquely determined by the requested property. The construction of $F^\cond$ follows immediately from Construction~\ref{construction:condensify}, applied to any Grothendieck topology $\tau$ on $\Ccal$ (whose covers are preserved by $S \otimes (-)$ since the latter operation amounts to the base change by $\underline{S}$).
\end{proof}

\section{K-theory of adic spaces}
\label{sec:k-theory}

In this section we recall the definitions and basic properties of the different K-theories for rigid spaces. These theories have been developed by Morrow~\cite{Morrow}, Kerz--Saito--Tamme  \cite{kst-i,kst-ii}, Efimov \cite{efimov-k-theory,efimov-inverse-limits}, and Andreychev \cite{andreychev-thesis}.
Moreover, we establish the new result that the continuous K-theory $\Kcontu$ satisfies Nisnevich descent (Theorem~\ref{Thm:Nisnevich-descent-Kcont}).

\vspace{6pt}\noindent\textbf{Notation.}
In this entire section, let $(R,R^+)$ be a Tate pair.
We set $\base=\Spa(R,R^+)$ and choose a pseudo-uniformiser $\varpi\in R$ so that $R=R_0[\varpi^{-1}]$.

\vspace{6pt} Some of the results crucially use the following assumption, cf.\@ \cite[\S 3.2]{kst-i}.
\begin{enumerate}[label=$(\dagger)_A$]
    \item \label{dagger-condition} There exists a ring of definition $A_0 \subset A$ which is noetherian and there exists a proper morphism $p \colon X \to \Spec(A_0)$ such that $X$ is regular and such that $p$ is an isomorphism over $\Spec(A)$.
\end{enumerate}
Note that the noetherian assumption can be relaxed in some statements, see Section~\ref{sec:non-noetherian-case} below.
In our setting, we have to make the following assumption for a Grothendieck topology $\tau$ on $\Adicsm_R$, mainly for $\tau$ being the analytic topology or the étale topology.

\begin{enumerate}[label=$(\spadesuit)_R^\tau$]
    \item
    \label{spadesuit-condition}
    \customlabel{spadesuit-condition-analytic}{$(\spadesuit)_R^\an$}
    \customlabel{spadesuit-condition-analytic-k}{$(\spadesuit)_k^\an$}
    \customlabel{spadesuit-condition-etale}{$(\spadesuit)_R^\et$}
    \customlabel{spadesuit-condition-etale-k}{$(\spadesuit)_k^\et$}
    Every object of $\Adicsm_R$ is $\tau$-locally of the form $\Spa(A,A^+)$, where $A$ satisfies \ref{dagger-condition} (over some choice of $R_0$). 
\end{enumerate}

\begin{rem}
\label{rem:key-assumption}
\begin{enumerate}
    \item The assumption \ref{spadesuit-condition-analytic} is satisfied if $R$ admits a quasi-excellent ring of definition $R_0$ of characteristic zero \cite{temkin-desing-char0}. For instance, any discrete valuation ring $R_0$ containing $\QQ$ meets this assumption \cite[07QW]{stacks-project}; in this case $R = \Frac(R_0)$ and one can take $R^+=R_0$.
    
    \item Let $A_0$ be a regular $k^\circ$-algebra of finite type over a discrete valuation ring $k^\circ$.
    Then its $\pi$-adic completion $\Ahat_0$ is a regular ring. Indeed, since $k$ is discretely valued and complete, the ring $k^\circ$ is (a complete noetherian local ring, hence) a G-ring; as a finite type algebra over it, $A_0$ is also a G-ring \cite[\S 33, Thm.~77]{matsumura1980commutativealgebra}; as a result, passing to its $\pi$-adic completion preserves the regularity \cite[\S 33, Thm.~79]{matsumura1980commutativealgebra}. Since $A_0$ is regular by assumption, so is $\Ahat_0$.
    \item Let $k$ be a discretely valued nonarchimedean field. Temkin's altered local uniformisation result \cite{temkin2017altered} and the last remark (2) imply that the assumption \ref{spadesuit-condition-etale-k} always holds. Indeed, one can achieve even that $A$ has a model $A_0$ which is the $\pi$-adic completion of a semistable $k^\circ$-algebra. 
    For a general nonarchimedean field $k$, Temkin's result still holds, but we need to modify the dagger condition that appears in \ref{spadesuit-condition-etale-k} in order that the latter condition holds; for details we refer to a note by the third named author \cite[1.6.4]{zhou-dagger-condition}.
\end{enumerate}
\end{rem}

\subsection{Continuous K-theory}
\label{subsec:continuous-k-theory}
The definition of continuous K-theory is due to Morrow \cite{Morrow} and the theory has been elaborated by Kerz-Saito-Tamme \cite{kst-i}. The first named author made some contributions about negative continuous K-groups \cite{dahli-thesis-paper}.

\begin{defi}
\label{def:continuous-k-theory}
Let $A_0$ be $\pi$-adic ring for some $\pi\in A_0$.  Then we set $\Kcontu(A_0)\coloneqq \prolim \K(A_0/\pi^{n})$ which is an object in $\Pro^\omega(\Sp)$. For a Tate ring $A$ with ring of definition $A_0$, we can define $\Kcontu(A)$ as the pushout of the diagram
$$ \begin{tikzcd}
		\K(A_0) \arrow[r,""]\arrow[d,""]& \K(A)\arrow[d,""]\\
		\Kcontu(A_0)\arrow[r,""]&\Kcontu(A),
\end{tikzcd} $$
within the category $\Pro^\omega(\Sp)$. When passing to the localisation $\Pro^\omega(\Sp^+)$, the object $\Kcontu(A)$ does not depend on the choice of $A_0$ \cite[Prop.~5.4]{kst-i}.
By nil-invariance of non-positive $\K$-groups, the pro-spectrum $\Kcontu(A_0)$ is eventually constant, hence $\Kcontu(A)$ is also eventually constant.


Eventually we will see that $\Kcontu$ extends uniquely to a sheaf $\Kcontu$ on the category $\AnAdic$ of analytic adic spaces (Theorem~\ref{Thm:Nisnevich-descent-Kcont}).
\end{defi}

This definition of continuous K-theory seems to be ad-hoc, but it shall be equivalent to the more conceptual approach via applying Efimov's extension of categorical K-theory \cite{efimov-k-theory} to the dualisable category of nuclear modules à la Clausen--Scholze \cite[Def.~13.10]{analytic}. For the spectrum-valued functors, this has been explored in Andreychev's thesis \cite{andreychev-thesis}. Let us introduce the following notion.

\begin{defi}[Nuclear K-theory]
\label{Def:nuclear-K-theory}
    Let $\Knucu \colon \AnAdic\op\to\Cond_\kappa(\Sp)$ be the sheaf associated with the sheaf $\Knuc \coloneqq \K\circ\Nuc \colon \AnAdic\op \to \Sp$ as in Proposition~\ref{prop:condensed-enrichment-of-sheaves}.
    
    We define a light version $\Knucu \colon \AnAdic\op \to \PSh(\ProFin^\omega, \Sp)$ (by abuse of notation of $\Knucu$) as the functor taking $X \in \AnAdic\op$ to the presheaf $S \mapsto \Knuc(S \otimes X)$ on light profinite sets.
\end{defi}

As we shall see in Proposition~\ref{prop:KnucHyperdescentProfinite}, the presheaf $\Knuc(- \otimes X)$ satisfies hyperdescent on profinite sets, hence the light version of $\Knucu(X)$ coincides with the restriction of the $\kappa$-condensed spectrum $\Knucu(X)$ to light profinite sets.

\begin{rem}
This nuclear K-theory is more generally defined for any analytic animated ring. In particular, for any condensed animated ring $R$, we may impose the induced analytic ring structure by considering the (derived) category $\Nuc(R)$ of nuclear modules over the analytic animated ring $(R,\ZZ)$, which is a dualisable category, and define $\Knuc(R) \coloneqq \K(\Nuc(R, \ZZ))$. This construction will be applied mainly for an adically complete animated ring $R$. This is consistent with $\Nuc(A) \coloneqq \Nuc(A, A^+)$; indeed, the latter category is independent of the choice of $A^+$ \cite[Korollar~3.18]{andreychev-thesis}.
\end{rem}

Recall that given a (commutative, unital) ring $\Lambda$, an ideal $I \subset \Lambda$ is called \emph{weakly proregular} if for some (resp. for any) finite set of generators $a_1, \dots, a_m \in I$, the pro-system
\[(H_i(\Kos(A; a_1^n, \dots, a_m^n))_{n \in \NN}\]
of abelian groups is essentially zero for any $i > 0$ \cite[\S 3]{yekutieli2021proregular}.

\begin{example}[{\cite[Prop.~5.6]{yekutieli2021proregular}}]
    \label{example:proregularOneElement}
    In the case of $m = 1$, the weak proregularity of a principal ideal $(a) \subset A$ is equivalent to the boundedness of the $a^\infty$-torsion of $A$.
\end{example}

\begin{thm}[Efimov Continuity Theorem {\cite[Cor.~7.10]{efimov-inverse-limits}}]
\label{thm:efimov-continuity-proregular}
    Let $\Lambda$ be a ring and $I \subset \Lambda$ be a weakly proregular ideal. Let $\Lambda^\wedge_I$ denote the derived $I$-adic completion of $\Lambda$. Then the canonical map
    \begin{equation}
        \label{eq:efimov-continuity-proregular}
        \Knuc(\Lambda^\wedge_I) \To[\simeq] \lim_n \K(\Lambda/I^n)
    \end{equation} 
    is an equivalence of spectra.
\end{thm}
\begin{proof}
    In the statement of \cite[Cor.~7.10]{efimov-inverse-limits} the ring is assumed to be noetherian, whereas in \cite[Satz~5.8]{andreychev-thesis} the wording is the same as here, but without proof. For the convenience of the reader, we briefly give an argument.
    
    Let $a_1, \dots, a_m$ be a finite collection of generators of the ideal $I \subset A$. By \protect{\cite[Cor.~7.10]{efimov-inverse-limits}}, the canonical map 
    \[\Knuc(\Lambda^\wedge_I) \To[\simeq] \lim_n \K(\Kos(\Lambda; a_1^n, \dots, a_m^n)\]
    is an equivalence of spectra.
    It has been shown also in op.\@ cit., Lemma~5.26, that if $\Lambda$ is a noetherian ring, then the natural map
    \[\varepsilon \colon \bigl(\Kos(\Lambda; a_1^n, \dots, a_m^n)\bigr)_{n \in \NN} \to (\Lambda/I^n)_{n \in \NN}\]
    of pro-systems of connective $\EE_1$-algebras in $\Dcal(\Lambda)$ is an equivalence; in particular, we obtain a canonical equivalence
    \begin{equation}
        \label{eq:isoKoszulIdealistic}
        \lim_n \K(\Kos(\Lambda; a_1^n, \dots, a_m^n)) \To[\simeq] \lim_n \K(\Lambda/I^n).
    \end{equation}
    More generally, in the case where $\Lambda$ is not necessarily noetherian and where the ideal $I = (a_1, \dots, a_m)$ is weakly proregular in $\Lambda$, the above natural map $\varepsilon$ is still an equivalence in $\Pro(\Algrm_{\EE_1}(\Dcal(\Lambda)))$. Indeed, the forgetful functor
    \[\Pro(\Algrm^\cn_{\EE_1}(\Dcal(\Lambda)) \to \Pro(\Dcal(\Lambda))\]
    is conservative on bounded objects by \cite[Lem.~2.28]{landtamme2018pullback}, and for the (weak) equivalence in the latter between bounded objects, it suffices to verify on homotopy pro-groups. Now, $\Kos(\Lambda; a_1^n, \dots, a_m^n)$ is connective and also bounded above by $m$. For any $i \in \NN$, the map $\pi_i(\varepsilon) \colon \bigl(H^i(\Kos(\Lambda; a_1^n, \dots, a_m^n))\bigr)_{n \in \NN} \to \bigl(H^i(\Lambda/I^n[0])\bigr)_{n \in \NN}$ is an equivalence in $\Pro(\Mod_\Lambda)$ by the weak proregularity hypothesis. Therefore, we deduce that $\varepsilon$ is an equivalence in $\Pro(\Dcal(\Lambda))$, whence also in $\Pro(\Algrm_{\EE_1}(\Dcal(\Lambda)))$; in particular, the equivalence (\ref{eq:isoKoszulIdealistic}) still holds.
\end{proof}

The following Lemma~\ref{lem:CartesianDiagramNuc} and its Corollary~\ref{cor:andreychev-Kcont=KNuc-proregular} are results from Andreychev's thesis. Since the reasoning in loc.\@ cit.\@ is quite terse, we present here extended arguments for the convenience of the reader.
Recall that $\Catdual$ is the category of dualisable stable categories and strongly continuous functors.
A \emph{localisation sequence} in $\Catdual$ is a cofibre sequence $\Acal \to \Bcal \to \Ccal$ with $\Acal \to \Bcal$ being fully faithful (see Corollary~\ref{cor:fibrecofibreinPrLL} for equivalent characterisations).

\begin{lem}[{\cite[p.~28]{andreychev-thesis}}]
    \label{lem:CartesianDiagramNuc}
    Let $A$ be a Tate ring and let $A_0 \subset A$ be a ring of definition containing a pseudo-uniformiser $\varpi$. There is a commutative diagram in $\Catdual$ with rows being localisation sequences and fibre-cofibre sequences:
    \begin{equation}
    \begin{tikzcd}
        \label{eq:CartesianDiagramNuc}
        \Tor_{A_0}(\varpi^\infty) \arrow[r] \arrow[d, "\simeq"] & \Dcal(A_0) \arrow[r, "L"] \arrow[d] & \Dcal(A) \arrow[d] \\
        \Tor^\nuc_{A_0}(\varpi^\infty) \arrow[r] & \Nuc(A_0) \arrow[r, "L'"] & \Nuc(A).
    \end{tikzcd}
    \end{equation}
\end{lem}

\begin{proof}
    All the categories involved are stable.

    All the arrows in the right square are strongly continuous.
    First, observe that $L$ and $L'$ are the base change functors along $A_0 \to A$ and $\underline{A_0} \to \underline{A}$ respectively, hence they preserve colimits. Since $A = A_0[\varpi^{-1}]$ is an idempotent $A_0$-algebra, the right adjoints of $L$ and $L'$ are fully faithful forgetful functors which preserve colimits.
    Then, let us explain the vertical arrows. They are induced by the (fully faithful) condensification functor \cite[Def.~5.8, Thm.~5.9]{andreychev2021descent}, which is left adjoint to the evaluation functor $M \mapsto M(*)$, which in turn preserves colimits.\footnote{%
        The full subcategory $\Dcal(\Acal, \Mcal) \subset \Dcal(\Acal)$ is stable under colimits for any analytic ring $(\Acal, \Mcal)$, and colimits in $\Dcal(\Acal)$ commutes with evaluating at any extremally disconnected set $S$.
    }
    The right square commutes by \cite[Lem.~5.10]{andreychev2021descent}.

    Consider the kernels $\Tor_{A_0}(\varpi^\infty) \simeq \ker(L)$ and $\Tor^\nuc_{A_0}(\varpi^\infty) \simeq \ker(L')$, which are taken in the category $\widehat{\Cat}_\infty$ of all $\infty$-categories. They are also the kernels taken in the category $\PrLst$ of presentable stable categories with continuous functors, since the inclusions $\PrLst \subset \PrL \subset \widehat{\Cat}_\infty$ create limits (see \cite[Prop.~4.8.2.18]{HA} and \cite[Prop.~5.5.3.13]{HTT}).
    The map between the kernels $\Tor_{A_0}(\varpi^\infty) \to \Tor^\nuc_{A_0}(\varpi^\infty)$ is an equivalence by \cite[Satz 4.11]{andreychev-thesis}. 
    
    Since all the categories involved in the diagram are stable, and $L$, $L'$ have fully faithful right adjoints which preserve colimits, the condition (ii) in Corollary~\ref{cor:fibrecofibreinPrLL} is satisfied in $\PrLLst$,\footnote{%
        One can check explicitly that the inclusion functors of both kernels are \emph{strongly} continuous.
        For the upper row, the right adjoint of the (colimit-preserving) inclusion $\Tor_{A_0}(\varpi^\infty) \subset \Dcal(A_0)$ is given by $ \fib(A_0 \to A) \otimes_{A_0}^\LL (-)$ \cite[\href{https://stacks.math.columbia.edu/tag/0952}{Section 0952}]{stacks-project}, which clearly commutes with colimits.
        For the bottom row, since we already know that $\Tor_{A_0}(\varpi^\infty) \To[\simeq] \Tor^\nuc_{A_0}(\varpi^\infty)$, the inclusion $\Tor^\nuc_{A_0}(\varpi^\infty) \subset \Nuc(A_0)$ is strongly continuous by composition with the left vertical arrow.
    }
    hence it is also satisfied in $\Catdual$ by Remark~\ref{rem:fibcofib} (since $\Dcal(A_0)$ and $\Nuc(A_0)$ are dualisable).
    Hence, the rows are localisation sequences and fibre-cofibre sequences in $\PrLLst$ and in $\Catdual$.
\end{proof}

\begin{cor}[{\cite[p.~28]{andreychev-thesis}}]
\label{cor:andreychev-Kcont=KNuc-proregular}
    Let $A$ be a Tate ring. Then there is a canonical equivalence
    \[ \Knuc(A) \To[\simeq] \Kcont(A) \]
    of spectra which is natural in $A$, where $\Knuc(A) \coloneqq \K(\Nuc(A))$.
    
\end{cor}
\begin{proof}
    Let $(A_0, I)$ be an ideal of definition of $A$; we may assume $I = \varpi A_0$ with $\varpi \in A_0$ being a topologically nilpotent ideal of $A$. Then $I \subset A_0$ is weakly proregular by Example~\ref{example:proregularOneElement}.
    Taking the ``continuous K-theory'' à la Efimov \cite{efimov-k-theory} of the diagram (\ref{eq:CartesianDiagramNuc}) of dualisable categories, which is valued in the stable category $\Sp$, and using that $\K(-)$ is a localising invariant, we obtain a (co)cartesian diagram of spectra
    \begin{equation}
    \begin{tikzcd}
        \label{eq:CartesianDiagramNucK}
        \K(A_0) \arrow[r] \arrow[d] & \K(A) \arrow[d] \\
        \Knuc(A_0) \arrow[r] & \Knuc(A),
    \end{tikzcd}
    \end{equation}
    which maps compatibly to the pushout diagram defining the continuous K-theory of $A$:
    \begin{equation}
    \begin{tikzcd}
        \label{eq:CartesianDiagramKcont}
        \K(A_0) \arrow[r] \arrow[d] & \K(A) \arrow[d] \\
        \Kcont(A_0) \arrow[r] & \Kcont(A).
    \end{tikzcd}
    \end{equation}
    The corollary follows since we know by Theorem~\ref{thm:efimov-continuity-proregular} that the natural map $\Knuc(A_0) \to \lim_n \K(A_0/I^n) = \Kcont(A_0)$ is an equivalence of spectra.
    
    Finally, the equivalence $\Knuc(A) \to \Kcont(A)$ is independent of the choice of $(A_0, I)$ hence it is canonical, and it is natural in $A$, since for any morphism $f \colon A \to B$ of Tate rings and any pair of definition $(A_0, I)$ of $A$, there exists a pair of definition $(B_0, J)$ of $B$ with $B_0 \supset f(A_0)$.
\end{proof}

This equivalence can be lifted to the associated functors valued  in condensed spectra.
For that purpose, let us start with the following observation, which should already be known to experts, justifying the fact that one can endow a spectrum with \emph{discrete topology} to obtain a condensed spectrum satisfying hyperdescent on $\ProFin_\kappa$, not only on $\EDS_\kappa$.

\begin{lem}
    \label{lem:hyperdescentDiscreteCondensed}
    Let $X$ be a spectrum. Consider the following functor 
    \[\underline{X} \colon \ProFin_\kappa\op \to \Sp, \quad S = \lim_i S_i \mapsto \colim_i \prod_{S_i} X.\]
    Then $\underline{X}$ satisfies hyperdescent on $\ProFin_\kappa$.
\end{lem}

\begin{proof}
    Let $S_\bullet \to S$ be a hypercover in $\ProFin_\kappa$.
    As in \cite[Proof of Thm.~3.3]{condensed} (see \cite[Lem.~2.1.7]{Mann-Thesis} for a proof), one can write the hypercover $S_\bullet \to S$ as the cofiltered limit of hypercovers $S_{\bullet, j} \to S_j$ (indexed by $j \in J$) of finite sets by finite sets; in particular, the latter hypercovers all split.
    By definition and t-exactness of filtered colimits for spectra, we have identifications of (cosimplicial) abelian groups:
    \begin{align}
        \label{eq:KcolimProFin1}
        \pi_n \left(\underline{X}(S)\right) \simeq \colim_j \pi_n \left(\underline{X}(S_j)\right), \\
        \label{eq:KcolimProFin2}
        \pi_n \left(\underline{X}(S_\bullet)\right) \simeq \colim_j \pi_n \left(\underline{X}(S_{\bullet, j})\right).
    \end{align}
    Let us now compute the homotopy groups of $\lim_\Delta \underline{X}(S_\bullet)$ via the Bousfield--Kan homotopy spectral sequence \cite[Ch.~X, \S 5]{bousfieldkan1972lecture}, whose dual version can be found in \cite[\S 1.2.4]{HA}.\footnote{%
        When we study the homotopy spectral sequence of a cosimplicial spectra, switching to the dual version requires working with $\Sp\op$, but then the results in \cite[\S 1.2.2, \S 1.2.4]{HA} do not apply naively.
        In fact, the reasons are twofold. On the one hand, in the opposite category $\Sp\op$, taking filtered colimits does not commute with the ``dual homotopy group'' functors as the ``dual $t$-structure'' is not compatible with sequential limits; thus \cite[Prop.~1.2.2.14]{HA} is not applicable. On the other hand, our objects may not lie in $\Sp_{\leq 0} = (\Sp\op)_{\geq 0}$ (nor up to shift), but rather have many nontrivial deep homotopy groups, hence the simplified version \cite[Prop.~1.2.4.5]{HA} is not applicable, either.
        Nevertheless, one can still resort to the complete convergence lemma \cite[Ch.~IX, 5.4]{bousfieldkan1972lecture}, which imposes a condition on vanishing of $\lim^1_\NN$.
    }
    The spectral sequence specialised to our case is the following:
    \[E_1^{s,t} = \pi_{-t} \left(\underline{X}(S_s)\right) \Rightarrow \pi_{-s-t} \left(\lim_\Delta \underline{X}(S_s)\right),\]
    where we have reindexed the degrees so that the differential $d_r$ on the $E_r$-page has degree $(r, -r+1)$.
    We claim that:
    \begin{enumerate}[label=(\roman*)]
        \item The complex $(E_1^{\bullet, t}, d_1)$ is isomorphic to the normalised chain complex associated with the cosimplicial abelian group $\pi_{-t} \left(\underline{X}(S_\bullet)\right)$.
        \item The $E_2$-page collapses to the degree $s = 0$.
    \end{enumerate}
    The claim (i) follows from \cite[Rem.~1.2.4.4]{HA} by passing to the opposite category.
    Let us now deduce (ii) from (i).
    Consider a similar homotopy spectral sequence for the cosimplicial spectrum $\underline{X}(S_{\bullet, j})$:
    \[E_{1, (j)}^{s,t} = \pi_{-t} \left(\underline{X}(S_{s, j})\right) \Rightarrow \pi_{-s-t} \left(\lim_\Delta \underline{X}(S_{\bullet, j})\right).\]
    Again, as in (i), the complex $(E_{1, (j)}^{\bullet, t}, d_1)$ is isomorphic to the normalised chain complex associated with the cosimplicial abelian group $\pi_{-t} \left(\underline{X}(S_{\bullet, j})\right)$; the latter splits as all the $S_{\bullet, j}$ are finite sets, hence $E_{2, (j)}^{s, t}$ collapses to the degree $s = 0$ where $E_{2, (j)}^{0, t} \simeq \pi_{-t} \left(\underline{X}(S_j)\right)$. Since filtered colimits is exact on abelian groups, so is the construction of normalised chain complexes and passage from $E_1$-page to $E_2$-page. Therefore, the natural map $\colim_j E_{1, (j)}^{s, t} \to E_1^{s, t}$ is an isomorphism by (\ref{eq:KcolimProFin2}) and thus $\colim_j E_{2, (j)}^{s, t} \To[\simeq] E_2^{s, t}$, which collapses to the degree $s = 0$ where
    \[E_2^{0, t} \simeq \colim_j \pi_{-t} \left(\underline{X}(S_j)\right).\]
    This proves (ii).
    In particular, we have $\lim^1_r E_r^{s,t} = 0$ for any $s, t \in \ZZ$, hence the spectral sequence $E_r^{s,t}$ converges by the complete convergence lemma \cite[Ch.~IX, 5.3 and 5.4]{bousfieldkan1972lecture}, so that
    \[E_2^{0, t} \simeq \pi_{-t} \left(\lim_\Delta \underline{X}(S_\bullet) \right).\]
    In conclusion, using (\ref{eq:KcolimProFin1}), we obtain 
    \[\pi_{-t} \left(\underline{X}(S)\right) \To[\simeq] \pi_{-t} \left(\lim_\Delta \underline{X}(S_\bullet)\right), \quad \forall t \in \ZZ,\]
    where the arrow is induced by the canonical morphism $\underline{X}(S) \to \lim_\Delta \underline{X}(S_\bullet)$.
\end{proof}

\begin{lem}
    \label{lem:K(R)discrete}
    Let $R$ be an ordinary ring. Let $\underline{\K(R)} \in \Cond(\Sp)$ be the associated discrete condensed K-theory spectrum via \textup{Lemma~\ref{lem:hyperdescentDiscreteCondensed}}.
    Then $\underline{\K(R)}(S) \simeq \K(\Cont(S, R))$ for any $S \in \ProFin$.
\end{lem}

\begin{proof}
    We have
    \begin{align*}
        \underline{\K(R)}(S) & = \colim_i \prod_{S_i} \K(R) \\
        & \simeq \colim_i \K(\Map(S_i, R)) \\
        & \simeq \K(\colim_i \Map(S_i, R)) \\
        & \simeq \K(\Cont(S, R)),
    \end{align*}
    where the second to last equivalence is due to algebraic K-theory commuting with filtered colimits of rings.
\end{proof}

\begin{prop} 
\label{prop:Kcond=Kcont-enriched}
    There exists an equivalence
    \[ \Knucu \simeq \gamma_\kappa\Kcontu \]
    of presheaves on affinoid Tate adic spaces with values in $\Cond_\kappa(\Sp)$.
\end{prop} 

\begin{proof}
    For a (complete) Tate ring $A$ and $S\in\EDS_\kappa$ the topological ring $\Cont(S,A)$ is still a (complete) Tate ring. 
    Let $A_0 \subset A$ be a ring of definition. 
    We have
    \begin{align*}
        (\gamma_\kappa \Kcontu(A_0))(S) & \simeq \left(\lim_n \underline{\K(A_0/\varpi^n)}\right)(S) & \text{(Lemma~\ref{comparison-functor--lemma})} \\
        & \simeq \lim_n \left( \underline{\K(A_0/\varpi^n)}(S) \right)& \\
        & \simeq \lim_n \K(\Cont(S, A_0/\varpi^n)) & \text{(Lemma~\ref{lem:K(R)discrete})} \\
        & \simeq \Kcont(\Cont(S,A_0)).
    \end{align*}
    Hence we get equivalences
    \[(\gamma_\kappa\Kcontu(A))(S) \simeq \Kcont(\Cont(S,A)) \simeq \K(\Nuc(\Cont(S,A))) \simeq \Knucu(A)(S)\]
    which are natural in $A$ and $S$,
    where the middle equivalence is Corollary~\ref{cor:andreychev-Kcont=KNuc-proregular} and the right equivalence follows from combining Lemma~\ref{lem:profinite-tensor-adic} with Proposition~\ref{prop:condensed-enrichment-of-sheaves}. 
\end{proof}

One can in fact prove a stronger descent property with respect to $S \in \ProFin_\kappa$, beyond extremally disconnected sets. 

\begin{lem}
    \label{lem:hyperdescentProfiniteForDiscreteRings}
    Let $R$ be a discrete ring.
    Let $F$ be a $\Sp$-valued functor on commutative $R$-algebras which preserves finite products and filtered colimits.
    For any hypercover $S_\bullet \to S$ in $\ProFin_\kappa$, the natural morphism
    \begin{equation}
        \label{eq:KProFinDescentDiscrete}
        F(\Cont(S, R)) \to \lim_\Delta F(\Cont(S_\bullet, R))
    \end{equation}
    of spectra is an equivalence.
\end{lem}

\begin{proof}
    By assumption, for any $T = \lim_i T_i \in \ProFin_\kappa$, we have 
    \[
        F(\Cont(T, R)) \simeq F(\colim_i \Cont(T_i, R)) \simeq \colim_i F(\Cont(T_i, R)) \simeq \colim_{T_i} \prod_i F(R).
    \]    
    Hence, the functor
    \[\ProFin_\kappa\op \to \Sp, \quad T \mapsto F(\Cont(T, R))\]
    is identified with the functor $\underline{F(R)}$ associated with the spectra $F(R)$ as in Lemma~\ref{lem:hyperdescentDiscreteCondensed}. Then we conclude by this lemma. 
\end{proof}

\begin{prop}
    \label{prop:KnucHyperdescentProfinite}
    For any Tate ring $A$ and any hypercover $S_\bullet \to S$ in $\ProFin$, the natural map
    \[\Knuc(\Cont(S, A)) \to \lim_\Delta \Knuc(\Cont(S_\bullet, A))\]
    is an equivalence of spectra.
    In particular, the light version of $\Knucu$ is canonically equivalent to the restriction of the $\kappa$-condensed version of $\Knucu$ to light profinite sets.
\end{prop}

\begin{proof}
    Choose a pseudo-uniformiser $\varpi \in A_0$.
    From the cartesian diagram (\ref{eq:CartesianDiagramNucK}) and Efimov's equivalence (\ref{eq:efimov-continuity-proregular}), we obtain a fibre sequence
    \begin{equation}
        \label{eq:KonPiForS}
        F(\Cont(T, A_0)) \to \lim_n \K(\Cont(T,A_0)/\varpi^n) \to \Knuc(\Cont(T, A))
    \end{equation}
    for any profinite set $T$, where $F(B) \coloneqq \fib(\K(B) \to \K(B[\varpi^{-1}]))$ for any (ordinary) $A_0$-algebra $B$. 
    The descent for $\Knuc(\Cont(-, A_0))$ is then reduced to that for $F(\Cont(-, A_0))$ and for $\K(\Cont(-, A_0/\varpi^n))$.

    On the one hand, applying Lemma~\ref{lem:hyperdescentProfiniteForDiscreteRings} to the algebraic K-theory functor and $R = A_0/\varpi^n$, one obtains the equivalences (for $n\in\NN$).
    \begin{equation}
        \label{eq:torsionProFinDescent}
        \K(\Cont(S, A_0/\varpi^n)) \To[\simeq] \lim_\Delta \K(\Cont(S_\bullet, A_0/\varpi^n)).
    \end{equation}
    
    On the other hand, applying Lemma~\ref{lem:hyperdescentProfiniteForDiscreteRings} to $F(-)$ and $R = A_0^\discrete$ (the ring $A_0$ equipped with the discrete topology), we see that there is a natural equivalence
    \begin{equation}
        \label{eq:fiberProFinDescentDiscrete}
        F(\Cont(S, A_0^\discrete)) \To[\simeq] \lim_\Delta F(\Cont(S_\bullet, A_0^\discrete))
    \end{equation}
    for any hypercover $S_\bullet \to S$.
    Let us get rid of the discrete topology by considering the localisation sequence $\Tor_B(\varpi^\infty) \to \Dcal(B) \to \Dcal(B[\varpi^{-1}])$. Using the fact that $\K(-)$ is a localising invariant, we see that $F(B) \simeq \K(\Tor_B(\varpi_\infty))$, which is insensitive to the $\varpi$-adic completion. We deduce that for any profinite set $T$, the natural map
    \begin{equation}
        \label{eq:KonPiInsensitiveToCompletion}
        F(\Cont(T, A_0^\discrete)) \to F(\Cont(T, A_0))
    \end{equation}
    of spectra is an equivalence, hence (\ref{eq:fiberProFinDescentDiscrete}) becomes
    \begin{equation}
        \label{eq:fiberProFinDescentAdic}
        F(\Cont(S, A_0)) \To[\simeq] F(\Cont(S_\bullet, A_0)),
    \end{equation}
    which shows the desired hyperdescent for $F$.
\end{proof}

\begin{rem}[Nisnevich topology]
\label{Rem:Nisnevich-topology}
The inclusion $\AnAdic^\qcqs \inj \AnAdic$ induces an equivalence 
\[ \Sh(\AnAdic) \To[\simeq] \Sh(\AnAdic^\qcqs). \]
On the category $\AnAdic^\qcqs$ there is the Nisnevich topology \cite[Appendix~A]{andreychev-thesis} inducing a reflective subcategory $\Sh_\Nis(\AnAdic^\qcqs) \inj \Sh(\AnAdic)$. We define the Nisnevich topology on $\AnAdic$ to be the one that corresponds to the fully faithful right afjoint functor $\Sh_\Nis(\AnAdic^\qcqs) \inj \Sh(\AnAdic)$.

More explicitly, the Nisnevich topology on $\AnAdic$ as the topology generated by Nisnevich squares
\begin{equation}
    \begin{tikzcd}
        \label{Diagram:NisnevichSquareSec4}
        \tag{$\square$}
        V \arrow[r,""]\arrow[d,""]& Y \arrow[d,""]\\
        U \arrow[r,""]& X
    \end{tikzcd}
\end{equation}
consisting of \emph{qcqs} analytic adic spaces and by \emph{all} analytic covers. For this topology, Nisnevich squares formed by \emph{arbitrary} analytic adic spaces are also covers.\footnote{%
    Indeed, by base change along an affinoid open cover of $X$, we may assume $X$ to be qcqs; then, by choosing affinoid open covers of $U$ and $Y$ we may refine the family $\{U \to X, Y \to X\}$ to a family of étale morphisms of affinoids $\{U_i \to X, Y_j \to X\}_{i, j}$ satisfying the Nisnevich condition on points, hence this refinement forms a Nisnevich cover in $\AnAdic^\qcqs$ by \cite[Satz~A.21]{andreychev-thesis}.
}
Moreover, as one may expect, an étale over $\{U_\alpha \to X\}_\alpha$ in $\AnAdic$ is a Nisnevich cover if and only if for any henselian affinoid field $(K, K^+)$ \cite[Def.~A.9]{andreychev-thesis}, the map $\{U_\alpha(K, K^+) \to X(K, K^+)\}_\alpha$ is jointly surjective; cf.\@ \cite[Def.~A.10]{andreychev-thesis}.\footnote{%
    We may assume $X$ to be affinoid. The ``if'' part can be reduced to the case where $X$ and $U_\alpha$ are all affinoid by base change along an affinoid cover of $X$ and then passing to affinoid subcovers of each $U_\alpha$. The ``only if'' part can be reduced to analytic covers and Nisnevich squares for qcqs objects, where it holds trivially resp. holds by definition.
}
In particular, the non-qcqs version of \cite[Satz~A.21]{andreychev-thesis} holds: an étale cover $\{U_\alpha \to X\}_\alpha$ in $\AnAdic$ is a Nisnevich cover if and only if for any point $x \in X$, there exists a point $y$ of some $U_\alpha$ such that the degree $d(y/x)$ is one.
\end{rem}

\begin{thm}[Nisnevich descent for continuous K-theory]
\label{Thm:Nisnevich-descent-Kcont}
The presheaf
\[ \Kcontu \,\colon\, \AnAdic^{\aff,\opp} \To \Pro(\Sp^+),\quad \Spa(A,A^+) \mapsto \Kcontu(A) \]
satisfies descent for the analytic topology on affinoids, hence it extends uniquely to a sheaf for the analytic topology, which we still denote by $\Kcontu$. Furthermore, it satisfies Nisnevich descent.
\end{thm}

\begin{cor}
\label{Cor:Kcond=Kcont-enriched-as-sheaves}
    There exists an equivalence
    \[ \Knucu \simeq \gamma_\kappa\Kcontu \]
    in the category $\Sh_\Nis(\AnAdic,\Cond_\kappa(\Sp))$.
\end{cor}
\begin{proof}
    This follows from Theorem~\ref{Thm:Nisnevich-descent-Kcont} and Proposition~\ref{prop:Kcond=Kcont-enriched}.
\end{proof}

\begin{proof}[Proof of Theorem~\ref{Thm:Nisnevich-descent-Kcont}]
    The Nisnevich topology is generated by the analytic topology and by Nisnevich squares of qcqs objects (Remark~\ref{Rem:Nisnevich-topology}), hence it suffices to show that $\Kcontu$ sends any Nisnevich square of affinoid resp.\@ qcqs objects to a pullback square.
    
    First, recall that $\Knuc$ is a Nisnevich for qcqs analytic adic spaces \cite[Satz~5.12]{andreychev-thesis}. Next, we shall lift this to condensed spectra.
    Given a Nisnevich square (\ref{Diagram:NisnevichSquareSec4}) of \emph{affinoid Tate} adic spaces, for every extremally disconnected set $S\in\EDS$ 
    the tensored square (\ref{Diagram:NisnevichSquareTensorS})
    is a Nisnevich square by Lemma~\ref{Lem:EDS-tensor-Nisnevich-square}. Since limits of condensed spectra are computed pointwise on $\ProFin_\kappa$, the square 
    $$ \begin{tikzcd}
		\Knucu(V) \arrow[r,""]\arrow[d,""]& \Knucu(Y) \arrow[d,""]\\
		\Knucu(U) \arrow[r,""]& \Knucu(X)
    \end{tikzcd} $$
    is cartesian.
    
    The equivalence $\gamma\Kcontu\simeq\Knucu$ (Proposition~\ref{prop:Kcond=Kcont-enriched}) yields that the square $\gamma \Kcontu(\textup{\ref{Diagram:NisnevichSquareSec4}})$ is cartesian for \emph{affinoid Tate} adic spaces.
    The comparison functor $\gamma \colon \Pro^\omega(\Sp^+) \to \Cond^\omega(\Sp)$ preserving limits and being conservative on eventually constant objects (Theorem~\ref{thm:comparison-conservative}), we deduce the cartesianness of $\Kcontu(\textup{\ref{Diagram:NisnevichSquareSec4}})$ as a result of the fact that $\Kcontu(-)$ of any Tate ring is eventually constant (Definition~\ref{def:continuous-k-theory}).
    Therefore, we have shown descent for binary open affinoid open covers; by standard reduction steps, this implies the descent for the analytic topology on Tate affinoids. Hence $\Kcontu$ extends uniquely to a sheaf for the analytic topology.
    
    By induction on the cardinality of affinoid covers, we see that $\Kcont(-)$ of any qcqs analytic adic space is also eventually constant. Then the same argument as above works and shows that $\Kcontu(\textup{\ref{Diagram:NisnevichSquareSec4}})$ is cartesian for any Nisnevich square (\ref{Diagram:NisnevichSquareSec4}) of \emph{qcqs} analytic adic spaces.
\end{proof}

\begin{prop}
    If $X$ is locally of the form $\Spa(A, A^+)$ such that \ref{dagger-condition} holds, then the canonical map
    \[ \Kcont(X) \To[\simeq] \Kcont(X\times \Arig) \] 
    is an equivalence in $\Pro(\Sp^+)$. Moreover, the connective cover of $\Kcont(X)$ agrees with the connective analytic K-theory $\kan(X)$ \textup{\cite[Def.~6.5]{kst-i}}.
\end{prop}
\begin{proof}
The claims can be checked locally. In this case, $\Kcont(X)\simeq\Kan(X)$ \cite[Thm.~6.19]{kst-i} and $\Kan$ is pro-$\Brig$-invariant by design, hence $\Arig$-invariant by Lemma~\ref{pro-invariant=Aan-invariant--lem}, since it is a sheaf \cite[\S 4.1]{kst-ii}.
\end{proof}

\subsection{The non-noetherian case}
\label{sec:non-noetherian-case}

In the work of Kerz--Saito--Tamme often there appears the condition on a Tate ring $A$ that it admits a noetherian ring of definition. The main necessity of this assumption lies in the use of pro-cdh descent --- which back then has been established only for noetherian schemes by Kerz--Strunk--Tamme \cite[Thm.~A]{kerz-strunk-tamme}. Thanks to a generalisation to arbitrary qcqs (derived) schemes by Kelly--Saito--Tamme \cite[Thm.~A]{kelly-saito-tamme}, those results about continuous K-theory which rely on pro-cdh descent generalise immediately.

For a morphism of animated rings (resp.\@ of $\EE_\infty$-rings), there is a notion of being \emph{almost of finite presentation}, which means that the morphism can be built using only finitely many cells in each dimension. This property can be checked on the underlying morphism of $\EE_\infty$-rings. See \cite[\S 2.4]{kelly-saito-tamme} and \cite[\S 4.1]{SAG} for a more detailed presentation.
This property is stable under base change and compositions.
We also have the following typical examples: the morphisms $R \to R[T_1, \dots, T_m]$ for any animated ring $R$ is almost of finite presentation; a surjection $R \twoheadrightarrow R'$ of ordinary commutative rings is almost of finite presentation if $R'$ is a pseudo-coherent $R$-module \cite[Cor.~5.2.2.3]{SAG}, i.e.\@ if $R'$ has a (possibly infinite) resolution by finite free $R$-modules.

\begin{lem}
    \label{lem:flattening_over_kcirc}
    Let $R$ be a flat $k^\circ$-algebra of topologically finite type. Let $B$ be a finite type $R$-algebra which is flat over $k^\circ$. When regarded as a morphism of animated rings or as a morphism of $\EE_\infty$-rings, $R \to B$ is almost of finite presentation.
\end{lem}

\begin{proof}
    Let $R[T_1, \dots, T_m] \twoheadrightarrow B$ be a presentation of $B$ as a finite type $R$-algebra. It suffices to show that $B$ is a pseudo-coherent $R[T_1, \dots, T_m]$-module \cite[Cor.~5.2.2.3]{SAG}.

    For this, we resort to certain flattening techniques of Raynaud--Gruson in the setting of an arbitrary nonarchimedean base field $k$, which is available more generally for so-called \emph{topologically universally adhesive pairs} \cite[Chap.~0, \S 8.5]{fuji-kato} (see also \cite[\S A.2, \S A.3]{zavyalov2024quotient} for a nice summary).
    For any $k^\circ$-algebra $R$ of topologically finite type, the pair $(R, \pi)$ is topologically universally adhesive by a result of Gabber \cite[Chap.~0, Cor.~9.2.7]{fuji-kato}; in particular, $(R[T_1, \dots, T_m], \pi)$ is adhesive, hence $B$, which is flat over $k^\circ$, is a finitely presented $R[T_1, \dots, T_m]$-module \cite[Chap.~0, Prop.~8.5.3 (a)$\Leftrightarrow$(c)]{fuji-kato}. As $R$ is also flat over $k^\circ$, the ring $R$ is universally coherent \cite[Chap.~0, Thm.~8.5.25 (1)]{fuji-kato}; in particular, $R[T_1, \dots, T_m]$ is a coherent ring. This implies then that $B$ is a pseudo-coherent $R[T_1, \dots, T_m]$-module \cite[\href{https://stacks.math.columbia.edu/tag/05CU}{Section 05CU}]{stacks-project}.
\end{proof}

\begin{prop}
\label{prop:Kcont-model-fibre-sequence}
    Let $A$ be a Tate ring with ring of definition $A_0$ which is of topologically finite type over $k^\circ$ or is noetherian and finite-dimensional. Let $\varpi\in A_0$ be a pseudo-uniformiser.
    Then for every proper morphism $X \to \Spec(A_0)$ of finite presentation which is an isomorphism outside the locus $V(\varpi)$ such that $X$ is flat over $k^\circ$ in the former case, there is a fibre sequence
    \[ \K(X\on\varpi) \To \Kcontu(X) \To \Kcontu(A) \]
    in $\Pro(\Sp^+)$.
\end{prop}

\begin{proof}
    The noetherian case is \cite[Prop.~5.8]{kst-i}. In the case over $k^\circ$, the same proof as for \cite[Prop.~5.8]{kst-i} goes through by using \cite[Thm.~A]{kelly-saito-tamme} instead of \cite[Thm.~A]{kerz-strunk-tamme}.
    In order to apply this last theorem, we need that the admissible blow-up morphisms are locally almost of finite presentation, which is guaranteed by Lemma~\ref{lem:flattening_over_kcirc}.
\end{proof}

\begin{prop}[Weibel vanishing]
\label{prop:weibel-vanising-Kcont}
    Let $\base = \Spa(k, k^\circ)$ for a nonarchimedean field $k$, or $\base = \Spa(R, R^+)$ with $R$ being a Tate ring\footnote{%
        As follows from the proof, by patching arguments, one could relax the condition of $R$ being ``Tate'' to being ``analytic''. We do not pursue this direction of generality here.
    }
    that admits a noetherian and finite-dimensional ring of definition.
    Let $X$ be a qcqs adic space of locally finite type over $\base$ and of finite dimension~$d$.
    Then there exists an integer $N$ such that:
    \begin{enumerate}
        \item $\Kcontu_{-n}(X)=0$ for $n>N$.
        \item The canonical morphism $\Kcontu_{-n}(X) \to \Kcontu_{-n}(X\times\BB^m_k)$ is an isomorphism for $m\geq 0$ and $n\geq N$.
    \end{enumerate}
    Furthermore, one can take $N = d$ if $\base = \Spa(k, k^\circ)$ or if $\base = \Spa(R, R^+)$ and the ring of definition is moreover excellent.
\end{prop}
\begin{proof}
    Let us first consider the case of $\base = \Spa(k, k^\circ)$.
    Let $\mfr \subset k^\circ$ be the maximal ideal and let $\varpi\in k^\circ$ a pseudo-uniformiser so that $\mfr = \sqrt{(\varpi)}$.
    Any qcqs $X \in \Adic_k^\lft$ admits a cofiltered system of admissible formal models $\Xcal$ (see for example \cite[\S 8.4, Lemma~4 (e)]{bosch} and \cite[Ch.~II, Thm.~A.5.2]{fuji-kato}).
    For any admissible formal model $\Xcal$ of $X$, define $\Kcontu(\Xcal\on\mfr)$ to be the fibre of the map $\Kcontu(\Xcal) \to \Kcontu(X)$ as in \cite[Def.~8.2]{dahli-thesis-paper} and consider the induced exact sequence
    \[ \Kcontu_n(\Xcal\on\mfr) \To \Kcontu_n(\Xcal) \To \Kcontu_n(X) \To \Kcontu_{n-1}(\Xcal\on\mfr) \]
    of pro-abelian groups.
    Recall that we have that $\dim(\Xcal/\mfr) = \dim(X) = d$ \cite[Ch.~II, Cor.~10.1.11]{fuji-kato}.\footnote{%
        For the convenience of the reader, we make a transcription of the conditions in the statement \cite[Ch.~II, Cor.~10.1.11]{fuji-kato}: the adjective ``coherent'' means qcqs as in the topos theory; the condition ``of type $(V_\RR)$'' means being locally of finite type over a rank-one nonarchimedean field; the adjective ``distinguished'' for formal models amounts to admissibility.
    }
    Hence, for $n\leq -d$, we have 
    \[ \Kcontu_n(\Xcal) = \prolim_k\K_n(\Xcal/\varpi^k) = \K_n(\Xcal/\varpi) = \K_n(\Xcal/\mfr) \]
    by definition, by nil-invariance in low degrees \cite[Lem.~6.3]{dahli-thesis-paper}, and since algebraic K-theory commutes with filtered colimits.
    Using algebraic Weibel vanishing \cite[Thm.~B (i)]{kerz-strunk-tamme} we deduce, for $n<-d$, a monomorphism $\Kcontu_n(X) \inj \Kcontu_{n-1}(\Xcal\on\mfr)$.
    The proof of \cite[Lem.~8.6]{dahli-thesis-paper} shows that $\colim_{\Xcal}\Kcontu_n(\Xcal\on\mfr)=0$ for $n\leq -2$, provided that this holds for $X$ being affinoid and $n\leq -1$; the latter follows as in the proof of \cite[Prop.~7]{kerz-icm2018} thanks to Proposition~\ref{prop:Kcont-model-fibre-sequence}.\footnote{%
        More precisely, Kerz shows that for $n<0$ and any element $\alpha\in \Kcontu_n(A_0 \on \mfr) = \K_n(A_0 \on \mfr)$, where $\Xcal=\Spf(A_0)$, there exists an admissible blow-up $p\colon X'\to \Spec(A_0)$ such that $p^*(\alpha)=0$ in $\Kcontu_n(X' \on \mfr)$. The morphism $p$, which is provided by Raynaud--Gruson's \emph{platification par éclatement} \cite[Thm.~5.2.2]{raynaud-gruson}, is a $U$-admissible blow-up in the sense of loc.\@ cit.\@ (for $U = \Spec(A_0) \times_{k^\circ} k = \Spec(A)$) and hence of finite presentation \cite[Def.~5.1.3]{raynaud-gruson}. Thus we are allowed to use Proposition~\ref{prop:Kcont-model-fibre-sequence} for the model $X'$. Recall that, since $\Xcal$ is affine, every admissible formal blow-up is algebraic\cite[Lem.~5.5]{dahli-thesis-paper}.
    }
    This shows (1). The same argument applied to the cofibre of the map $\Kcontu(X) \to \Kcontu(X\times\BB^n_k)$ instead of $\Kcontu(X)$, where we use in addition the homotopy invariance of sufficiently negative $\K$-groups \cite[Thm.~B (ii)]{kerz-strunk-tamme}, yields (2) since the induced morphism on pro-homotopy groups are split monomorphisms.
    
    In the case of $\base = \Spa(R, R^+)$ (with a pseudo-uniformiser $\varpi$), the arguments are similar, but with the following modifications.
    For any affinoid open $\Spa(A, A^+) \subset X$, the complete ring $A$ is also Tate and admits a noetherian and finite-dimensional ring of definition as well (see the general discussion above \cite[Lem.~4.2]{kst-ii}); 
    hence we no longer need the radical $\mfr = \sqrt{\varpi R^+}$ to pass to $\Xcal/\mfr$, since $\Xcal/\varpi$ is already noetherian.
    Any qcqs $X \in \Adic_\base^\lft$ still admits a cofiltered system of noetherian admissible formal models $\Xcal$ but by applying \cite[Ch.~II, Thm.~A.5.3]{fuji-kato} this time.
    The main difference now is that the dimension formula $\dim(\Xcal/\mfr) = \dim(X)$ in the previous case may no longer hold in general. Nevertheless, $\dim(\Xcal/\varpi)$ is uniformly bounded above when $\Xcal$ runs through all admissible formal models of a given qcqs $X$ (since admissible formal blow-ups do not change the dimension of the underlying space). Then, the same arguments as above show that (1) and (2) hold for qcqs $X \in \Adic_\base^\lft$ with (finite) lower bound $- \dim(\Xcal/\varpi)$ for one, hence for all, $\Xcal$.
    
    Finally, assume that $R$ admits an \emph{excellent} ring of definition.
    Then for any affine open $\Spf(A_0) \subset \Xcal$ (where $A_0$ is $\varpi$-adically complete and $\Xcal$ is not necessarily qcqs), $A_0$ is also excellent (see for example \cite[Main Theorem 2]{excellence-adic-completion}). Let us prove the dimension identity $\dim(\Xcal/\varpi) = \dim(X)$. It suffices to treat the affine case.
    On the one hand, we have $\dim(\Spf(A_0)/\varpi) = \dim(A_0/\varpi)$.
    On the other hand, we have $\dim(\Spf(A_0)_\eta) = \dim(A_0[\varpi^{-1}])$ by \cite[Ch.~II, Cor.~10.1.10]{fuji-kato}.
    Observe that any maximal ideal of $A_0$ contains $\varpi$ by $\varpi$-adic completeness, and that $A_0$ is $\varpi$-torsion-free by admissibility. The desired dimension equality now follows from Lemma~\ref{lem:commutativeAlgebraDimension} below, as $A_0$ is noetherian and catenary.
\end{proof}

\begin{lem}
    \label{lem:commutativeAlgebraDimension}
    Let $A_0$ be a commutative algebra and let $\varpi$ be an element of the Jacobson ideal of $A_0$ (i.e.\@ it belongs to all maximal ideals).
    \begin{enumerate}
        \item We have $\dim(A_0[\varpi^{-1}]) \leq \dim(A_0)-1$.
    \end{enumerate}
    Assume moreover that $A_0$ is noetherian and $\varpi$-torsion-free.
    \begin{enumerate}[resume]
        \item If $A_0$ is catenary, then we have $\dim(A_0[\varpi^{-1}]) \geq \dim(A_0) - 1$.
        
        \item We have $\dim(A_0/\varpi) = \dim(A_0) - 1$.
    \end{enumerate}
\end{lem}

\begin{proof}
    (1) 
    This is because no maximal ideal of $A_0$ survives in $A_0[\varpi^{-1}]$, by hypothesis that $\varpi$ belongs to all maximal ideals.

    (2) 
    It suffices to show that for any maximal ideal $\mfr \subset A_0$, we have $\dim(A_0[\varpi^{-1}]) \geq \dim((A_0)_\mfr) - 1$. Observe that as $A_0$ is $\varpi$-torsion-free, hence $\varpi$ does not belong to any minimal prime of $(A_0)_\mfr$. We may assume that $(A_0)_\mfr$ is of finite Krull dimension; otherwise we will get infinite dimension on both sides of the inequality, in which case this becomes trivial. We choose furthermore a minimal prime $\pfr_o$ of $A_0$ under $\mfr$ such that $\dim((A_0)_\mfr) = \dim((A_0)_\mfr/\pfr_o)$.

    Let $\pfr \subset A_0$ be maximal among the prime ideals between $\pfr_o$ and $\mfr$ but not containing $\varpi$; by abuse of notation, we denote by $\pfr$ the corresponding prime ideal of $(A_0)_\mfr$.
    Such $\pfr$ exists since $\varpi$ does not belong to $\pfr_o$, and $\pfr \neq \mfr$ because $\varpi \in \mfr$ by hypothesis.
    Then $(A_0)_\mfr/\pfr$ is an integral ring of positive dimension where every non-zero prime ideal $\bar\qfr \subset (A_0)_\mfr/\pfr$ contains the image $\bar\varpi (\neq 0)$ of $\varpi$.
    Since $(A_0)_\mfr/\pfr$ is integral, every non-zero element of $\bar\qfr$ is contained in some prime ideal of height one of $(A_0)_\mfr/\pfr$ by Krull's Hauptidealsatz. These prime ideals are all minimal prime ideals over $\bar\varpi$ by maximality of $\bar\pfr$.
    However, there are only finitely many minimal prime ideals over $\bar\varpi$ by noetherianity. So $\bar\qfr$ is contained in the \emph{finite} union of these prime ideals of height one of $(A_0)_\mfr/\pfr$, hence $\bar\qfr$ is contained in one of them by prime avoidance lemma.
    This implies that $\height(\bar\qfr) = 1$.
    Therefore we obtain $\dim((A_0)_\mfr/\pfr) = 1$, thus by catenarity, $\dim((A_0)_\mfr) = \height(\pfr) + 1 \leq \dim(A_0[\varpi^{-1}]) + 1$.

    (3)
    For any maximal ideal $\mfr$ of $A_0$, which always contains $\varpi$ by hypothesis, we have $\dim((A_0)_\mfr/\varpi) = \dim((A_0)_\mfr) - 1$ by Krull's Hauptidealsatz and $\varpi$-torsion-freeness. Taking supremum over $\mfr$, the left hand side becomes $\dim(A_0/\varpi)$, and the right hand side becomes $\dim(A_0) - 1$. 
\end{proof}

\subsection{Analytic K-theory}
\label{subsec:analytic-k-theory}
In this subsection, we want to recall the definition of analytic K-theory as a pro-object. Afterwards, we will see that under finiteness assumptions, analytic K-theory agrees with the $\Arig$-localisation of continuous K-theory. Consequently, we will only work with continuous K-theory and keep the comparison in mind.

The idea of (connective) analytic K-theory is to define an analogue of topological K-theory for rigid spaces. The new idea in \cite{kst-i} is to view analytic K-theory as an $\Arig$-invariant object, similarly as Weibel's definition of homotopy invariant K-theory $\KH$.

Let us be more precise: Let $A$ be a Tate ring with uniformiser $\pi$. Then for $n,j\geq 0$ define the Tate ring
\[ A\skp{\Delta^n_{\pi^j}} := A\skp{t_0,\ldots,t_n}/(t_0+\ldots+t_n-\pi^j). \]
For every $j$, this yields a simplicial Tate ring $A\skp{\Delta_{\pi^j}}$ via the standard structure maps. Via geometric realisation of the simplicial spectrum $[n] \mapsto \K_{\geq 1}(A\skp{\Delta^n_{\pi^j}})$ we get a connected spectrum
\[ \K_{\geq 1}(A\skp{\Delta_{\pi^j}}) := \colim_{[n]\in\Delta\op} \K_{\geq 1}(A\skp{\Delta^n_{\pi^j}}). \]
The structure morphisms $A\skp{\Delta_{\pi^{j+1}}} \to A\skp{\Delta_{\pi^j}}$ defined by $t_i\mapsto\pi t_i$ induce an inverse system of connected pro-spectra so that we can define the \emph{connected analytic K-theory} of $A$ to be the pro-spectrum
    \[ \Kanu_{\geq 1}(A)\coloneqq \prolim_{j} \K_{\geq 1} (A\skp{\Delta_{\pi^j}}). \]
Non-connective analytic K-theory is defined by delooping connected analytic K-theory via the analytic Bass construction $(-)^{\Bass}$ \cite[\S 6.3]{kst-i}; for details consider Section~\ref{sec:analytic-bass-construction} below.

\begin{defi}[{\cite[Def.~6.15]{kst-i}}]
    For a Tate ring $A$, its \emph{non-connective analytic K-theory} is defined to be the pro-spectrum
    \[ \Kanu(A) := (\Kanu_{\geq 1})^{\Bass}(A).\]
\end{defi}

Before applying Bass delooping, one could do the same construction starting with \emph{connective} algebraic K-theory obtaining a connective pro-spectrum $\kanu(A)$ \cite[Def.~6.5]{kst-i}. For any $A$, we have at least a fibre sequence
\begin{align}
\label{fibre-sequence-1-connective-cover}
    \Kanu_{\geq 1}(A)\rightarrow \kanu(A)\rightarrow \K_{0}(A)
\end{align} 
in $\Pro(\Sp^+)$ \cite[Proof of Lem.~6.13]{kst-i}.
If $A$ is regular and satisfies \ref{dagger-condition}, then the canonical map $\Kanu(A) \to (\kanu)^{\Bass}(A)$ is an equivalence \cite[Rem.~6.15, Lem.~6.18]{kst-i}. In particular, the canonical map $\Kanu_{\geq 0}(A) \to \kanu(A)$ is an equivalence for such $A$.
The reason for starting with connected algebraic K-theory is that this construction has better formal properties:

\begin{thm}[Kerz--Saito--Tamme]
\label{thm:KST-Kan-sheaf}
Let $\base=\Spa(R,R^+)$ where $R$ is a Tate ring that admits a noetherian and finite-dimensional ring of definition. Then the presheaf
\[ \Kanu \colon \Adic_\base^{\aff,\ft,\opp} \To \Pro(\Sp^+),\quad \Spa(A,A^+) \mapsto \Kanu(A) \]
is an $\Arig$-invariant sheaf for the analytic topology. In particular, it extends to a sheaf on the category $\Adic_\base^\lft$.
\end{thm}
\begin{proof}
By \cite[Thm.~4.1]{kst-ii}, it is a sheaf. By \cite[Cor.~2.7]{kst-ii}, it is pro-homotopy invariant, hence $\Arig$-invariant by Lemma~\ref{pro-invariant=Aan-invariant--lem}.
\end{proof}

The next proposition shows that homotopy invariant continuous K-theory is equivalent to analytic K-theory for certain adic spaces.
Note that if $\base=\Spa(R, R^+)$ with $(R, R^+)$ as above, then for any $X \in \Adic_\base^\lft$ and any open affinoid $U = \Spa(A, A^+) \subset X$, the Tate ring $A$ admits a finite-dimensional, noetherian ring of definition.

\begin{prop}[$\Arig$-localisation]
\label{prop:A1-localisation}
Let $A$ be a Tate ring of topologically finite type over $k$ or a Tate ring that admits a noetherian and finite-dimensional ring of definition.
Then there exists a canonical weak equivalence
\[ (\L_{\Arig}\Kcontu)(A) \simeq \Kanu(A). \]
\end{prop}
\begin{proof}
We have the following chain of equivalences:
\begin{align*}
(\L_{\Arig}\Kcontu)(A) 
&\simeq \colim_{n\in\Delta\op} \Kcontu(\Delta_A^{\an,n})
&\cite[\textup{Lem.}~3.6]{DY} \\
&\simeq \colim_{n\in\Delta\op} \prolim_j \Kcontu(A\skp{\Delta_{\pi^j}^n})
&(\textup{sheaf condition}) \\
&\simeq \prolim_j \colim_{n\in\Delta\op} \Kcont(A\skp{\Delta_{\pi^j}^n})
& (\clubsuit) \\
&\simeq \prolim_j \Kcontu(A\skp{\Delta_{\pi^j}})
& (\textup{definition}) \\
&\simeq \Kan(A) 
& \cite[\textup{Cor.}~2.9]{kst-ii}
\end{align*}
Note that the cited result for the last equivalence assumes that $A$ admits a noetherian and finite-dimensional ring of definition. This assumption is not necessary any more if $A$ is of topologically finite type over $k$, thanks to pro-cdh-descent for arbitrary qcqs schemes \cite[Thm.~A]{kelly-saito-tamme} and generalised Weibel vanishing Proposition~\ref{prop:weibel-vanising-Kcont}. For the equivalence $(\clubsuit)$ we use that the geometric realisation commutes with the pro-limit \cite[\textup{Lem.}~2.8]{kst-ii} since the spectra in question are uniformly bounded below according to Proposition~\ref{prop:weibel-vanising-Kcont}, by our assumption on $A$.
\end{proof}

\begin{defi} \label{def:KHcont}
We define \emph{continuous homotopy K-theory} to be the $\Arig$-localisation of continuous K-theory, i.e. $\KHcontu \coloneqq \L_{\Arig}\Kcontu\in \PSh^{\AA^1}(\Adic^\lft_\base, \Pro(\Sp^+))$.
\end{defi}

\begin{lem}[Nisnevich descent for $\KHcontu$]
\label{lem:Nisnevich-descent-KHcont}
Let $\base = \Spa(k, k^\circ)$ for a nonarchimedean field $k$, or $\base = \Spa(R, R^+)$ with $R$ being a Tate ring that admits a noetherian and finite-dimensional ring of definition.
Then the presheaf
\[ \KHcontu \,\colon\, \Adic_\base^{\lft,\opp} \To \Pro(\Sp^+),\quad \Spa(A,A^+) \mapsto \KHcontu(A)\]
is a Nisnevich sheaf on quasi-separated spaces of finite dimension (e.g.\@ qcqs spaces).
In particular, $\Kanu$ is a Nisnevich sheaf.
\end{lem}
\begin{proof}
    For any elementary Nisnevich square (\ref{Diagram:NisnevichSquare}), the square $\Kcontu \left( (\textup{\ref{Diagram:NisnevichSquare}}) \skp{\Delta_{\pi^j}^n} \right)$ is cartesian by Theorem~\ref{Thm:Nisnevich-descent-Kcont}, since the square $(\textup{\ref{Diagram:NisnevichSquare}}) \skp{\Delta_{\pi^j}^n}$, being the base change of an elementary Nisnevich square, is still an elementary Nisnevich square.\footnote{%
        It is clear that étaleness is preserved, and the property of being degree one over the non-isomorphic locus is also preserved by \cite[Lemma~A.19]{andreychev-thesis}.
    }
    
    For $X \in \Adic_\base^\lft$, we have the following chain of equivalences in $\Pro^\omega(\Sp^+)$:
    \begin{align*}
        \KHcontu(X) &\simeq \colim_{n\in\Delta\op} \Kcontu(X \times \Delta^{\an,n}) &\cite[\textup{Lem.}~3.6]{DY} \\
        &\simeq \colim_{n\in\Delta\op} \prolim_j \Kcontu(X\skp{\Delta_{\pi^j}^n}) & \textup{(analytic descent for $\Kcontu$)} \\
        &\simeq \prolim_j \colim_{n\in\Delta\op} \Kcontu(X\skp{\Delta_{\pi^j}^n}) & \textup{if $X$ is qcqs, by $(\clubsuit)$}
    \end{align*}
    For the equivalence $(\clubsuit)$ we use that the geometric realisation commutes with the pro-limit \cite[\textup{Lem.}~2.8]{kst-ii} since the spectra in question are uniformly bounded below by Weibel vanishing (Proposition~\ref{prop:weibel-vanising-Kcont}).
    Since $\Pro^\omega(\Sp^+)$ is stable by Lemma~\ref{Pro-omega-stable--lemma}, cartesian squares commute with colimits, 
    hence any elementary Nisnevich square (\ref{Diagram:NisnevichSquare}) of qcqs objects becomes cartesian after applying $\KHcontu(-)$. Such squares generate the Nisnevich topology \cite[Korollar~A.28]{andreychev-thesis} for qcqs spaces. Therefore, $\KHcontu$ is a Nisnevich sheaf on qcqs spaces.
    
    For any quasi-separated $X \in \Adic_\base^\lft$, one can write it as a \emph{filtered} union of qcqs opens $X = \colim_{i \in I} U_i$ (so the index set $I$ is filtered).
    If $\dim(X) < \infty$, then $\dim(U_i)$ are all bounded by $\dim(X)$.
    We have equivalences in $\Pro^\omega(\Sp^+)$:
    \begin{align*}
        \KHcontu(X) & \simeq \colim_{n\in\Delta\op} \prolim_j \Kcontu(X\skp{\Delta_{\pi^j}^n}) & \textup{(see above)} \\
        & \simeq\colim_{n\in\Delta\op}  \prolim_j \lim_{i \in I\op} \Kcontu(U_i\skp{\Delta_{\pi^j}^n}) & \textup{(analytic descent for $\Kcontu$)} \\
        & \simeq \prolim_j \lim_{i \in I\op} \colim_{n\in\Delta\op} \Kcontu(U_i\skp{\Delta_{\pi^j}^n}) & (\clubsuit) \\
        & \simeq \lim_{i \in I\op} \prolim_j \colim_{n\in\Delta\op} \Kcontu(U_i\skp{\Delta_{\pi^j}^n}) & \\
        & \simeq \lim_{i \in I\op} \KHcontu(U_i) & \textup{(see above)}.
    \end{align*}
    The equivalence $(\clubsuit)$ is again an application of \cite[\textup{Lem.}~2.8]{kst-ii}; here we use that $I\op$ is cofiltered and that $\Kcontu(U_i\skp{\Delta_{\pi^j}^n})$ are uniformly bounded below (Proposition~\ref{prop:weibel-vanising-Kcont}).
\end{proof}

\subsection{$\Arig$-invariance for local Tate pairs}
The assumption \ref{spadesuit-condition-analytic} on the base ring $R$ in the Representability Theorem~\ref{thm:representability} that we shall prove could be dropped if the following holds true.

\begin{conjecture}
\label{conj:A1-invariance-for-Kcont}
    Let $X$ be a regular analytic adic space. Then the canonical map
    \[ \Kcontu(X) \To \Kcontu(X \times \Arig) \]
    is an equivalence in $\Pro(\Sp^+)$.
\end{conjecture}

Since $\Kcontu$ is a sheaf with respect to the analytic topology, the question is of local nature. Hence we may assume that $X=\Spa(A) \coloneqq \Spa(A,A^+)$ for a Tate ring $A=A_0[\varpi^{-1}]$ (with integrally closed subring $A^+$ which does not matter here). Using the cover $\Arig_A = \bigcup_{\rho\in\RR_{>0}} \Spa(A\skp{t}_\rho) \cong \colim_{t\mapsto\varpi t} \Spa(A\tea)$ and the sheaf condition we see that Conjecture~\ref{conj:A1-invariance-for-Kcont} is equivalent to showing that the canonical map
\[ \Kcontu(A) \To \varplim \Kcontu(A\tea) \]
is an equivalence of pro-spectra, i.e. that continuous K-theory is \emph{pro-homotopy invariant} on regular Tate rings.
This has been proven by Kerz--Saito--Tamme under the assumption \ref{dagger-condition} \cite[Prop.~5.14]{kst-i}. Providing further evidence for the conjecture, we establish the case of local Huber pairs, without even assuming regularity.

\begin{defi}[{\cite[\S 5]{huebner-adic}}]
\label{def:local-huber-pair}
    A Huber pair $(A,A^+)$ is called \emph{local}, if $A$ is a local ring with residue field $\kappa$ and if $A^+$ is the preimage of a valuation ring $\kappa^+ \subset \kappa$ such that the valuation on $A$ defined by $\kappa^+$ is continuous.
    A \emph{local Tate pair} is a local Huber pair $(A,A^+)$ such that $A$ is a Tate ring.
\end{defi}

\begin{rem}
\label{rem:local-Huber-pairs}
    Given $x\in X$, the Huber pair $(\Ocal_{X,x},\Ocal^+_{X,x})$, equipped with the colimit topology, is local, but not uniform (i.e.\@ $\Ocal^+_{X,x}$ is not bounded). This can be modified as follows. Let $(A_x,A_x^+)$ be the Huber pair whose underling pair of rings is the same as for $(\Ocal_{X,x},\Ocal^+_{X,x})$, but with a different topology: $A_x^+$ is equipped with the $IA_x^+$-adic topology for some (or equivalently, any) ideal of definition $I \subset \Ocal_X(U)$ for any open neighbourhood $x\in U\subseteq X$ and $A_x$ carries the induced topology.
\end{rem}

\begin{thm}
\label{thm:Kcont-local-Huber-pair}
    Let $(A,A^+)$ be a local Tate pair with pseudo-uniformiser $\varpi$ such that $A^+$ is bounded (i.e.\@ a ring of definition). Then the canonical map
    \[ \Kcontu(A) \To \varplim \Kcontu(A\tea) \]
    is an equivalence of pro-spectra.
\end{thm}
\begin{proof}
Let $\kappa = A/\mathfrak{m}$ be the residue field with valuation ring $\kappa^+ = A^+/\mathfrak{m} \subset \kappa$.
We may assume that $\varpi\in A^+$.

For $R$ being any of the rings in question and any functor $F$ (appropriately defined), let $\Nup F(R) \coloneqq \cofib(F(R) \to F(R[t]))$ and $\Nuphat F(R) \coloneqq \cofib( F(R) \to F(R\tea))$. 
We use that finite limits and finite colimits in pro-categories can be computed levelwise.
Hence the desired assertion is equivalent to $\varplim\Nuphat\Kcontu(A) \simeq 0$.

By Definition~\ref{def:local-huber-pair}, the underlying ring of $A^+$ is a semi-valuation ring in the sense of \cite[Def.~1]{dahlhausen-k-theory-semi-valuation} with valuative ideal $\mfr$. Hence it follows from \cite[Thm.~21]{dahlhausen-k-theory-semi-valuation} that the induced diagram of spectra
\begin{center}
\begin{tikzcd}
\NK(A^+) \arrow[r] \arrow[d] & \NK(A) \arrow[d] \\
\NK(\kappa^+) \arrow[r] & \NK(\kappa)
\end{tikzcd}
\end{center}
is cartesian, hence also in $\Pro(\Sp^+)$.
Since both $\kappa^+$ and $\kappa$ are regular, stably coherent rings, both horizontal maps are equivalences \cite[Thm.~4]{waldhausen-free-1}.
Thus $\NK(A^+ \on \varpi) \simeq 0$.
By Weibel's analytic isomorphism theorem \cite[Thm.~1.3]{weibel-analytic-iso} the canonical map
\[ \NK(A^+\on\varpi) \To \Nuphat\K(A^+\on\varpi) \]
is an equivalence. In particular, $\varplim\Nuphat\K(A^+\on\varpi) \simeq 0$ in $\Pro(\Sp^+)$. By Lemma~\ref{lem:A1-invariance-for-models} below and the fibre sequence
\[ \varplim\Nuphat\K(A^+\on\varpi) \To \varplim\Nuphat\Kcontu(A^+) \To \varplim\Nuphat\Kcontu(A), \]
we conclude the desired assertion.
\end{proof}

\begin{lem}
\label{lem:A1-invariance-for-models}
    Let $R$ be a $\varpi$-adic ring (for some $\varpi\in R$) and let $X$ be an $R$-scheme. Then the canonical map
    \[ \Kcontu(X) \To \varplim\Kcontu(X\tea) \]
    is an equivalence in $\Pro(\Sp)$ where $X\tea \coloneqq X\times_{\Spec(R)}\Spec(R\tea)$.
\end{lem}
\begin{proof}
    Recall that $\Kcontu(X) = \prolim_{n\geq 1}\K(X/\varpi^n)$ where $X/\varpi^n = X \times_{\Spec(R)}\Spec(R/\varpi^n)$. The composition
    \[ \icolimvarpi\,X/\varpi^n\tea \To[p] X/\varpi^n \To[\sigma] \icolimvarpi\,X/\varpi^n\tea \]
    given by the levelwise projections $p$ and zero-sections $\sigma$ is an isomorphism in the category $\Ind(\Sch_R)$; indeed, this morphism is levelwise given by $t\mapsto \varpi^nt$ and hence equivalent to the identity. As always $p\circ\sigma = \id_X$, we are done.
\end{proof}

\section{Representability for analytic K-theory}
\label{sec:representability}
\subsection{Unstable case}
\label{subsec:unstable-representability}

In this subsection, we want to prove Theorem \ref{thm:representability-intro}. Specifically, we want to prove that $\KHcontu$ is representable in $\RigSH(\base)$. Crucially, we will use the identification with analytic K-theory $\Kanu$ (Proposition~\ref{prop:A1-localisation}). This identification allows us to use results of Kerz--Saito--Tamme, which we will make apparent. 

Let $\base = \Spa(R, R^+)$ be an adic space such that \emph{$R$ admits a ring of definition which is noetherian and of finite dimension}. Under the assumption \ref{spadesuit-condition-analytic}, the analytic K-theory $\Kanu \colon \Adic_\base^{\lft,\opp} \to \Pro^\omega(\Sp^+)$ is an $\Arig$-invariant Nisnevich sheaf, see the preceding section. Since both $\tau_{\geq 0}$ and $\Omega^\infty$ are right adjoint functors, the induced functor 
\[ \Adic_\base^{\lft,\opp}\to\Pro^\omega(\Sp^+) \To \Pro^\omega(\Sp) \To[\tau_{\geq 0}] \Pro^\omega(\Sp_{\geq 0}) \To[\Omega^\infty] \Pro^\omega(\Spc) \]
is again an $\Arig$-invariant Nisnevich sheaf, hence an object in $\RigH(\base,\Pro^\omega(\Spc))$ which we denote by $\Omega^\infty\Kan_{\geq 0}$. We would like to identify this object with more concrete objects. Unfortunately, we only can do this under the assumption \ref{spadesuit-condition-analytic} in the beginning of Section~\ref{sec:k-theory}.
Note that smooth affinoid $R$-algebras $A$ satisfying \ref{dagger-condition} form a basis of $\Adicsm_R$ and for any such $A$ we have an equivalence \cite[Lem.~7.5]{kst-i}
\[
\Omega^\infty\tau_{\geq 1}\Kanu(A) \simeq \prolim_j\, \BGL(A\skp{\Delta_{\pi^j}}).
\]
This equivalence enables us to identify analytic K-theory similarly as the classical identification of algebraic K-theory with $\ZZ\times\BGL$ in the Morel-Voevodsky category \cite[\S 4.3, Prop.~3.10]{MV1}.

\begin{thm}[Representability]
\label{thm:representability}
Let $\base = \Spa(R, R^+)$ be an adic space such that $R$ admits a ring of definition which is noetherian and of finite dimension.
We assume \ref{spadesuit-condition-analytic}.
Then there is a canonical equivalence
\[ \Omega^\infty\Kanu_{\geq 0} \simeq \Lmot(\ZZ\times\BGL) \]
in the category $\RigH(\base,\Pro(\Spc^+))$.
In particular, for every $X \in \Adicsm_\base$ there is a functorial equivalence
\[ \Omega^\infty\gamma^\omega\KHcontu_{\geq 0}(X)\simeq  \Omega^\infty\gamma^\omega\Kanu_{\geq 0}(X) \simeq \Homline_{\Cond^\omega(\Spc)}(\Lmot X,\Lmot(\ZZ\times\BGL)) \]
in the category $\Cond^\omega(\Spc)$; here the right-hand side denotes the enriched Hom-space as a left $\Cond^\omega(\Spc)$-module, see \textup{Appendix~\ref{app.yoneda}}.
\end{thm}

\begin{rem}
Let us explain the notation $\RigH(\base,\Pro(\Spc^+))$ that appears in the statement of Theorem~\ref{thm:representability}.

(1)
Analogously as in \cite[\S 2.2]{kst-i} or in the proof of Lemma~\ref{lem-t-structure-pro}, the inclusion $\iota\colon\Spc^+\inj \Spc$ induces an adjunction
$$
 \begin{tikzcd}
	\iota^{*}\colon \Pro(\Spc)\arrow[r,"",shift left = 0.3em]&\arrow[l,"",shift left = 0.3em]\Pro(\Spc^+)\colon  \Pro(\iota)
\end{tikzcd}
$$
whose right-adjoint $\Pro(\iota)$ is fully faithful \cite[Ex.~A.8.1.8, Prop.~A.8.1.9]{SAG}; indeed, the inclusion $\iota$ preserves finite limits since $n$-truncated spaces are closed under finite limits.

(2)
In \cite{DY}, the category of coefficients is assumed to be presentable for simplicity.
However, for the coefficients in $\Pro(\Dcal)$, where $\Dcal$ is a small category with finite limits, the $\AA^1$-localisation $\L_{\AA^1}$, Nisnevich localisation $\L_\Nis$ and motivic localisation $\Lmot$ all exist. Indeed, we may pass to the opposite coefficient category $\Pro(\Dcal)\op \simeq \Ind(\Dcal\op)$, which is a presentable category (we have taken cardinal cut-off to make $\Sp$ small), in order to apply the Adjoint Functor Theorem \cite[Cor.~5.5.2.9 (1)]{HTT} to obtain the corresponding colocalisations (also need cardinal cut-off on $\Adicsm_\base$, which is fine).\footnote{%
    Indeed, we need to see that, given any family $\Kscr$ of (small) cocone diagrams in a small category $\Ccal$, and given any presentable category $\Ecal$, the full subcategory $\Fun^\Kscr(\Ccal, \Ecal) \subset \Fun(\Ccal,\Ecal)$ that consists of functors sending cocones in $\Kscr$ to colimit cocones in $\Ecal$ is presentable.
    (For example, if $(\Ccal, \tau)$ is a site, we may take $\Kscr$ to be the collection of augmented simplicial diagrams for $\tau$-covers; when considering $\Arig$-invariance on $\Adic_\base$, we may take $\Kscr$ to be the collection of morphisms $\Arig_X \to X$ for $X \in \Adic_\base$.)
    In order to show the claim above, we follow the arguments in the proof of \cite[Prop.~5.5.3.8]{HTT}: it suffices to consider a single cocone diagram $K^\rhd$ and then take the intersection of the resulting subcategories. The stability of the full subcategory $\Fun^{\{K^\rhd\}}(\Ccal, \Ecal)$ under colimits is clear; to prove its accessibility, observe that it is the pullback of $\Fun(\Ccal, \Ecal) \to \Fun(K^\rhd, \Ecal) \leftarrow \Fun^{\{K^\rhd\}}(K^\rhd, \Ecal)$, hence the accessibility follows from that of $\Fun^{\{K^\rhd\}}(K^\rhd, \Ecal)$ as in loc.\@ cit.
}%
Nevertheless, the same description of motivic realisation $\Lmot$ as the sequential colimit $\Lmot^\omega \simeq \colim(\L_\Nis \to \L_{\AA^1} \L_\Nis \to \L_\Nis \L_{\AA^1} \L_\Nis \to \L_{\AA^1} \L_\Nis \L_{\AA^1} \L_\Nis \to \cdots)$ \cite[\S 3.2]{DY} may not hold; indeed, $\Lmot^\omega$ is $\AA^1$-local as can be checked using the subsequences started with $\L_{\AA^1}$, but it may not be Nisnevich local unless we restrict to the finitary subsite $\Adicsmqcqs_\base \subset \Adicsm_\base$.

Alternatively, one might want to take $\Pro^\omega(\Spc)$ as a category of coefficient. However, the category $\Pro^\omega(\Spc)$ does not have uncountable limits, but we have covers of cardinality strictly larger than $\omega$ on the site $\Adicsm_\base$. Replacing the latter site by the finitary site $\Adicsmqcqs_\base$, then we can consider $\Pro^\omega(\Spc)$ as the coefficient category, at the price of adjusting the definition of $\RigH(\base,\Pro^\omega(\Spc))$. 
However, then $\Arig_\base$ does not lie in $\Adicsmqcqs_\base$, so the $\Arig$-invariance should be modified to the pro-$\Brig$-invariance.
\end{rem}

\begin{proof}[Proof of Theorem~\ref{thm:representability}]
First, we construct a map. For every affinoid algebra $A$ we have a map 
\[ \ZZ(A)\times\BGL(A) \To \Omega^\infty\K_{\geq 0}(A) \To \Omega^\infty\Kanu_{\geq 0}(A) \]
which is functorial in $A$, hence inducing a map $\ZZ\times\BGL\to\Omega^\infty\Kanu_{\geq 0}$ on the category of presheaves on affinoid spaces. Thus we obtain an induced map $\Lmot(\ZZ\times\BGL) \to \Omega^\infty\Kanu_{\geq 0}$ in $\RigH(\base)$. 

Using the fibre sequence \eqref{fibre-sequence-1-connective-cover} we can separately investigate the induced maps $\Lmot\BGL\to\Omega^\infty\Kanu_{\geq 1}$ on 1-connective covers and $\Lmot\ZZ\to\Kanu_0$ in degree zero. On 1-connective covers, we compute for every affinoid algebra $A$ as follows:\footnote{%
    Note that $\BGL$ arises from sheafifying the composition of the presheaf $\Adic_\base^{\lft,\opp} \to \Spc,\,X\mapsto\BGL(\Ocal_X(X)),$ with the inclusion $\Spc\inj\Pro(\Spc)$.
}
\begin{align*}
(\L_{\Arig}\BGL)(\Spa(A)) &\simeq \colim_{n\in\Delta\op}\, \BGL(\Delta^{\an, n}_A)\\
&\simeq \colim_{n\in\Delta\op}\, \prolim_j \BGL(A\skp{\Delta^n_{\pi^j}}).\\
(\clubsuit) &\simeq \prolim_j\, \colim_{n\in\Delta\op} \BGL(A\skp{\Delta^n_{\pi^j}})\\
&\simeq \prolim_j\, \BGL(A\skp{\Delta_{\pi^j}}) \\
&\simeq \Omega^\infty\tau_{\geq 1}\Kanu(A),
\end{align*}
where the weak equivalence $(\clubsuit)$ follows from Lemma~\ref{lem:pro-spaces-universal-colimits}  and the last equivalence is a consequence of \ref{spadesuit-condition-analytic} as explained above the statement of the theorem.
In degree zero we have that $\Kanu_0(A) \cong \Kcontu_0(A) \cong \K_0(A) \cong \ZZ$ for local rings $A$, which concludes the first claim.

In order to prove the $\Cond^\omega(\Spc)$-analogue of the first claim, we use similar arguments but with some modifications.
First, observe that $\gamma^\omega$ commutes with $\tau_{\geq 1}$, $\tau_{\geq 0}$ and $\pi_0$ for $\Kanu(A)$, because the pro-systems $\pi_1 \Kanu(A)$ and $\pi_0 \Kanu(A)$ are Mittag-Leffler; see \cite[B.2.5]{zhou-these}.
Besides, the weak equivalence $(\clubsuit)$ this time becomes an equivalence in $\Cond^\omega(\Spc)$\footnote{%
    Notice that the naive analogue of \cite[Lem.~2.8]{kst-ii} for light condensed spaces instead of pro-spaces is false.
} but will be justified alternatively by the following claim:
for any sequential inverse system (indexed by $j \in \NN$) of simplicial objects $X_{\bullet, j}$ of \emph{connected} spaces, we have
\[ \colim_{\Delta\op} \lim_j X_{\bullet, j} \xrightarrow{\simeq} \lim_j \colim_{\Delta\op} X_{\bullet, j} \] in $\Spc$;
by writing sequential limits as a pullback diagram of countable products, this follows from the fact that
geometric realisation of connected spaces commutes with countable products (see for example \cite[Prop.~2.1.14 (i)]{balderrama-pstragowski-2026-maclev}) and pullbacks (Bousfield-Friedlander criterion, see \cite[Thm.~B.4]{BousfieldFriedlander1978}, or \cite{Rezk2014BousfieldFriedlander} for a modern treatment).
The replacement of $(\clubsuit)$ in the setting of condensed spaces, that is,
\begin{equation}
    \label{eq:geomRealCommutesWithLimitsInCondSpc}
    \colim_{\Delta\op} \lim_j \underline{X_{\bullet, j}} \xrightarrow{\simeq} \lim_j \colim_{\Delta\op} \underline{X_{\bullet, j}}
\end{equation}
in $\Cond^\omega(\Spc)$ resp. $\Cond_\kappa(\Spc)$, follows from the above claim.
Indeed, for any profinite set $S = \lim_i S_i$, from the claim we obtain
\begin{align}
    \notag
    \colim_{\Delta\op} \bigl( \lim_j \underline{X_{\bullet, j}}(S) \bigr)& \xrightarrow{\simeq} \lim_j \colim_{\Delta\op} \underline{X_{\bullet, j}}(S) \\
    \label{eq:claimKST2.8CondSpc(2)}
    & \simeq \lim_j \bigl(\colim_{\Delta\op} \underline{X_{\bullet, j}} \bigr)(S) \\
    \label{eq:claimKST2.8CondSpc(3)}
    & \simeq \Bigl(\lim_j \bigl(\colim_{\Delta\op} \underline{X_{\bullet, j}} \bigr)\Bigr)(S).
\end{align}
For the second equivalence (\ref{eq:claimKST2.8CondSpc(2)}), recall that colimits in $\Cond^\omega(\Spc)$ and $\Cond_\kappa(S)$ are computed as the hypersheafification of pointwise colimits. 
Using the definition $\underline{X_{\bullet, j}}(S) = \colim_i \prod_{S_i} X_{\bullet, j}$ and that both colimits and finite products commute with geometric realisation in $\Spc$, we get
\[ \colim_{\Delta\op} \underline{X_{\bullet, j}}(S) = \colim_i \prod_{S_i} \colim_{\Delta\op} X_{\bullet, j} = \underline{\colim_{\Delta\op} X_{\bullet, j}}(S); \]
this already satisifes hyperdescent with respect to $S$ by \cite[Lem.~4.14]{clausen2025duality}, which is the $\Spc$-analogue of Lemma~\ref{lem:hyperdescentDiscreteCondensed}. Hence the above formula computes the colimit $\colim_{\Delta\op} \underline{X_{\bullet, j}}$ in $\Cond^\omega(\Spc)$ and $\Cond_\kappa(S)$.
For the last equivalence (\ref{eq:claimKST2.8CondSpc(2)}), it suffice to recall that limits of condensed spaces can be computed pointwise on profinite sets.
Moreover, likewise we have $\lim_j \underline{X_{\bullet, j}}(S) \simeq \bigl( \lim_j \underline{X_{\bullet, j}} \bigr)(S)$, thus we deduce from above that
\[ \colim_{\Delta\op} \bigl( \lim_j \underline{X_{\bullet, j}} \bigr) (S) \xrightarrow{\simeq} \Bigl(\lim_j \bigl(\colim_{\Delta\op} \underline{X_{\bullet, j}} \bigr)\Bigr)(S).\]
Again, as the right hand side is already a hypersheaf, so is the left hand side, hence it computes the colimit $\colim_{\Delta\op} (\lim_j \underline{X_{\bullet, j}})$, which concludes the proof of (\ref{eq:geomRealCommutesWithLimitsInCondSpc}).

It remains to show that
$$ 
\Omega^\infty\gamma^\omega\Kanu_{\geq 0}(X) \simeq \Homline_{\Cond^\omega(\Spc)}(\Lmot X,\Lmot(\ZZ\times\BGL)).
$$
For this let us note that by construction, we have that $\Lmot$, as a localisation from\linebreak $\Fun((\Adicsm_\base)\op,\Cond^{\omega}(\Spc))\to \RigH(\base,\Cond^{\omega}(\Spc))$, is equivalently given by base\linebreak change of the motivic localisation functor $\PSh(\Adicsm_\base)\to \RigH(\base)$. In particular, it is compatible with the $\Cond^{\omega}(\Spc
)$-module structure. Thus, we have 
$$
\Homline_{\Cond^\omega(\Spc)}(\Lmot X,\Lmot(\ZZ\times\BGL))\simeq \Homline_{\Cond^\omega(\Spc)}(X,\Lmot(\ZZ\times\BGL)),
$$
where we view $X$ as a functor in $\Fun(\Adicsm_\base,\Cond^\omega(\Spc))$ via the composition of the Yoneda functor and the constant functor $\Spc\to \Cond^\omega(\Spc)$.
Our results above together with the enriched Yoneda lemma finish the proof  (cf.\@ Lemma \ref{lem-yoneda}).
\end{proof}

\begin{rem}[Étale case]
\label{rem:representability-etale-case}
    When $R$ is a nonarchimedean field $k$ which is discretely valued, the étale analogue of Theorem~\ref{thm:representability} holds and is unconditional. Indeed, the assumption \ref{spadesuit-condition-etale-k} always holds by Remark~\ref{rem:key-assumption} (3), and similar arguments as above work (see \cite[Thm.~B.3.3]{zhou-these} for details).
    Furthermore, since $\Kcontu(A) \simeq \Kanu(A)$ for all smooth affinoid algebras $A$ over $k$ satisfying \ref{dagger-condition} \cite[Thm.~1.3]{kst-i}, we obtain a representability result for the étale sheafification of $\Kcontu$.
\end{rem}

\begin{defi}
For $n\geq d\geq 1\in\ZZ$ denote by $\Grass_{d,n}$ be the Grassmannian $R$-scheme representing submodules $\Ecal\subseteq\Ocal^n$ such that the quotient sheaf $\Ocal^n/\Ecal$ is locally free of rank $n-d$ \cite[(8.4)]{GW1}. We write $\Grass_d \coloneqq \colim_n\Grass_{d,n}$ and $\Grass \coloneqq \colim_d\Grass_d$. Denote by $\Grass_{d,n}^\an$, $\Grass_d^\an$, and $\Grass^\an$ the respective analytifications.

Note that $\Grass_{d,n}^\an$ coincides with the representable functor $\Grass_{d,n}^\rig$ representing $d$-dimensional subspaces of $\kappa(x)^n$ on points $\Spa(\kappa(x))\to B$.
\end{defi}

\begin{prop}[Arndt, Sigloch {\cite[Prop.~5.34]{sigloch-thesis}}] \label{Prop:BGL=Grass}
For every $n\geq 1$, there is an equivalence $\BGL_n\simeq\Grass_n^\an$ in $\RigH(\base)$. Consequently, there is an equivalence $\BGL\simeq\Grass^\an$.
\end{prop}

\begin{proof}
    The same proof in loc.\@ cit.\@ works verbatim in the relative setting over $\base$.
\end{proof}

\begin{rem}
\label{rem-analyitification-monoidal}
    Let us consider the constant functor $\const\colon \Spc\to\Cond^{\omega}(\Spc)$. This is a map of presentable categories by Lemma~\ref{condensed-objects-presentable--lem}.
    As a left-adjoint, $\const$ is colimit preserving. Moreover, $\Cond^{\omega}(\Spc)$ is the sheaf category associated to a Grothendieck topology. In particular, the localisation functor $\PSh(\ProFin^{\omega})\rightarrow \Cond^{\omega}(\Spc)$ is left exact. Therefore, $\const$ also preserves finite products and in particular is symmetric monoidal. Thus, the constant functor $\const$ is a map in $\CAlgPr$. The analytification functor \cite[\S 3.3]{DY} yields a symmetric monoidal colimit preserving functor\footnote{%
        We suspect that the functor $\RigH(\base^{\an})\rightarrow \RigH(\base^{\an},\Cond^{\omega}(\Spc))$ induced via base change with $\const$ is equivalent to the functor $\widehat{\const}_*$ constructed in \cite[Rem.~3.14]{DY}. However, we did not find a reference.
    } 
    $$
        \H(\Spec(R),\Spc)\To[\an] \RigH(\base)\To[\id\otimes\const] \RigH(\base,\Cond^{\omega}(\Spc)).
    $$
\end{rem}

\begin{cor}
\label{cor-algebraic-analytic-K}
We assume \ref{spadesuit-condition-analytic}. Then the image of connective algebraic K-theory under the analytification functor $\alpha \colon \H(\Spec(R),\Spc)\to\RigH(\base,\Cond^\omega(\Spc))$ identifies with connective analytic K-theory.
\end{cor}
\begin{proof}
In $\H(\Spec(R),\Spc)$, algebraic K-theory coincides with $\Lmot(\ZZ\times\BGL)$ \cite[\S 4.3, Thm.~3.13]{MV1}.\footnote{
    In order to apply the cited result, notice that under the assumption \ref{dagger-condition} for $A = R$, the ring $R$ is regular.
}
In $\RigH(\base,\Cond^\omega(\Spc))$, connective analytic K-theory coincides with $\Lmot(\ZZ\times\BGL)$ (Theorem~\ref{thm:representability}). The identification $\alpha(\ZZ)\simeq\ZZ$ is clear. The identification $\alpha(\BGL)\simeq\BGL$ follows since $\BGL\simeq\Grass$ on both sides by \cite[\S 4.3, Prop.~3.7]{MV1} and Proposition~\ref{Prop:BGL=Grass}.
\end{proof}

\subsection{Analytic Bass delooping}
\label{sec:analytic-bass-construction}

Nonconnective algebraic K-theory is obtained from connective algebraic K-theory by Bass delooping within the category of spectra \cite[\S 6]{tt90}. The construction of the delooping uses the cover of $\PP^1$ by two copies of $\Arig$ whose intersection is $\Gm$. Such a delooping construction can also be done globally within the category of presheaves of spectra \cite[\S 2]{cisinski-descente}. For analytic K-theory of rigid spaces, an analytic analogue of the Bass construction is performed by Kerz-Saito-Tamme within the category of pro-spectra and using the cover of $\PP^1$ by two copies of $\Brig$ whose intersection we denote by $\Hm$. Putting 1-connective analytic K-theory into this machine, one obtains, by definition, nonconnective analytic K-theory, see subsection~\ref{subsec:analytic-k-theory}.
In this subsection, we examine such an analytic Bass delooping within the category $\PSh(\Rig_\base,\Pro(\Sp^+))$ using the alternative cover of $\PP^1$ by two copies of $\Arig$ whose intersection is $\Gm$, the analytification of $\Spec(R[t,t^{-1}])$. It turns out that the $\Arig$-delooping is equivalent to the $\Brig$-delooping so that  analytic K-theory also satisfies the Bass Fundamental Theorem for $\Arig$ (Corollary~\ref{bass-fundamental-theorem-for-Kan--cor}). Our exposition follows closely the description of the $\Brig$-analytic Bass delooping by Kerz-Saito-Tamme \cite[\S 4.4]{kst-i}.

\vspace{6pt}\noindent\textbf{Notation.} We fix an embedding $\Arig \inj \PP^1$ and write $\PP^1\setminus\Arig\eqqcolon\{\infty\}$ and $\AA^{-1}\coloneqq\PP^1\setminus\{O\}$ where $O\in\Arig$ denotes the origin. Thus $\Gm=\Arig\cap\AA^{-1}$ as subspaces of $\PP^1$. Analogously, we write $\Brig\coloneqq\Spa(R\skp{t},R^+\skp{t}) \inj \PP^1$ and $\BB^{-1}\coloneqq\Spa(R\skp{t^{-1}},R^+\skp{t^{-1}})$, so that $\Brig\cap\BB^{-1}=\Spa(R\skp{t,t^{-1}},R^+\skp{t,t^{-1}})\eqqcolon\Hm$.
Finally, let $\Dcal$ be a stable bicomplete category; thus it carries an essentially unique structure of a module over the category of spectra \cite[4.8.2.18]{HA}.

\begin{defi} \label{bass-delooping--defi}
Let $E\in\PSh(\Adic_\base^\lft,\Dcal)$ be a presheaf. We define a presheaf $\Gamma_{\Arig}E$  via
\[ \Gamma_{\Arig}E(X) \coloneqq E(X\times\Arig) \sqcup_{E(X)} E(X\times\AA^{-1}). \]
The commutativity of the square 
\carre{\Gmbase}{\Arig_\base}{\AA^{-1}_\base}{\base}
yields a map $\Gamma_{\Arig}E(X) \to E(X\times \Gm)$ so that we can define a presheaf $\Lambda_{\Arig}E$ via
\[ \Lambda_{\Arig}E(X) \coloneqq \fib\bigl(\Gamma_{\Arig}E(X) \to E(X\times\Gm)\bigr). \]

Replacing $\Arig$, $\AA^{-1}$, and $\Gm$ with $\Brig$, $\BB^{-1}$, and $\Hm$, respectively, we get the presheaves $\Gamma_{\Brig}E$ and $\Lambda_{\Brig}E$ as defined in the affinoid setting by Kerz-Saito-Tamme \cite[\S 4.4]{kst-i}. 
Note that the inclusions $\BB^1\inj \Arig$ and $\BB^{-1} \inj \AA^{-1}$ induce a canonical map
\[ \theta_{E(X)} \colon \Lambda_{\Arig}E(X) \to \Lambda_{\Brig}E(X) \]
which is functorial in $X\in\Adic_\base^\lft$.
\end{defi}

\begin{lem} \label{bass-constructions-agree--lem}
Given an analytic sheaf $E\in\Sh(\Adic_\base^\lft,\Dcal)$, there are equivalences
\[ \Lambda_{\Arig}E(X) \simeq \Sigma\fib\bigl(E(X) \to E(X\times\PP^1)\bigr) \simeq \Lambda_{\Brig}E(X) \]
which are functorial in $X\in\Adic_\base^\lft$.
Moreover, the canonical map $\theta_{E} \colon \Lambda_{\Arig}E \to \Lambda_{\Brig}E$ is an equivalence.
\end{lem}
\begin{proof}
We have a map of cartesian squares
$$ \begin{tikzcd}
    E(X) \arrow[r] \arrow[dd] 
    & E(X\times\Arig) \arrow[dd]
    &&
    E(X\times\PP^1) \arrow[r] \arrow[dd] 
    & E(X\times\Arig) \arrow[dd]
    \\ && \overset{\alpha}{\longrightarrow} \\
    E(X\times\AA^{-1}) \arrow[r] 
    &\Gamma_{\Arig}E(X) 
    &&
    E(X\times\AA^{-1}) \arrow[r] 
    & E(X\times\Gm),
\end{tikzcd} $$
where the cartesianness of the right one is an instance of the sheaf condition on $E$.
Denoting $\Phi E(X) \coloneqq \fib\bigl(E(X) \to E(X\times\PP^1)\bigr)$, the fibre of the map $\alpha$ is the cartesian square
\carre{\Phi E(X)}{\ast}{\ast}{\Lambda_{\Arig}E(X).}
This shows the first asserted equivalence. The same argument with $\Gamma_{\Brig}E$ and $\Lambda_{\Brig}E$ yields the second asserted equivalence. 
The composite equivalence $\Lambda_{\Arig}E(X) \simeq \Lambda_{\Brig}E(X)$ is identified with $\theta_{E(X)}$ by construction of the latter, and this is natural in $X$. Hence $\theta_E$ is an equivalence.
\end{proof}

\begin{defi} \label{bott-element--defi}
Let $E\in\PSh(\Adic_\base^\lft,\Dcal)$ carrying a $\K$-module structure.\footnote{%
    To be more precise, when restricted to affine spaces, $E$ carries a $\K$-module structure.
}
We fix a map $t\colon\SS\to\Omega\K(\ZZ[t,t^{-1}])$ representing the class $t\in\K_1(\ZZ[t,t^{-1}])$. Then we get an induced map
\[ \lambda_{\Arig} \colon E(X) \To[-\cup t] \Omega E(X\times\Gm) \To[\partial] \Lambda_{\Arig}E(X) \]
where the first map comes from the cup product on K-theory and the second map is from the shifted fibre sequence defining $\Lambda_{\Arig}E$ above.
Analogously, there is a map $\lambda_{\Brig} \colon E \to \Lambda_{\Brig}E$ which is compatible with the map $\theta_E$. 
When $E$ is an analytic sheaf, $\theta_E$ is an equivalence by Lemma~\ref{bass-constructions-agree--lem}, hence we just write $\lambda \colon E \to \Lambda E$ where $\Lambda E \coloneqq \Lambda_{\Arig}E \simeq\Lambda_{\Brig}E$.
\end{defi}

\begin{defi} \label{analytic-bass-construction--defi}
Let $E\in\PSh(\Adic_\base^\lft,\Dcal)$ carrying a $\K$-module structure. We define the \emph{analytic Bass construction of} $E$  as the colimit
\[ E^\Bass \coloneqq \colim\bigl( E \To[\lambda_{\Arig}] \Lambda_{\Arig} E \To[\lambda_{\Arig}] \Lambda_{\Arig}^2 E \To[\lambda_\Arig] \ldots \bigr) \]
within the category $\PSh(\Adic_\base^\lft,\Dcal)$.
\end{defi}

As a formal consequence from our construction we get the following.

\begin{prop} \label{bass-fundamental-theorem-abstractly--prop}
Let $\Dcal\in\{\Sp,\Pro(\Sp),\Cond^\omega(\Sp)\}$.
Let $E\in\PSh(\Adic_\base^\lft,\Dcal)$ carrying a $\K$-module structure and assume that the map $\lambda_{\Arig} \colon E\to \Lambda_{\Arig} E$ is an equivalence. Then for every $n\in\ZZ$ and every $X \in \Adic_\base^\lft$ we have an exact sequence
\[ 0 \To E_n(X) \To E_n(X\times\Arig) \oplus E_n(X\times\AA^{-1}) \To[\pm] E_{n}(X\times\Gm) \To[\partial] E_{n-1}(X) \To 0 \]
in the category $\Dcal^\heartsuit$.
Furthermore, the map $\partial$ is split by the map
\[
-\cup t \,\colon\, E_{n-1}(X) \To E_n(X\times\Gm) 
\]
induced by the cup product with $t$ from \textup{Definition~\ref{bott-element--defi}}.
\end{prop}
\begin{proof}
We follow verbatim the proof for the $\Brig$-Bass Fundamental Theorem \cite[Prop.~4.13]{kst-i}. By definition of $\Gamma_{\Arig}$ and since the projection $X\times \Arig \to X$ admits a split, we have short exact sequences
\[ 0 \To E_n(X) \To E_n(X\times\Arig) \oplus E_n(X\times\AA^{-1}) \To \Gamma_{\Arig}E_n(X) \To 0 \]
and by defintion of $\Lambda_{\Arig}$ we have a long exact sequence
\[ \ldots \To \Lambda_{\Arig}E_n(X) \To \Gamma_{\Arig}E_n(X) \To E_n(X\times\Gm) \To[\partial] \Lambda_{\Arig}E_{n-1}(X) \To \ldots. \]
Since we have assumed the map $\lambda_{\AA^1} \colon E\to \Lambda_{\Arig}E$  to be an equivalence, the boundary morphisms $\partial$ are surjective so that we obtain short exact sequences
\[ 0 \To \Gamma_{\Arig}E_n(X) \To E_n(X\times\Gm) \To[\partial] \Lambda_{\Arig}E_{n-1}(X) \To 0. \]
Now the claim follows by splicing together these two short exact sequences.
\end{proof}

\begin{cor} \label{bass-fundamental-theorem-for-Kan--cor}
For every $n\in\ZZ$ and every $X\in\Adic_\base^\lft$ there is an exact sequence
\[ 0 \To \Kanu_n(X) \To \Kanu_n(X\times\Arig) \oplus \Kanu_n(X\times\AA^{-1}) \To[\pm] \Kanu_{n}(X\times\Gm) \To[\partial] \Kanu_{n-1}(X) \to 0 \]
of pro-abelian groups where the map $\partial$ has a split.
\end{cor}
\begin{proof}
This follows from Proposition~\ref{bass-fundamental-theorem-abstractly--prop} since the map $\lambda_{\BB^1} \colon \Kan \to \Lambda\Kan$ is an equivalence \cite[Cor.~2.6]{kst-ii} and since we have $\lambda_{\BB^1} \simeq \lambda_{\AA^1}$ (Definition~\ref{bott-element--defi}) as $\Kanu$ is an analytic sheaf by Theorem~\ref{thm:KST-Kan-sheaf}.
\end{proof}

\begin{cor}
\label{cor-K-P1-spectrum}
  There exists a $\PP^1$-action on $\ZZ\times\BGL$ such that the induced $\PP^1$-spectrum
    $$
    \KGL^{\an}\coloneqq (\ZZ\times\BGL,\ZZ\times\BGL,\dots) 
    $$
    is a commutative algebra object in $\RigSH(\base,\Cond^\omega(\Sp))$. Furthermore, under assumption \ref{spadesuit-condition-analytic} the $\PP^1$-spectrum $\KGL^{\an}$ represents (non-connective) analytic K-theory.
\end{cor}
\begin{proof}
    Classically, it is known that $\ZZ\times \BGL\in \H(\base)$, admits the structure of a $\PP^1$-spectrum \cite[\S 6.2]{voevodsky-icm-1998}, \cite{PPR-KGL}, \cite[\S 3.2]{MotCoh}.
    This action yields a $\PP^1$-spectrum $\KGL\coloneqq(\ZZ\times \BGL,\ZZ\times \BGL,\dots)$ representing homotopy invariant K-theory  \cite[Thm.~2.20]{cisinski-descente} and carrying an $\EE_\infty$-ring structure \cite{naumann-sptzweg-oestvaer}.
    According to the discussion in \cite[\S 4.4]{DY}, we see that the analytification of $\KGL$ yields a $\PP^1$-spectrum $\KGL^{\an}$ inside $\RigSH(\base,\Cond^{\omega}(\Sp))$ and the $\EE_\infty$-structure is preserves, too.
    Furthermore, by Lemma \ref{bass-constructions-agree--lem} and the construction of homotopy invariant K-theory ring spectrum via the Bass-construction \cite{k-book}, we see that $\KGL^{\an}$ represents (non-connective) analytic K-theory under assumption \ref{spadesuit-condition-analytic} using Theorem \ref{thm:representability} by analogous arguments as in the algebraic case.
\end{proof}

\begin{appendix}

\section{Localisation sequences of dualisable categories}
\label{sec:dualisable-categories}
This appendix contains preliminaries on localisation sequences, and will be used in Lemma~\ref{lem:CartesianDiagramNuc}. We wanted to apply \cite[Lem.~2.5]{andreychev-thesis}, but its statement is false: for example, the condition (ii) in loc.\@ cit.\@ is insensitive to replacing $G \colon \Dcal \to \Ecal$ by $(G,0) \colon \Dcal \to \Ecal \times \Ecal'$ for any $\Ecal'$ in $\PrLst$. Besides, the notion of localisation sequences depends on the ambient category (e.g.\@ $\PrLst$, $\PrLLst$, etc.), but this was often not spelled out clearly in \cite{andreychev-thesis}. We give a more detailed account of this below.

Recall that a \emph{localisation sequence} in $\PrL$ (resp. $\PrLL$, $\PrLst$, $\PrLLst$, $\Catdual$) is a cofibre sequence $\Acal \to \Bcal \to \Ccal$ with the functor $\Acal \to \Bcal$ being fully faithful.

It turns out that in $\Pr_\stable^{\L\L}$, a localisation sequence is also a fibre sequence according to the following criteria:

\begin{lem}
    \label{lem:fibrecofibre}
    Consider the following sequence in $\PrL$:
    \begin{equation}
        \label{eq:verdierSeq}
        \Acal \xrightarrow{F} \Bcal \xrightarrow{G} \Ccal,
    \end{equation}
    whose right adjoint is denoted by $\Ccal \xrightarrow{G^R} \Bcal \xrightarrow{F^R} \Acal$.
    Consider the following statements:
    \begin{enumerate}[label=\textup{(\roman*)}]
        \item The sequence \textup{(\ref{eq:verdierSeq})} is a cofibre sequence and $F$ is fully faithful.
        \item The sequence \textup{(\ref{eq:verdierSeq})} is a fibre sequence and $G^R$ is fully faithful.
        \item The sequence \textup{(\ref{eq:verdierSeq})} is a fibre-cofibre sequence.
        \item The natural sequence $F F^R \to \id_\Bcal \to G^R G$ gives rise to a fibre-cofibre sequence in $\Bcal$ when applied to any $b \in \Bcal$.
    \end{enumerate}
    We have the following implications:
    \begin{enumerate}
        \item The \textup{(iii)} implies \textup{(i)} and \textup{(ii)}.
        \item If $\Bcal$ is stable and $F^R$ preserves cofibre sequences, then \textup{(i)} implies \textup{(ii)}, \textup{(iii)} and \textup{(iv)}.
        \item If $\Bcal$ is stable and $G$ preserves fibre sequences, then \textup{(ii)} implies \textup{(i)}, \textup{(iii)} and \textup{(iv)}.
        \item If all the categories $\Acal$, $\Bcal$ and $\Ccal$ are stable, then the \textup{(i)}, \textup{(ii)} and \textup{(iii)} are equivalent to each other, and they imply \textup{(iv)}.
        \item Assume that \textup{(\ref{eq:verdierSeq})} is a cofibre sequence. If $F^R$ preserves colimits, then $G^R$ preserves colimits.
        \item Assume that $\Acal$ and $\Bcal$ are stable and that \textup{(i)} holds. If $G^R$ preserves colimits, then $F^R$ preserves colimits.
    \end{enumerate}
\end{lem}

\begin{proof}
    (1)
    It is clear.

    (2)
    Given $b \in \Bcal$, consider the cofibre sequence
    \begin{equation}
        \label{eq:cofibTest}
        F F^R(b) \to b \to b'.
    \end{equation}
    Apply $F^R$ to (\ref{eq:cofibTest}), which preserves cofibre sequences, one obtains a cofibre sequence
    \[F^R F F^R(b) \to F^R(b) \to F^R(b').\]
    By full faithfulness of $F$, we have $\id_\Acal \To[\simeq] F^R F$, hence its cofibre $F^R(b') \simeq 0$. The cofibre sequence hypothesis says that
    \[\Ccal \xrightarrow[\simeq]{G^R} \ker(F^R),\]
    and in particular $G^R$ is fully faithful (hence $G G^R \To[\simeq] \id_\Ccal$) and $G F \simeq 0$.
    One sees that $b' \simeq G^R(c)$ for some $c \in \Ccal$.
    Apply now $G$ to (\ref{eq:cofibTest}), one obtains a cofibre sequence $0 \to G(b) \to G G^R(c)$, so that $G(b) \To[\simeq] c$ by full faithfulness of $G^R$.
    In particular, there is a (natural) cofibre sequence
    \begin{equation}
        \label{eq:fibcofibinB}
        F F^R(b) \to b \to G^R G(b),
    \end{equation}
    which is also a fibre sequence by stability of $\Bcal$.
    
    It remains to check that $\Acal \xrightarrow{F} \ker(G)$ is an equivalence.
    Since $F$ is fully faithful, it suffices to prove that for any $b \in \Bcal$ such that $G(b) \simeq 0$, there exists $a \in \Acal$ such that $b \simeq F(a)$. It is clear since the fibre sequence (\ref{eq:fibcofibinB}) becomes $F F^R(b) \To[\simeq] \fib(b \to 0) \simeq b$.

    (3)
    Similar proof as (2).

    (4)
    It follows formally from (2) and (3).

    (5)
    By assumption, $G^R$ exhibits the fibre of a continuous morphism, hence it is also continuous, since the inclusion $\PrL \subset \widehat{\Cat}_\infty$ creates limits \cite[Proposition~5.5.3.13]{HTT}.

    (6)
    The assumptions imply that (i)--(iv) hold.
    Assume that $G^R$ preserves colimits.
    Using the cofibre sequence from (iv):
    \[F F^R(b_i) \to b_i \to G^R G(b_i)\]
    in $\Bcal$, and the property that $F$, $G$ and $G^R$ all preserve colimits, one obtains a cofibre sequence
    \[F (\colim_i F^R(b_i)) \to \colim_i b_i \to G^R G(\colim_i b_i).\]
    This sequence maps naturally to another cofibre sequence:
    \[F F^R(\colim_i b_i) \to \colim_i b_i \to G^R G(\colim_i b_i).\]
    Since it is also a fibre sequence in $\Bcal$, one obtains
    \[F (\colim_i F^R(b_i)) \To[\simeq] F F^R(\colim_i b_i).\]
    Since $F$ is fully faithful, We conclude that $\colim_i F^R(b_i) \To[\simeq] F^R(\colim_i b_i)$.
\end{proof}

\begin{cor}[cf.\@ {\cite[Lem.~2.5]{andreychev-thesis}}]
    \label{cor:fibrecofibreinPrLL}
    In $\PrLst$ (resp. in $\PrLLst$, resp. in $\Catdual$), consider the following sequence
    \begin{equation}
        \label{eq:fibrecofibreinPrLL}
        \Acal \xrightarrow{F} \Bcal \xrightarrow{G} \Ccal.
    \end{equation}
    The following assertions are equivalent:
    \begin{enumerate}[label=\textup{(\roman*)}]
        \item The sequence \textup{(\ref{eq:fibrecofibreinPrLL})} is a cofibre sequence and $F$ is fully faithful.
        \item The sequence \textup{(\ref{eq:fibrecofibreinPrLL})} is a fibre sequence and $G^R$ is fully faithful.
        \item The sequence \textup{(\ref{eq:fibrecofibreinPrLL})} is a fibre-cofibre sequence.
    \end{enumerate}
    Moreover, in this case, the natural sequence $F F^R \to \id_\Bcal \to G^R G$ gives rise to a fibre-cofibre sequence in $\Bcal$ when applied to any $b \in \Bcal$.
\end{cor}

\begin{proof}
    According to Lemma~\ref{lem:fibrecofibre} (4), the equivalences hold in $\PrL$ for $\Acal, \Bcal, \Ccal$ that are stable.
    Since the inclusion $\PrLst \subset \PrL$ creates limits \cite[Proposition~4.8.2.18]{HA}, and that colimits in $\PrLst$ are computed as the stabilisation of corresponding colimits in $\PrL$, the equivalences hold true also in $\PrLst$.

    Let us then consider the case of $\PrLLst$. Since the functor $\PrLLst \to \PrLst$ commutes with colimits \cite[Proposition~1.7]{efimov-k-theory}, the sequence (\ref{eq:fibrecofibreinPrLL}) in $\PrLLst$ is a cofibre sequence (resp. a localisation sequence) in $\PrLLst$ if and only if it becomes so in $\PrLst$. Also, the fibre of the morphism $G \colon \Bcal \to \Ccal$ in $\PrLLst$ is the largest full stable subcategory of the fibre in $\PrLst$ such that its inclusion into $\Bcal$ is strongly continuous.
    Therefore, the equivalences hold true in $\PrLLst$ as per Lemma~\ref{lem:fibrecofibre} (5)--(6).

    Finally, the statement for $\Catdual$.
    Given a sequence (\ref{eq:fibrecofibreinPrLL}) with $\Bcal$ dualisable. If $F$ is fully faithful (resp. $G^R$ is fully faithful), then $\Acal$ (resp. $\Ccal$) is dualisable, being retract of $\Bcal$ in $\PrLst$ \cite[Proposition~1.16]{efimov-k-theory}.
    We know how to compute cofibre and fibre in $\Catdual$: recall that $\Catdual \subset \PrLLst$ is a full subcategory stable under colimits \cite[Proposition~1.65]{efimov-k-theory}, and that the kernel in $\Catdual$ of a strongly continuous functor $G \colon \Bcal \to \Ccal$ is the largest dualisable full subcategory of its kernel $\ker(G)$ in $\PrLst$ such that the inclusion map $\Acal \to \Bcal$ is strongly continuous \cite[Proposition~1.84]{efimov-k-theory}.
    It follows that the equivalences hold true in $\Catdual$ by Lemma~\ref{lem:fibrecofibre} (5)--(6).
\end{proof}

\begin{rem}
    \label{rem:fibcofib}
    The last paragraph of the proof shows that if (\ref{eq:fibrecofibreinPrLL}) satisfies one of the conditions (i)--(iii) in $\PrLLst$ and if $\Bcal$ is dualisable, then (\ref{eq:fibrecofibreinPrLL}) satisfies the conditions (i)--(iii) also in $\Catdual$ (by the argument of retract).
\end{rem}

\section{Enriched Yoneda}
\label{app.yoneda}

In this section, we want to recall the enriched Yoneda lemma in the case that is of interest for us, following Hinich \cite{Hinich}. We will give a short summary of the constructions.

In the following, we will fix a category $\Ccal$ and a symmetric monoidal presentable category $\Ecal$. Note that the work of Hinich works in greater generality but we restrict to this case.

\begin{notation}
    We denote by $Y\colon \Ccal\op\times \Ccal\rightarrow \Spc$ the functor $(x,y)\mapsto \Hom(x,y)$. We denote the composition of this functor with the unit $\Spc\rightarrow \Ecal$ by $\Ytilde$.
\end{notation}

In \cite[\S 3.1.3]{Hinich} Hinich defines a category of quivers $\Quiv_\Ccal^\BM(\Ecal)$ as an $\infty$-operad\footnote{%
    In loc.\@ cit.\@ Hinich uses a different notion of operads, based on strong approximations of $\infty$-operads. We will ignore this detail as this section stays valid, having \cite[\S 2.7]{Hinich} in mind.
} fibred over the $\infty$-operad of bimodules $\BM$ (the $\infty$-operad with colours $\lbrace a_{+},m,a_{-}\rbrace$). The colours of $\Quiv_{\Ccal}^\BM(\Ecal)$ can be described as
\begin{enumerate}
    \item[$\bullet$] $\Quiv_{\Ccal}(\Ecal)\coloneqq\Quiv_{\Ccal}(\Ecal)_{a_{+}} \simeq \Fun(\Ccal^{op}\times \Ccal,\Ecal)$,
    \item[$\bullet$]  $\Quiv_{\Ccal}(\Ecal)_{m} \simeq \Fun(\Ccal,\Ecal)$, and 
    \item[$\bullet$] $\Quiv_{\Ccal}(\Ecal)_{a_{-}}\simeq\Ecal$.
\end{enumerate}

Let us remark that since $\Ecal$ is symmetric monoidal, we have an equivalence categories 
$$
    \Quiv_{\Ccal}(\Ecal)\simeq \Quiv_{\Ccal\op}(\Ecal).
$$
In particular, the functor $\Ytilde$ defines an associative algebra object $\Ytilde\op$ inside $\Quiv_{\Ccal\op}(\Ecal)$ via the above equivalence in the sense of \cite{Hinich}. Thus, $\Ytilde\op$ acts on $\Fun(\Ccal\op,\Ecal)$ from the left. 
\begin{defi}[\protect{\cite[\S 6.2.2]{Hinich}}]
    We define the category of $\Ecal$-presheaves on $\Ccal$ as 
    $$
        P_{\Ecal}(\Ytilde) \coloneqq \LMod_{\Ytilde\op}(\Fun(\Ccal\op,\Ecal)).
    $$
\end{defi}
There is a lot to digest in this definition of $\Ecal$-presheaves. But in our context we can greatly simplify the definition. 

First note $\Quiv_{\Ccal\op}(\Ecal)\simeq \Quiv_{\Ccal}(\Ecal)^\rev$,
where the superscript $\rev$ denotes the reversed monoidal category in the sense of \cite[Rem. 4.1.1.7]{HA} (this follows from the construction \cite[\S 6.2.1]{Hinich}). Further, by \cite[\S 4.7.3]{Hinich}, we have that $\Ytilde$ is the unit object inside  $\Quiv_{\Ccal}(\Ecal)$ and thus $Y\op$ is the unit object in $\Quiv_{\Ccal}(\Ecal)^\rev$. In particular, we obtain
\begin{equation}\tag{$\clubsuit$}
    \LMod_{Y\op}(\Fun(\Ccal\op,\Ecal))\simeq \Fun(\Ccal\op,\Ecal)
\end{equation}
under the above identifications \cite[Prop. 4.2.4.9]{HA}.
By design of $\Quiv_{\Ccal\op}^\BM(\Ecal)$, we get an induced right $\Ecal^\rev$ action on $\Fun(\Ccal\op,\Ecal)$, which is the same as a left $\Ecal$-action. This action is induced by the pointwise monoidal structure.

The left $\Ecal$-module structure on $\Fun(\Ccal\op,\Ecal)$ induces an enrichment as the $\otimes$-product of $\Ecal$ commutes with arbitrary colimits, i.e. we obtain an internal mapping object\linebreak $\Homline_{\Ecal}(F,G)\in \Ecal$ for all $F,G\in \Fun(\Ccal\op,\Ecal)$. By \cite[\S 6.2.4]{Hinich}, we obtain a map $Y'\colon \Ccal\rightarrow \LMod_{Y\op}(\Fun(\Ccal\op,\Ecal))$ with a compatible system of maps 
$$
    \Ytilde(x,y) \otimes Y'(x)\rightarrow Y'(y)
$$
for all $x,y\in \Ccal$. Under the equivalence ($\clubsuit$), we have $Y'\simeq \Ytilde$ \cite[\S 4.7.4]{Hinich} (see also the proof of \cite[6.2.6 Cor.]{Hinich}). This allows us to formulate an enriched Yoneda lemma.

\begin{lem}[Enriched Yoneda \protect{\cite[\S 6.2.7]{Hinich}}]
\label{lem-yoneda}
    Let $F\in \Fun(\Ccal\op,\Ecal)$ be an $\Ecal$-presheaf on $\Ccal$. Then, we have an equivalence 
    $$
        F(x)\simeq \Homline_{\Ecal}(\Ytilde(x),F)
    $$
    for all $x\in \Ccal$. Moreover, the induced functor $\Ytilde\colon \Ccal\rightarrow \Fun(\Ccal\op,\Ecal)$ is $\Ecal$-fully faithful.
\end{lem}

\end{appendix}




\bibliographystyle{alphaurl}
\bibliography{literature}

@misc{stacks-project,
	author = {The {Stacks Project Authors}},
	howpublished = {\url{https://stacks.math.columbia.edu}},
	shorthand = {Stacks},
	title = {\textit{Stacks Project}}}

@book{HTT,
	author = {Lurie, Jacob},
	doi = {10.1515/9781400830558},
	isbn = {978-0-691-14049-0; 0-691-14049-9},
	mrclass = {18-02 (18B25 18E35 18G30 18G55 55U40)},
	mrnumber = {2522659},
	mrreviewer = {Mark Hovey},
	pages = {xviii+925},
	publisher = {Princeton University Press, Princeton, NJ},
	series = {Annals of Mathematics Studies},
	title = {Higher topos theory},
	volume = {170},
	year = {2009}}

@misc{HA,
	author = {Lurie, Jacob},
	howpublished = {\url{https://www.math.ias.edu/~lurie/papers/HA.pdf}},
	title = {Higher algebra},
	year = {2017}}

@misc{SAG,
	author = {Lurie, Jacob},
	howpublished = {\url{https://www.math.ias.edu/~lurie/papers/SAG-rootfile.pdf}},
	title = {Spectral algebraic geometry},
	year = {2018}}

@article{MV1,
     author = {Morel, Fabien and Voevodsky, Vladimir},
     title = {{${\bf A}^1$}-homotopy theory of schemes},
     journal = {Publications Math\'ematiques de l'IH\'ES},
     pages = {45--143},
     year = {1999},
     publisher = {Institut des Hautes \'Etudes Scientifiques},
     volume = {90},
     doi = {10.1007/BF02698831},
     zbl = {0983.14007},
     mrnumber = {1813224},
}

@article {AGV,
    AUTHOR = {Ayoub, Joseph and Gallauer, Martin and Vezzani, Alberto},
     TITLE = {The six-functor formalism for rigid analytic motives},
   JOURNAL = {Forum Math. Sigma},
  FJOURNAL = {Forum of Mathematics. Sigma},
    VOLUME = {10},
      YEAR = {2022},
     PAGES = {Paper No. e61, 182},
       DOI = {10.1017/fms.2022.55},
       URL = {https://doi.org/10.1017/fms.2022.55},
}

@article {dahlhausen-k-theory-semi-valuation,
    AUTHOR = {Dahlhausen, Christian},
     TITLE = {Regularity of semi-valuation rings and homotopy invariance of
              algebraic {K}-theory},
   JOURNAL = {C. R. Math. Acad. Sci. Paris},
  FJOURNAL = {Comptes Rendus Math\'ematique. Acad\'emie des Sciences. Paris},
    VOLUME = {363},
      YEAR = {2025},
     PAGES = {989--1001},
      ISSN = {1631-073X,1778-3569},
       DOI = {10.5802/crmath.786},
       URL = {https://doi.org/10.5802/crmath.786},
}

@misc{DY,
      title={Towards $\mathbb{A}^1$-homotopy theory of rigid analytic spaces}, 
      author={Christian Dahlhausen and Can Yaylali},
      year={2024},
      eprint={2407.09606v3},
      archivePrefix={arXiv},
}

@book{fuji-kato,
	AUTHOR = {Fujiwara, Kazuhiro and Kato, Fumiharo},
     TITLE = {Foundations of {R}igid {G}eometry {I}},
    SERIES = {Monographs in Mathematics},
    VOLUME = {7},
 PUBLISHER = {EMS},
      YEAR = {2018},
     PAGES = {xxxiv+829},
      ISBN = {2523-5192},
}

@misc{DrewMHM,
      title={{Motivic Hodge modules}}, 
      author={Brad Drew},
      year={2018},
      eprint={1801.10129v1},
      archivePrefix={arXiv},
      primaryClass={math.AG}
}

@misc{Morrow,
      title={A historical overview of pro cdh descent in algebraic $K$-theory and its relation to rigid analytic varieties}, 
      author={Matthew Morrow},
      year={2016},
      eprint={1612.00418v1},
      archivePrefix={arXiv},
      primaryClass={math.KT}
}

@book {weinstein-scholze,
    AUTHOR = {Scholze, Peter and Weinstein, Jared},
     TITLE = {Berkeley lectures on {$p$}-adic geometry},
    SERIES = {Annals of Mathematics Studies},
    VOLUME = {207},
 PUBLISHER = {Princeton University Press, Princeton, NJ},
      YEAR = {2020},
     PAGES = {x+250},
      ISBN = {978-0-691-20209-9; 978-0-691-20208-2; 978-0-691-20215-0},
}

@article {kelly-saito-tamme,
    AUTHOR = {Kelly, Shane and Saito, Shuji and Tamme, Georg},
     TITLE = {On pro-cdh descent on derived schemes},
   JOURNAL = {Geom. Topol.},
  FJOURNAL = {Geometry \& Topology},
    VOLUME = {30},
      YEAR = {2026},
    NUMBER = {1},
     PAGES = {337--372},
      ISSN = {1465-3060,1364-0380},
       DOI = {10.2140/gt.2026.30.337},
       URL = {https://doi.org/10.2140/gt.2026.30.337},
}

@inproceedings {kerz-icm2018,
    AUTHOR = {Kerz, Moritz},
     TITLE = {On negative algebraic {$K$}-groups},
 BOOKTITLE = {Proceedings of the {I}nternational {C}ongress of
              {M}athematicians---{R}io de {J}aneiro 2018. {V}ol. {II}.
              {I}nvited lectures},
     PAGES = {163--172},
 PUBLISHER = {World Sci. Publ., Hackensack, NJ},
      YEAR = {2018},
      ISBN = {978-981-3272-91-0; 978-981-3272-87-3},
}

@article {kst-bass-quillen,
    AUTHOR = {Kerz, Moritz and Saito, Shuji and Tamme, Georg},
     TITLE = {Towards a non-archimedean analytic analog of the {B}ass-{Q}uillen conjecture},
   JOURNAL = {Journal of the Institute of Mathematics of Jussieu},
 PUBLISHER = {Cambridge University Press},
      YEAR = {2019},
     PAGES = {1-16},
       DOI = {10.1017/S147474801900001X},
       URL = {https://doi.org/10.1017/S147474801900001X}
}

@article {kst-i,
    AUTHOR = {Kerz, Moritz and Saito, Shuji and Tamme, Georg},
     TITLE = {K-theory of non-archimedean rings. {I}},
   JOURNAL = {Doc. Math.},
  FJOURNAL = {Documenta Mathematica},
    VOLUME = {24},
      YEAR = {2019},
     PAGES = {1365–1411},
       DOI = {10.25537/dm.2019v24.1365-1411},
}

@article {kst-ii,
    AUTHOR = {Kerz, Moritz and Saito, Shuji and Tamme, Georg},
     TITLE = {{$K$}-theory of non-{A}rchimedean rings {II}},
   JOURNAL = {Nagoya Math. J.},
  FJOURNAL = {Nagoya Mathematical Journal},
    VOLUME = {251},
      YEAR = {2023},
     PAGES = {669--685},
      ISSN = {0027-7630,2152-6842},
       DOI = {10.1017/nmj.2023.4},
       URL = {https://doi.org/10.1017/nmj.2023.4},
}

@article {kerz-strunk-tamme,
    AUTHOR = {Kerz, Moritz and Strunk, Florian and Tamme, Georg},
     TITLE = {Algebraic {$K$}-theory and descent for blow-ups},
   JOURNAL = {Invent. Math.},
  FJOURNAL = {Inventiones Mathematicae},
    VOLUME = {211},
      YEAR = {2018},
    NUMBER = {2},
     PAGES = {523--577},
      ISSN = {0020-9910,1432-1297},
       DOI = {10.1007/s00222-017-0752-2},
       URL = {https://doi.org/10.1007/s00222-017-0752-2},
}

@misc{condensed,
      title={Lectures on {C}ondensed {M}athematics}, 
      author={Peter Scholze},
      year={2026},
      eprint={2605.03658v1},
      archivePrefix={arXiv},
      primaryClass={math.NT},
}

@misc{analytic,
      title={Lectures on {A}nalytic {G}eometry}, 
      author={Peter Scholze},
      year={2026},
      eprint={2605.03655v1},
      archivePrefix={arXiv},
      primaryClass={math.AG},
}

@article {isaksen-calculating,
    AUTHOR = {Isaksen, Daniel C.},
     TITLE = {Calculating limits and colimits in pro-categories},
   JOURNAL = {Fund. Math.},
  FJOURNAL = {Fundamenta Mathematicae},
    VOLUME = {175},
      YEAR = {2002},
    NUMBER = {2},
     PAGES = {175--194},
      ISSN = {0016-2736},
       DOI = {10.4064/fm175-2-7},
       URL = {https://doi.org/10.4064/fm175-2-7},
}

@article {emmanouil,
    AUTHOR = {Emmanouil, Ioannis},
     TITLE = {Mittag-{L}effler condition and the vanishing of
              {$\varprojlim^1$}},
   JOURNAL = {Topology},
  FJOURNAL = {Topology. An International Journal of Mathematics},
    VOLUME = {35},
      YEAR = {1996},
    NUMBER = {1},
     PAGES = {267--271},
      ISSN = {0040-9383},
       DOI = {10.1016/0040-9383(94)00056-5},
       URL = {https://doi.org/10.1016/0040-9383(94)00056-5},
}

@book {bosch,
    AUTHOR = {Bosch, Siegfried},
     TITLE = {Lectures on formal and rigid geometry},
    SERIES = {Lecture Notes in Mathematics},
    VOLUME = {2105},
 PUBLISHER = {Springer, Cham},
      YEAR = {2014},
     PAGES = {viii+254},
      ISBN = {978-3-319-04416-3; 978-3-319-04417-0},
}

@phdthesis{andreychev-thesis,
author = {Andreychev, Grigory},
title = {K-Theorie adischer Räume},
school = {Rheinische Friedrich-Wilhelms-Universität Bonn},
year = 2023,
month = sep,
howpublished = {\url{https://nbn-resolving.org/urn:nbn:de:hbz:5-72174}},
url = {https://hdl.handle.net/20.500.11811/11040}
}

@incollection {PPR-KGL,
    AUTHOR = {Panin, Ivan and Pimenov, Konstantin and R\"ondigs, Oliver},
     TITLE = {On {V}oevodsky's algebraic {$K$}-theory spectrum},
 BOOKTITLE = {Algebraic topology},
    SERIES = {Abel Symp.},
    VOLUME = {4},
     PAGES = {279--330},
 PUBLISHER = {Springer, Berlin},
      YEAR = {2009},
      ISBN = {978-3-642-01199-3},
       DOI = {10.1007/978-3-642-01200-6\_10},
}

@inproceedings {voevodsky-icm-1998,
    AUTHOR = {Voevodsky, Vladimir},
     TITLE = {{${\bf A}^1$}-homotopy theory},
 BOOKTITLE = {Proceedings of the {I}nternational {C}ongress of
              {M}athematicians, {V}ol. {I} ({B}erlin, 1998)},
   JOURNAL = {Doc. Math.},
  FJOURNAL = {Documenta Mathematica},
      YEAR = {1998},
     PAGES = {579--604},
      ISSN = {1431-0635,1431-0643},
}

@book {egr,
    AUTHOR = {Abbes, Ahmed},
     TITLE = {{\'E}l\'ements de g\'eom\'etrie rigide. {V}olume {I}},
    SERIES = {Progress in Mathematics},
    VOLUME = {286},
      NOTE = {Construction et \'etude g\'eom\'etrique des espaces rigides.},
 PUBLISHER = {Birkh\"auser/Springer Basel AG, Basel},
      YEAR = {2010},
     PAGES = {xvi+477},
      ISBN = {978-3-0348-0011-2},
}

@article {raynaud-gruson,
    AUTHOR = {Raynaud, Michel and Gruson, Laurent},
     TITLE = {Crit\`eres de platitude et de projectivit\'e. {T}echniques de
              ``platification'' d'un module},
   JOURNAL = {Invent. Math.},
  FJOURNAL = {Inventiones Mathematicae},
    VOLUME = {13},
      YEAR = {1971},
     PAGES = {1--89},
      ISSN = {0020-9910,1432-1297},
       DOI = {10.1007/BF01390094},
       URL = {https://doi.org/10.1007/BF01390094},
}

@incollection {raynaud,
    AUTHOR = {Raynaud, Michel},
     TITLE = {G\'eom\'etrie analytique rigide d'apr\`es {T}ate,
              {K}iehl, {$\cdots $}},
 BOOKTITLE = {Table {R}onde d'{A}nalyse {N}on {A}rchim\'edienne ({P}aris,
              1972)},
    SERIES = {Suppl\'ement au Bull. Soc. Math. France},
    VOLUME = {Tome 102},
     PAGES = {319--327},
 PUBLISHER = {Soc. Math. France, Paris},
      YEAR = {1974},
   MRCLASS = {32K10 (14G20)},
  MRNUMBER = {470254},
MRREVIEWER = {J.\ S.\ Joel},
       DOI = {10.24033/msmf.170},
       URL = {https://doi.org/10.24033/msmf.170},
}

@article {cisinski-descente,
    AUTHOR = {Cisinski, Denis-Charles},
     TITLE = {Descente par \'eclatements en {$K$}-th\'eorie invariante par
              homotopie},
   JOURNAL = {Ann. of Math. (2)},
  FJOURNAL = {Annals of Mathematics. Second Series},
    VOLUME = {177},
      YEAR = {2013},
    NUMBER = {2},
     PAGES = {425--448},
      ISSN = {0003-486X,1939-8980},
       DOI = {10.4007/annals.2013.177.2.2},
       URL = {https://doi.org/10.4007/annals.2013.177.2.2},
}

@incollection {huebner-adic,
    AUTHOR = {H{\"u}bner, Katharina},
     TITLE = {Adic spaces},
 BOOKTITLE = {Non-archimedean geometry and eigenvarieties},
    SERIES = {M\"unst. Lect. Math.},
     PAGES = {55--111},
 PUBLISHER = {EMS Press, Berlin},
      YEAR = {2024},
      ISBN = {978-3-98547-081-5; 978-3-98547-581-0},
       DOI = {10.4171/mlm/4/2},
}

@book {GW1,
    AUTHOR = {G\"ortz, Ulrich and Wedhorn, Torsten},
     TITLE = {Algebraic {G}eometry {I}},
    SERIES = {Advanced Lectures in Mathematics},
      NOTE = {Schemes with examples and exercises},
 PUBLISHER = {Vieweg + Teubner, Wiesbaden},
      YEAR = {2010},
     PAGES = {viii+615},
      ISBN = {978-3-8348-0676-5},
       DOI = {10.1007/978-3-8348-9722-0},
       
}

@article {weibel-analytic-iso,
    AUTHOR = {Weibel, Charles A.},
     TITLE = {{$K$}-theory and analytic isomorphisms},
   JOURNAL = {Invent. Math.},
  FJOURNAL = {Inventiones Mathematicae},
    VOLUME = {61},
      YEAR = {1980},
    NUMBER = {2},
     PAGES = {177--197},
      ISSN = {0020-9910,1432-1297},
       DOI = {10.1007/BF01390120},
       URL = {https://doi.org/10.1007/BF01390120},
}

@phdthesis{sigloch-thesis,
author = {Sigloch, Helene},
title = {Homotopy Theory for Rigid Analytic Varieties},
school = {Albert-Ludwigs-Universit\"at Freiburg im Breisgau},
year = 2016,
month = mar,
howpublished = {\url{https://freidok.uni-freiburg.de/data/11742}},
url = {https://freidok.uni-freiburg.de/data/11742}
}

@incollection {tt90,
    AUTHOR = {Thomason, Robert W. and Trobaugh, Thomas},
     TITLE = {Higher algebraic {$K$}-theory of schemes and of derived
              categories},
 BOOKTITLE = {The {G}rothendieck {F}estschrift, {V}ol.\ {III}},
    SERIES = {Progr. Math.},
    VOLUME = {88},
     PAGES = {247--435},
 PUBLISHER = {Birkh\"auser Boston, Boston, MA},
      YEAR = {1990},
       DOI = {10.1007/978-0-8176-4576-2_10},
}

@book {k-book,
    AUTHOR = {Weibel, Charles A.},
     TITLE = {The {$K$}-book},
    SERIES = {Graduate Studies in Mathematics},
    VOLUME = {145},
      NOTE = {An Introduction to Algebraic $K$-theory},
 PUBLISHER = {American Mathematical Society, Providence, RI},
      YEAR = {2013},
     PAGES = {xii+618},
      ISBN = {978-0-8218-9132-2},
}

@article {dahli-thesis-paper,
    AUTHOR = {Dahlhausen, Christian},
     TITLE = {Continuous {K}-theory and cohomology of rigid spaces},
   JOURNAL = {Manuscripta Math.},
  FJOURNAL = {Manuscripta Mathematica},
    VOLUME = {173},
      YEAR = {2024},
    NUMBER = {1-2},
     PAGES = {119--153},
      ISSN = {0025-2611,1432-1785},
       DOI = {10.1007/s00229-023-01470-x},
       URL = {https://doi.org/10.1007/s00229-023-01470-x},
}

@misc{Mann-Thesis,
      title={A $p$-Adic 6-Functor Formalism in Rigid-Analytic Geometry}, 
      author={Lucas Mann},
      year={2022},
      eprint={2206.02022v1},
      archivePrefix={arXiv},
      primaryClass={math.AG}
}

@article{Hinich,
title = {Yoneda lemma for enriched $\infty$-categories},
journal = {Advances in Mathematics},
volume = {367},
pages = {107129},
year = {2020},
issn = {0001-8708},
doi = {https://doi.org/10.1016/j.aim.2020.107129},
author = {Vladimir Hinich}
}

@book{MotCoh,
author = {Dundas, Bj{\o}rn and Levine, Marc and {\O}stv{\ae}r, Paul Arne and R\"ondigs, Oliver and Voevodsky, Vladimir},
year = {2007},
month = {01},
pages = {},
publisher = {Springer Berlin, Heidelberg},
title = {Motivic Homotopy Theory: Lectures at a Summer School in Nordfjordeid, Norway, August 2002},
isbn = {978-3-540-45895-1},
doi = {10.1007/978-3-540-45897-5}
}

@article {temkin-desing-char0,
    AUTHOR = {Temkin, Michael},
     TITLE = {Desingularization of quasi-excellent schemes in characteristic
              zero},
   JOURNAL = {Adv. Math.},
  FJOURNAL = {Advances in Mathematics},
    VOLUME = {219},
      YEAR = {2008},
    NUMBER = {2},
     PAGES = {488--522},
      ISSN = {0001-8708,1090-2082},
       DOI = {10.1016/j.aim.2008.05.006},
       URL = {https://doi.org/10.1016/j.aim.2008.05.006},
}

@misc{efimov-inverse-limits,
      title={Localizing invariants of inverse limits}, 
      author={Alexander I. Efimov},
      year={2025},
      eprint={2502.04123v2},
      archivePrefix={arXiv},
      primaryClass={math.KT}
}

@misc{efimov-k-theory,
      title={K-theory and localizing invariants of large categories}, 
      author={Alexander I. Efimov},
      year={2025},
      eprint={2405.12169v3},
      archivePrefix={arXiv},
      primaryClass={math.KT},
}

@phdthesis{zhou-these,
  TITLE = {{Sur les cohomologies {\'e}tale et syntomique pour les vari{\'e}t{\'e}s analytiques p-adiques arithm{\'e}tiques}},
  AUTHOR = {Zhou, Yicheng},
  URL = {https://theses.hal.science/tel-05238283},
  NUMBER = {2025SORUS136},
  SCHOOL = {{Sorbonne Universit{\'e}}},
  YEAR = {2025},
  MONTH = Jul,
  TYPE = {Thesis},
  HAL_ID = {tel-05238283},
  HAL_VERSION = {v1},
}

@unpublished{zhou-dagger-condition,
  TITLE = {{A note on the dagger condition of Kerz--Saito--Tamme}},
  AUTHOR = {Zhou, Yicheng},
  URL = {https://hal.science/hal-05820422},
  NOTE = {working paper or preprint},
  YEAR = {2026},
  MONTH = Oct,
  HAL_ID = {hal-05820422},
  HAL_VERSION = {v1},
}

@article{temkin2017altered,
  title={Altered local uniformization of {B}erkovich spaces},
  author={Temkin, Michael},
  journal={Israel Journal of Mathematics},
  volume={221},
  number={2},
  pages={585--603},
  year={2017},
  publisher={Springer}
}

@Book{matsumura1980commutativealgebra,
  Author = {Matsumura, Hideyuki},
  Title = {Commutative algebra. 2nd ed},
  FSeries = {Mathematics Lecture Note Series},
  Series = {Math. Lect. Note Ser.},
  Volume = {56},
  Year = {1980},
  Publisher = {The Benjamin/Cummings Publishing Company, Reading, MA},
  Language = {English},
  zbMATH = {3687501},
  Zbl = {0441.13001}
}

@article{kedlaya2017sheaves,
  title={{Sheaves, stacks, and shtukas}},
  author={Kedlaya, Kiran S.},
  journal={Perfectoid Spaces: Lectures from the 2017 Arizona Winter School},
  year={2017}
}

@article {huber1994generalization,
    AUTHOR = {Huber, R.},
     TITLE = {A generalization of formal schemes and rigid analytic varieties},
   JOURNAL = {Math. Z.},
  FJOURNAL = {Mathematische Zeitschrift},
    VOLUME = {217},
      YEAR = {1994},
    NUMBER = {4},
     PAGES = {513--551},
      ISSN = {0025-5874,1432-1823},
       DOI = {10.1007/BF02571959},
}

@Book{huber1996etale,
 Author = {Roland {Huber}},
 Title = {{\'Etale cohomology of rigid analytic varieties and adic spaces}},
 FJournal = {{Aspects of Mathematics}},
 Journal = {{Aspects Math.}},
 ISSN = {0179-2156},
 Volume = {E30},
 ISBN = {3-528-06794-2/hbk},
 Pages = {x + 450},
 Year = {1996},
 Publisher = {Wiesbaden: Vieweg},
 Language = {English},
 MSC2010 = {14F20 32P05},
 Zbl = {0868.14010}
}

@article{landtamme2018pullback,
   title={On the $K$-theory of pullbacks},
   volume={190},
   ISSN={0003-486X},
   DOI={10.4007/annals.2019.190.3.4},
   number={3},
   journal={Annals of Mathematics},
   publisher={Annals of Mathematics},
   author={Land, Markus and Tamme, Georg},
   year={2019},
   month=Nov
}

@article{yekutieli2021proregular,
 author = {Yekutieli, Amnon},
 title = {Weak proregularity, derived completion, adic flatness, and prisms},
 fjournal = {Journal of Algebra},
 journal = {J. Algebra},
 issn = {0021-8693},
 volume = {583},
 pages = {126--152},
 year = {2021},
 language = {English},
 doi = {10.1016/j.jalgebra.2021.04.033},
 zbMATH = {7363519},
 Zbl = {1476.14054}
}

@book {bousfieldkan1972lecture,
    AUTHOR = {Bousfield, Aldridge K. and Kan, Daniel M.},
     TITLE = {Homotopy limits, completions and localizations},
    SERIES = {Lecture Notes in Mathematics},
    VOLUME = {Vol. 304},
 PUBLISHER = {Springer-Verlag, Berlin-New York},
      YEAR = {1972},
     PAGES = {v+348},
}

@misc{andreychev2021descent,
      title={{Pseudocoherent and Perfect Complexes and Vector Bundles on Analytic Adic Spaces}}, 
      author={Grigory Andreychev},
      year={2021},
      eprint={2105.12591v1},
      archivePrefix={arXiv},
      primaryClass={math.AG}
}

@article{mondalreinecke2025postnikov,
   title={On Postnikov completeness for replete topoi},
   volume={27},
   ISSN={1532-0081},
   DOI={10.4310/hha.2025.v27.n1.a10},
   number={1},
   journal={Homology, Homotopy and Applications},
   publisher={International Press of Boston},
   author={Mondal, Shubhodip and Reinecke, Emanuel},
   year={2025},
   pages={179–196} }

@misc{camargo2026notesolid,
      title={Notes on Solid Geometry}, 
      author={Juan Esteban Rodríguez Camargo},
      year={2026},
      eprint={2603.03012v1},
      archivePrefix={arXiv},
      primaryClass={math.AG},
}

@article {waldhausen-free-1,
    AUTHOR = {Waldhausen, Friedhelm},
     TITLE = {Algebraic {$K$}-theory of generalized free products. {I},
              {II}},
   JOURNAL = {Ann. of Math. (2)},
  FJOURNAL = {Annals of Mathematics. Second Series},
    VOLUME = {108},
      YEAR = {1978},
    NUMBER = {1},
     PAGES = {135--204},
      ISSN = {0003-486X},
       DOI = {10.2307/1971165},
       URL = {https://doi.org/10.2307/1971165},
}

@article{huber1993continuousvaluation,
  title={Continuous valuations},
  author={Huber, Roland},
  journal={Mathematische Zeitschrift},
  volume={212},
  number={1},
  pages={455--477},
  year={1993},
  publisher={Springer},
  doi={10.1007/BF02571668}
}

@article{buzzardverberkmoes2018stablyuniformsheafy,
  title={Stably uniform affinoids are sheafy},
  author={Buzzard, Kevin and Verberkmoes, Alain},
  journal={Journal f{\"u}r die reine und angewandte Mathematik},
  volume={2018},
  number={740},
  pages={25--39},
  year={2018},
  publisher={De Gruyter},
  doi = {10.1515/crelle-2015-0089},
}

@article {naumann-sptzweg-oestvaer,
    AUTHOR = {Naumann, Niko and Spitzweck, Markus and {\O}stv{\ae}r, Paul Arne},
     TITLE = {Existence and uniqueness of {$E_\infty$} structures on motivic
              {$K$}-theory spectra},
   JOURNAL = {J. Homotopy Relat. Struct.},
  FJOURNAL = {Journal of Homotopy and Related Structures},
    VOLUME = {10},
      YEAR = {2015},
    NUMBER = {3},
     PAGES = {333--346},
      ISSN = {2193-8407,1512-2891},
       DOI = {10.1007/s40062-013-0062-3},
       URL = {https://doi.org/10.1007/s40062-013-0062-3},
}

@article{zavyalov2024quotient,
   title={Quotients of admissible formal schemes and adic spaces by finite groups},
   volume={18},
   ISSN={1937-0652},
   DOI={10.2140/ant.2024.18.409},
   number={3},
   journal={Algebra \& Number Theory},
   publisher={Mathematical Sciences Publishers},
   author={Zavyalov, Bogdan},
   year={2024},
   month=Feb, pages={409--475}
}

@article{excellence-adic-completion,
   title={Ideal-adic completion of quasi-excellent rings (after {G}abber)},
   author={Kurano, Kazuhiko and Shimomoto, Kazuma},
   volume={61},
   ISSN={2154-3321},
   DOI={10.1215/21562261-2021-0011},
   number={3},
   journal={Kyoto Journal of Mathematics},
   publisher={Duke University Press},
   year={2021},
   month=Sep, 
   pages={707–722}
}

@article{haine-2026,
   title={Profinite completions of products},
   volume={26},
   ISSN={1472-2747},
   DOI={10.2140/agt.2026.26.397},
   number={1},
   journal={Algebraic \& Geometric Topology},
   publisher={Mathematical Sciences Publishers},
   author={Haine, Peter J.},
   year={2026},
   month=Jan, pages={397–410}
}

@misc{bachmann-engelmann-mattis-2025-monadic,
      title={Monadic resolutions for generalized spaces}, 
      author={Tom Bachmann and Anton Engelmann and Klaus Mattis},
      year={2025},
      eprint={2511.07064v1},
      archivePrefix={arXiv},
      primaryClass={math.AT},
}

@misc{balderrama-pstragowski-2026-maclev,
      title={Unstable synthetic deformations I: Malcev theories}, 
      author={William Balderrama and Piotr Pstrągowski},
      year={2026},
      eprint={2601.08802v1},
      archivePrefix={arXiv},
      primaryClass={math.AT},
}

@incollection{BousfieldFriedlander1978,
  author = {Bousfield, Aldridge K. and Friedlander, Eric M.},
  title = {Homotopy theory of $\Gamma$-spaces, spectra, and bisimplicial sets},
  booktitle = {Geometric applications of homotopy theory (Proc. Conf., Evanston, Ill., 1977), {II}},
  series = {Lecture Notes in Math.},
  volume = {658},
  publisher = {Springer},
  address = {Berlin},
  year = {1978},
  pages = {80--130},
  doi = {10.1007/BFb0068574}
}

@unpublished{Rezk2014BousfieldFriedlander,
  author = {Charles Rezk},
  title = {When are homotopy colimits compatible with homotopy base change?},
  year = {2014},
  url = {https://rezk.web.illinois.edu/i-hate-the-pi-star-kan-condition.pdf},
  note = {Unpublished note}
}

@misc{clausen2025duality,
      title={Duality and linearization for p-adic lie groups}, 
      author={Dustin Clausen},
      year={2025},
      eprint={2506.18174v1},
      archivePrefix={arXiv},
      primaryClass={math.AT},
}

\bigskip
  \footnotesize

  Christian Dahlhausen:\\  \textsc{Institut für Mathematik, Universität Heidelberg, Im Neuenheimer Feld 205, 69120 Heidelberg, Germany}\par\nopagebreak
  \textit{E-mail address:}  \texttt{cdahlhausen@mathi.uni-heidelberg.de}

  \medskip

  Can Yaylali:\\ \textsc{Fakultät für Mathematik, Universität Duisburg-Essen, Thea-Leymann-Str. 9, 45127 Essen, Germany}\par\nopagebreak
  \textit{E-mail address:} \texttt{can.yaylali@uni-due.de}

  \medskip

  Yicheng Zhou:\\ \textsc{Université Paris-Saclay, CNRS, Laboratoire de Mathématiques d'Orsay, 91405 Orsay, France}\par\nopagebreak
  \textit{E-mail address:} \texttt{yicheng.zhou@universite-paris-saclay.fr}


\end{document}